\documentclass{amsart}
\usepackage[utf8]{inputenc}
\usepackage[margin=1.25in]{geometry}
\usepackage{amsthm,amsmath,amssymb, amscd}
\usepackage{soul}
\usepackage[colorlinks=true,allcolors=blue!80!black]{hyperref}
\usepackage[capitalize]{cleveref}
\usepackage[parfill]{parskip} 
\usepackage{graphics}
\usepackage{tikz, tikz-cd}
\usetikzlibrary{matrix, patterns, shadows.blur, backgrounds,
  cd,
  calc,
  positioning,
  fit,
  arrows,
  decorations.pathreplacing,
  decorations.pathmorphing,
  decorations.markings,
  shapes.geometric,
  backgrounds,
  bending
}
\usepackage[all]{xy}
\usepackage{enumerate}
\usepackage{bbm}
\usepackage{stmaryrd}
\usepackage{nicematrix}
\usepackage{ dsfont }
\usepackage{caption}
\usepackage{subcaption}
\usepackage{import}

\theoremstyle{plain}
\newtheorem{thm}{Theorem}[section]
\newtheorem{prop}[thm]{Proposition}
\newtheorem{lem}[thm]{Lemma}
\newtheorem{cor}[thm]{Corollary}
\newtheorem{conj}[thm]{Conjecture}
\theoremstyle{definition}
\newtheorem{defn}[thm]{Definition}
\newtheorem{rem}[thm]{Remark}

\newtheorem{ex}[thm]{Example}

\DeclareMathOperator\Diff{Diff}
\DeclareMathOperator\Isom{Isom}
\DeclareMathOperator\Aut{Aut}
\DeclareMathOperator\Aff{Aff}

\DeclareMathOperator\emb{Emb}
\DeclareMathOperator\Map{Map}

\DeclareMathOperator\im{im}
\DeclareMathOperator\Conf{Conf}
\DeclareMathOperator\SO{SO}

\DeclareMathOperator\Or{O}

\DeclareMathOperator\SL{SL}
\DeclareMathOperator\GL{GL}
\DeclareMathOperator\PSL{PSL}

\DeclareMathOperator\Nil{Nil}
\DeclareMathOperator\Sol{Sol}

\DeclareMathOperator\Fr{Fr}
\DeclareMathOperator{\umb}{Sub} 
\newcommand{\interior}[1]{\smash{\mathring{#1}}}
\DeclareMathOperator{\Out}{Out}
\DeclareMathOperator{\SF}{SF}
\DeclareMathOperator{\PI}{PI}

\DeclareMathOperator{\JSJ}{JSJ}
\DeclareMathOperator{\tr}{tr}
\DeclareMathOperator{\id}{id}
\DeclareMathOperator{\Mod}{Mod}
\DeclareMathOperator{\Sym}{Sym}
\DeclareMathOperator{\Inn}{Inn}
\DeclareMathOperator{\PL}{PL}

\newcommand{\scl}[1]{\smash{\widehat{#1}}} 
\DeclareMathOperator\cut{Cut} 
\newcommand{\cutNE}{\cut^{\neq \emptyset}} 
\DeclareMathOperator\sep{Sep} 
\newcommand{\qsep}[1]{\cut^{\mathrm{pf}\leq{#1}}} 
\newcommand{\sepNP}{\sep^{\nparallel}} 
\newcommand{\sepNPsemi}{\sep^{\nparallel,\subset}} 
\newcommand{\sepL}{\sep^{\mathrm{L}}} 
\newcommand{\sepLtwo}{\sep^{\mathrm{L},2}} 
\newcommand{\ca}{\smash{\mid}} 

\newcommand{\DiffB}{\Diff_\partial}

\newcommand{\hq}{/\!\!/} 

\renewcommand{\ul}[1]{\underline{#1}}

\newcommand{\BDiff}{B\!\Diff}
\newcommand{\calM}{\mathcal{M}}

\newcommand{\bbZ}{\mathbb{Z}}
\newcommand{\bbQ}{\mathbb{Q}}
\newcommand{\bbR}{\mathbb{R}}
\newcommand{\bbS}{\mathbb{S}}
\newcommand{\bbH}{\mathbb{H}}
\newcommand{\bbE}{\mathbb{E}}
\newcommand{\bbT}{\mathbb{T}}
\newcommand{\RP}{\mathbb{R}\mathrm{P}}

\newcommand{\Z}{\mathbb{Z}}
\newcommand{\bt}{\bullet}
\newcommand{\G}{\Gamma}
\newcommand{\Kx}{K\widetilde{\times}I}

\usepackage{fp}
\newsavebox\foobox
\newcommand\slbox[2]{%
  \FPdiv{\result}{#1}{57.296}
  \FPtan{\result}{\result}%
  \slantbox[\result]{#2}%
}%
\newcommand{\slantbox}[2][30]{%
        \mbox{%
        \sbox{\foobox}{#2}%
        \hskip\wd\foobox
        \pdfsave
        \pdfsetmatrix{1 0 #1 1}%
        \llap{\usebox{\foobox}}%
        \pdfrestore
}}
\newcommand\rotslant[3]{\rotatebox{#1}{\slbox{#2}{#3}}}
\usepackage{mwe}

\newcommand{\fakeenv}{} 
\newenvironment{restate}[2]  
{
  \renewcommand{\fakeenv}{#2}   
  \theoremstyle{plain}
  \newtheorem*{\fakeenv}{#1~\ref{#2}} 
  \begin{\fakeenv}  
}
{ \end{\fakeenv} }

\newenvironment{rerestate}[3]  
{
  \renewcommand{\fakeenv}{#3}   
  \theoremstyle{plain}
  \newtheorem*{\fakeenv}{#1~\ref{#2}} 
  \begin{\fakeenv}  
}
{ \end{\fakeenv} }

\title{Moduli spaces of $3$-manifolds with boundary are finite}
\author{Rachael Boyd}
\address{School of Mathematics and Statistics, University of Glasgow, Glasgow G12 8QQ, UK}
\email{rachael.boyd@glasgow.ac.uk}
\urladdr{https://www.maths.gla.ac.uk/~rboyd/} 

\author{Corey Bregman}
\address{Department of Mathematics, Tufts University, Medford, MA 02155, USA}
\email{corey.bregman@tufts.edu}
\urladdr{https://sites.google.com/view/cbregman}

\author{Jan Steinebrunner}
\address{Institut for Matematiske Fag,
Københavns Universitet,
København, 
Denmark}
\email{js2675@cam.ac.uk}
\urladdr{https://www.jan-steinebrunner.com} 


\begin{document}

\newpage
\begin{abstract}
    We study the classifying space~$\BDiff(M)$ of the diffeomorphism group of a connected, compact, orientable 3-manifold $M$. In the case that~$M$ is reducible we build a contractible space parametrising the systems of reducing spheres. We use this to prove that if~$M$ has non-empty boundary, then~$\BDiff_\partial(M)$ has the homotopy type of a finite CW complex. This was conjectured by Kontsevich and appears on the Kirby problem list as Problem 3.48. As a consequence, we are able to show that for every compact, orientable $3$-manifold $M$, $\BDiff(M)$ has finite type.
\end{abstract}
\maketitle

\setcounter{tocdepth}{1} 
\tableofcontents

\section{Introduction}
For a $3$-manifold $M$, let $\Diff(M)$ denote group of diffeomorphisms, equipped with the $C^\infty$~topology.
Its classifying space $\BDiff(M)$ is the moduli space of $M$: homotopy classes of maps $X \to \BDiff(M)$ are in bijection with concordance classes of smooth $M$-bundles over $X$ \cite{GRW-users-guide}.
When $M$ has boundary, we let $\DiffB(M) < \Diff(M)$ denote the subgroup of those diffeomorphisms that fix the boundary pointwise and consider $\BDiff_\partial(M)$, which classifies smooth bundles with trivialised boundary.
We prove the following conjecture of Kontsevich \cite[Problem 3.48]{KirbyProblems}.

\begin{restate}{Theorem}{thm:Kontsevich-finiteness-conjecture}
    Let $M$ be a compact, connected, orientable $3$-manifold with non-empty boundary~$\partial M$.
    Then 
        $\BDiff_\partial(M)$
    is \emph{homotopy finite}, 
    i.e.~it has the homotopy type of a finite CW complex.
\end{restate}

When $M$ is irreducible, this was proven by Hatcher and McCullough \cite{HatcherMcCullough} (we restate this in \cref{thm: HatcherMcCullough finiteness}) using that in this setting~$\BDiff(M)$ is aspherical. 
In fact we prove a slightly stronger version of the conjecture, where we only fix a union of boundary components~$F$, such that $\partial M \setminus F$ consists of spheres and incompressible tori.
This theorem implies that the group cohomology of the topological group $\Diff_\partial(M)$ is finitely generated and that it has finite cohomological dimension.
(Under the simplifying assumption that each prime factor of $M$ has non-empty boundary, this (co)homological version of Kontsevich's conjecture was proven by Nariman \cite{Nariman}.)
Our results are effective in the sense that for any given $M$ one can read off a bound on the cohomological dimension by going through the steps of the proof, but we have not attempted to give a closed formula as the many case distinctions make this quite cumbersome.

Recall that every topological 3-manifold admits an essentially unique PL and smooth structure, by work of Cairns \cite{Cairns1940} and Whitehead \cite{Whitehead}, Moise \cite{Moise1952,Moise54} and Bing \cite{Bing54} (see also \cite{Bing59}), and Munkres \cite{Munkres-Announcement,Munkres, Munkres-Obstructions}  and Whitehead \cite{Whitehead61}. Moreover, a combination of work by Cerf \cite{Cerf59} and Hatcher \cite{Hatcher} shows that $\Diff(M) \simeq \PL(M)\simeq \operatorname{Homeo}(M)$ so our statements hold for all of these groups. (The equivalence of $\operatorname{PL}(M)$ with the other two relies on a theorem attributed to Morlet but first proven by Burghelea--Lashof \cite{BurgheleaLashof}, see Kirby--Siebenmann \cite[Essay V]{KirbySiebenmann} for more details.)

The analogue of \cref{thm:Kontsevich-finiteness-conjecture} for surfaces is true as a consequence of results of Earle--Schatz \cite{EarleSchatz} and Gramain \cite{Gramain}, but fails in dimensions $n\geq 6$: 
$\pi_0\Diff_\partial(S^1 \times D^{n-1})$ is not finitely generated by work of Hatcher--Wagoner \cite[$n\geq 7$: Part II, Corollary 5.5, $n=6$: p230]{HatcherWagoner}. 
Budney--Gabai \cite{BudneyGabai2021} prove that the mapping class group of $S^1\times D^{3}$ relative to the boundary is also not finitely generated in both the smooth and topological categories.
Thus for $n=4$ or $n\geq 6$, $\BDiff_\partial(S^1\times D^{n-1})$ is not even of finite type. (Recall that a space is said to be of \emph{finite type} if it is homotopy equivalent to a CW complex with finite $n$-skeleton for each $n$.)

When a 3-manifold $M$ is closed or when the boundary is not fixed, $\BDiff(M)$ rarely has the homotopy type of a finite CW complex, but we can ask when it is of finite type. Hatcher--McCullough showed mapping class groups of 3-manifolds are finitely presented \cite{HatcherMcCullough1990} so, unlike in higher dimensions, counterexamples arising from $M$ having infinitely generated mapping class groups do not exist. Using \cref{thm:Kontsevich-finiteness-conjecture}, we deduce the following for all 3-manifolds.

\begin{restate}{Theorem}{thm:finite-type}
    Let $M$ be a compact, orientable $3$-manifold.
    Then $\BDiff(M)$ is of finite type.
\end{restate}

This completes the picture for all manifolds of dimension $\leq 3$. We briefly survey for which cases we know that $\BDiff(M)$ is of finite type. In dimension 1, since $\Diff(S^1)\simeq \Or(2)$, one knows that $\BDiff(S^1)\simeq B\!\Or(2)$ is of finite type. The corresponding result for closed orientable surfaces combines work of Smale \cite{Smale59}, who showed $\BDiff(S^2)\simeq B\!\Or(3)$, and Earle--Eells \cite{EarleElls69} for surfaces of genus $g\geq 1$.
In fact, when $g\geq 2$, they prove that $\Diff_0(M)$ is contractible, whence $\BDiff(M)\simeq B\pi_0\Diff(M)=K(\Mod_g^\pm,1)$, where $\Mod^\pm_g$ is the (full) mapping class group of the surface of genus~$g$. 
In this case, finiteness fails because of torsion in $\Mod^\pm_g$, but $B\!\Mod^\pm_g$ is of finite type because $\Mod^\pm_g$ has a finite index subgroup whose classifying space is homotopy finite, see Ivanov \cite[\S6.4]{Ivanov-Teichmuller89}.
In higher dimensions $n \neq 4,5, 7$, Kupers showed that $\BDiff(S^n)$ and $\BDiff_\partial(D^n)$ are of finite type \cite{Kupers19}, and in even dimensions $2n \ge 6$, Bustamante--Krannich--Kupers \cite{BustamanteKrannichKupers} showed that if~$\pi_1(M)$ is finite, then $\BDiff(M)$ is of finite type.

\subsection{Overview of the proof}
We will prove the following generalisation of Kontesevich's conjecture.
Here, for a union of boundary components $F \subseteq \partial M$, we let $\Diff_F(M) \le \Diff(M)$ denote the subgroup of those diffeomorphisms that fix $F$ pointwise.
\begin{restate}{Theorem}{thm:more general Kontsevich-finiteness-conjecture}
    Let $M$ be a compact, connected, orientable $3$-manifold, and let $\emptyset \neq F \subseteq \partial M$ be a non-empty union of connected components.
    Assume that $\partial M \setminus F$ consists of spheres and incompressible tori.
    Then $\BDiff_F(M)$ is homotopy finite.
\end{restate}

Let $\scl{M}$ denote the \emph{spherical closure} of~$M$: the 3-manifold obtained by filling in spherical boundary components.
Kneser \cite{Kneser1929} showed that every oriented 3-manifold~$M$ without 2-spheres in its boundary admits a connected sum decomposition into oriented irreducible factors~$P_i$ that are not homeomorphic to~$S^3$ and some number of copies of $S^1\times S^2$. Milnor~\cite{Milnor62} then proved this decomposition is uniquely determined up to reordering. Therefore every oriented 3-manifold $M$ can be written as
\[
M\cong \left(P_1\#\cdots \#P_n \# (S^1\times S^2)^{\# g}\right)\setminus \{\amalg_m \interior{D}^3\}
\]
where the $P_i$ and $g$ copies of $S^1\times S^2$ are the prime factors of its spherical closure $\scl{M}$ and $m$ is the number of spherical boundary components.

The $S^1\times S^2$ factors arising in the prime decomposition  correspond to a maximal collection of non-separating spheres in $M$. By allowing cuts along both separating and non-separating spheres, one can obtain a decomposition of $\scl{M}$ into only irreducible components, at the expense of possibly introducing copies of $S^3$ (we make the notion of cutting along a sphere precise in \cref{section: preliminaries}). While the diffeomorphism classes (counted with multiplicity) of irreducible factors $P_i\neq S^3$ appearing in such a decomposition are uniquely determined, the spheres along which~$M$ is cut into irreducible pieces are not, even when considered up to isotopy. 
With this in mind, we build a space $\sep(M)$ that parameterises all such decompositions of a given 3-manifold~$M$. 
A point in this space corresponds to a collection of disjointly embedded spheres $\Sigma$ that cut $M$ into irreducible pieces, which we call a \emph{separating system} for $M$.  A schematic of a 3-manifold with two examples of separating systems are shown in Figure~\ref{fig: sep systems}. 
\begin{figure}[h!]
    \centering
     \resizebox{\linewidth}{!}{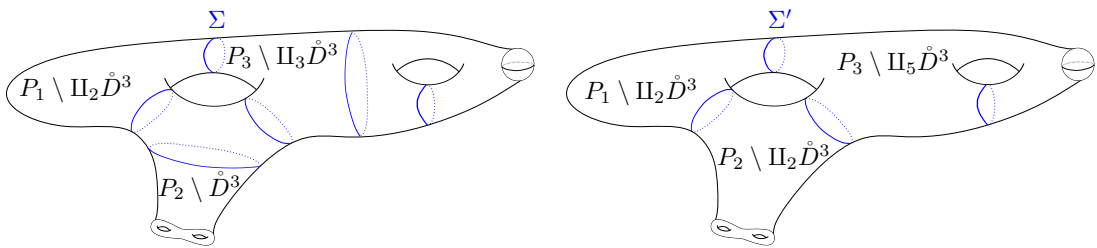}
    \caption{A schematic for the 3-manifold $M \cong \left(P_1\# P_2 \#P_3 \# (S^1\times S^2)^{\# 2}\right)\setminus \interior{D}^3$. $M$ has one spherical boundary component and one genus 2 surface boundary component. In this case $P_1$ and~$P_3$ are closed, and $P_2$ has boundary a genus 2 surface. Two separating systems~$\Sigma$ and~$\Sigma'$ are shown. Note that $\Sigma$ is obtained from $\Sigma'$ by adding two spheres, thus $\Sigma
'\subset \Sigma$.}
    \label{fig: sep systems}
\end{figure}

We then restrict to a subspace $\sepNP(M) \subset \sep(M)$ of separating systems $\Sigma \subset M$ such that no two spheres in $\Sigma$ are isotopic. This is a topological poset under inclusion of separating systems, e.g.~in Figure~\ref{fig: sep systems}, $\Sigma' \subset \Sigma$. A key step in the proof of Theorem~\ref{thm:more general Kontsevich-finiteness-conjecture} is showing that when~$M\not\cong S^1\times S^2$ the geometric realisation of the nerve of this poset is contractible.

\begin{restate}{Theorem}{thm: sepNP contractible}
    Suppose $M\not\cong S^1 \times S^2$. Then $\|\sepNP_\bullet(M)\|\simeq *$.
\end{restate}

Because $\Diff_F(M)$ acts on $ \sepNP(M)$, there is an induced action on~$\|\sepNP_\bullet(M)\|$, and thus we acquire a model for the classifying space:
\[
\BDiff_F(M)\simeq \|\sepNP_\bullet(M)\|\hq \Diff_F(M)
\]
where the notation $\hq$ denotes the homotopy orbit construction. Since no two spheres in $\Sigma \in \sepNP(M)$ are isotopic, the poset has finite depth, \emph{i.e.}~there is a bound on the longest non-degenerate descending chain.
We use this fact and some additional properties of $\sepNP_\bullet(M)$ to arrive at an inductive argument to prove Theorem \ref{thm:more general Kontsevich-finiteness-conjecture}. The induction is over the number of irreducible pieces into which a separating system~$\Sigma$ cuts $M$, the base case being when $\scl{M}$ is irreducible. 
When $M=\scl{M}$ is itself irreducible (\emph{i.e.}~$M$ has no spherical boundary components) this is due to Hatcher and McCullough  \cite{HatcherMcCullough}, see Theorem~\ref{thm: HatcherMcCullough finiteness}. In the case where
$M$ has spherical boundary, the base case further reduces to the following theorem, equivalent to the conjecture when $\scl{M}$ is irreducible and $M$ has exactly one spherical boundary component that is fixed.

\begin{restate}{Theorem}{thm:sect4main}
    Let $M$ be an irreducible 3-manifold with either empty or incompressible toroidal boundary, and let $D^3\subset \interior{M}$ be an embedded disk. 
    Then $\BDiff_{D^3}(M)$ has the homotopy type of a finite CW complex.
\end{restate}

The proof of this theorem comprises a large portion of our work and breaks down into many sub-cases. Via a fiber sequence argument, it is equivalent to showing that $\Fr(M)\hq \Diff(M)$ has the homotopy type of a finite CW complex, where $\Fr(M)$ is the frame bundle of $M$. In a similar vein to the proof of Theorem \ref{thm:more general Kontsevich-finiteness-conjecture}, we cut $M$ up along embedded tori and prove the result for the simpler pieces, via a combination of the JSJ  and geometric decompositions of our irreducible 3-manifold. 
We use this to reduce the proof to the case of hyperbolic manifolds, non-Haken Seifert-fibered manifolds,  Haken Seifert-fibered manifolds, and $\Sol$ manifolds that are torus bundles over the circle.  A schematic of the proof and its breakdown into cases is depicted in \cref{fig:flowchartirreducible}.

For hyperbolic or non-Haken Seifert-fibered manifolds, as well as  torus bundles admitting $\Sol$ geometry, the proof of \cref{thm:sect4main} relies on the fact that the (strong) generalised Smale conjecture holds: $\Isom(M)\hookrightarrow \Diff(M)$ is a homotopy equivalence. 
The case of Haken Seifert-fibered manifolds splits up further into several subcases, depending on the 
topology of the space of Seifert fiberings and the presence of singular fibers. The space of Seifert fiberings of such a manifold is completely determined by work of Hong--Kalliongis--McCullough--Rubinstein \cite{HKMR12}, and, in particular,  there are only six orientable Haken Seifert-fibered manifolds for which there does not exist a unique Seifert fibering up to isotopy. Apart from these exceptional cases, we  use the fibering to reduce the question of finiteness to that of diffeomorphisms of the base surface fixing a disk. 

For compact manifolds $M$ such that $\partial M$ is either empty or consists of spheres and incompressible tori, our proof that $\BDiff(M)$ is of finite type (\cref{thm:finite-type}) follows directly from \cref{thm:more general Kontsevich-finiteness-conjecture} by delooping the fiber sequence
\[
     \Diff_{D^3}(M)\longrightarrow
    \Diff(M) \longrightarrow
    \emb(D^3, \interior{M}).
\]
For the general case, when $M$ has higher genus or compressible boundary, we again use the contractibility of $\|\sepNP_\bullet(M)\|$ to reduced the claim to irreducible $M$ with $\partial M \neq \emptyset$.
For these, McCullough \cite{McCullough1991} showed that the mapping class group $\pi_0\Diff(M)$ is of finite type and Hatcher and Ivanov \cite{Hatcher76, Ivanov76} showed that $\Diff_0(M) \simeq (S^1)^{\times b}$ for $b\in \{0,1,2\}$, so $\BDiff(M)$ is of finite type.

\subsection{History and previous work on the conjecture}
 Kontsevich formulated his conjecture in the mid 1980s after reading notes of Thurston on the geometrisation conjecture~\cite{KontsevichCorrespondence2024}. 
 The original formulation of the conjecture was that $\BDiff_{D^3}(M)$ was homotopy finite for closed $M$ -- similar to the formulation that appears in our~\cref{thm:sect4main}. 
 In the 1990s, Kirby and Kontsevich adapted the conjecture to the more general setting of $M$ with non-empty boundary -- the form stated in the Kirby problem list \cite[Problem 3.48]{KirbyProblems}. This is the conjecture we call `Kontsevich's conjecture' and prove in \cref{thm:Kontsevich-finiteness-conjecture}.
 
 As we noted above, Hatcher and McCullough \cite{HatcherMcCullough} proved Kontsevich's conjecture when~$M$ is an irreducible 3-manifold. In fact, as is the case for our result, their result holds in a more general setting when only a subset of the boundary is required to be fixed.

\begin{thm}[{\cite[Main Theorem]{HatcherMcCullough}}]\label{thm: HatcherMcCullough finiteness}
    Let $M$ be an irreducible compact connected orientable 3-manifold and let $F$ be a non-empty union of components of $\partial M$, including all the compressible ones. Then~$\BDiff_F(M)$ has the homotopy type of an aspherical finite CW complex.
\end{thm}

Any $M$ that is irreducible and has non-empty boundary is automatically Haken, hence in this case it was known from work of Hatcher \cite{Hatcher, Hatcher76} and Ivanov~\cite{Ivanov76} that the classifying spaces $\BDiff_\partial(M)$ are aspherical.  Thus, Hatcher--McCullough prove their theorem by showing that the classifying space of the mapping class group~$\pi_0\DiffB(M)$ has the homotopy type of a finite CW complex.
They remark that this is not true for closed irreducible 3-manifolds in general, but had been previously shown to be virtually true for Haken manifolds by McCullough \cite{McCullough1991}.

More recently, a homological version of Kontsevich's conjecture was shown to be true by Nariman \cite{Nariman}, who proved that $\BDiff_\partial(M)$ has finitely many non-zero homology groups which are all finitely generated, in the case of $M$ being a connected sum of irreducible 3-manifolds such that each irreducible piece has nontrivial boundary. 
In particular, this additional assumption means that the case of an irreducible manifold with a fixed disk, which we cover in \cref{thm:sect4main}, is not necessary in his paper.
Nariman builds semi-simplicial spaces of separating spheres similar to some of the intermediate spaces in our arguments, and we utilise one of his contractibility results in Section~\ref{section: systems of spheres}. 

\subsection{Background on 3-manifold diffeomorphisms}

For prime 3-manifolds, the homotopy type of $\Diff(M)$ has been studied extensively.  We provide a brief, by no means exhaustive, overview of this history prior to geometrisation here, and refer the reader to \cref{section:strongsmale} for more references and recent results. Early work focused on the Smale conjecture (which asserts that $\Diff(S^3)\simeq \Or(4)$) and mapping class groups of 3-manifolds. Initially, Cerf proved that $\pi_0\Diff(S^3)\cong \pi_0(\Or(4))\cong \Z/2$ \cite{Cerf},  before Hatcher confirmed the Smale conjecture in \cite{Hatcher}. Hatcher \cite{Hatcher1981} also  showed $\Diff(S^1\times S^2)\simeq \Omega \SO(3)\times \Or(3)\times \Or(2)$, and Ivanov \cite{Ivanov79,Ivanov82} calculated the homotopy type of $\Diff(M)$ for certain classes of lens spaces. For closed Haken 3-manifolds, Waldhausen \cite{Waldhausen68} proved that $\pi_0\Diff(M)\cong \Out(\pi_1(M))$, the outer automorphism group of $\pi_1(M)$ (when $\partial M\neq \emptyset$, one must also preserve the peripheral structure). The results of Hatcher \cite{Hatcher76} and Ivanov \cite{Ivanov76} mentioned above imply that for Haken 3-manifolds, $\Diff_0(M)$ is homotopy equivalent to a torus of dimension at most 3.  Combined with Waldhausen's theorem, this largely reduces questions of homotopy type of $\Diff(M)$ to those of $\pi_0\Diff(M)$.

Mapping class groups of Haken Seifert-fibered 3-manifolds can be computed in terms of their base orbifold surfaces, while Thurston's proof  of hyperbolisation  for Haken atoroidal 3-manifolds \cite{ThurstonAtoroidal} implies that in this case $\pi_0\Diff(M)\cong \Isom(M)$, a finite group. More generally, results of McCullough \cite{McCullough1991} imply that $\pi_0\Diff(M)$ has finite type for any Haken 3-manifold. For certain classes of non-Haken Seifert-fibered 3-manifolds such as lens spaces, analogous results to Waldhausen's theorem were obtained along with computations of mapping class groups (e.g. \cite{Asano1978,Rubinstein1979,Bonahon1983,BirmanRubinstein1984,Scott85, BoileauOtal1986}).

Since the Smale conjecture asserts that $\Diff(S^3)\simeq \Or(4)=\Isom(\bbS^3)$, following Perelman's proof of the geometrisation theorem \cite{Perelman:2002-1,Perelman:2003-1,Perelman:2003-2} it is natural to ask whether a similar statement holds for other geometric 3-manifolds.  
We refer to the statement that $\Isom(M)\hookrightarrow \Diff(M)$ is a homotopy equivalence as the \emph{strong generalised Smale conjecture}, and to the statement that $\Isom_0(M)\hookrightarrow \Diff_0(M)$ is a homotopy equivalence as the \emph{weak generalised Smale conjecture}.
The strong form is generally false, even among irreducible geometric 3-manifolds.
For example, the isometry group of any flat metric on the torus $T^3$ has finitely many connected components, while $\pi_0\Diff(T^3)\cong \GL_3(\Z)$ by Waldhausen's theorem. However, for irreducible geometric manifolds, the weak form always holds, and furthermore, the homotopy type of $\Diff_0(M)$ is known. We discuss this further, with a complete list of references, as part of our proof in Section~\ref{section:strongsmale}. 

For reducible 3-manifolds, less is known about the homotopy type of $\Diff(M)$. Hatcher showed that for~$M\cong P_1 \# P_2$ the connected sum of two irreducible 3-manifolds, one can reduce the calculation of the homotopy type of $\Diff(M)$ to the calculation for $\Diff(P_1)$ and $\Diff(P_2)$ \cite{Hatcher1981}. We give an alternative proof of this result at the end of Section~\ref{section: systems of spheres}. In general, however, $\Diff_0(M)$ need not be homotopy equivalent to a finite CW complex as in the irreducible case; indeed if $M$ has at least 3 prime factors, McCullough and Kalliongis show that $\pi_1\Diff_0(M)$ is infinitely generated \cite{KalliongisMcCullough}.

In foundational work, Laudenbach \cite{Laudenbach73} established that homotopy implies isotopy for 2-spheres in any 3-manifold $M$ (assuming the Poincar\'{e} conjecture). He used this to show that the mapping class group of a connected sum $(S^1\times S^2)^{\# g}$ is an extension of $\Out(F_g)$ by $(\Z/2)^g$, where $F_g$ is the free group of rank $g$.  More generally, mapping class groups of Haken 3-manifolds were shown to be finitely presented by Waldhausen \cite{Waldhausen68} and Grasse \cite{Grasse1989}, and mapping class groups of reducible 3-manifolds were shown to be finitely presented by Hatcher and McCullough \cite{HatcherMcCullough1990}. As part of this proof they build a complex whose vertices are given by isotopy classes of embedded spheres, and prove this complex is simply connected \cite[Proposition 2.2]{HatcherMcCullough1990}. Note that in our proof we do not define our complexes of spheres in terms of isotopy classes -- this is because the isotopy class of an embedded sphere in a 3-manifold is not contractible in general.

A program announced by C\'esar de S\'a and Rourke \cite{CesardeSaRourke}, then completed by Hendriks and Laudenbach \cite{HendriksLaudenbach1984}, and Hendriks and McCullough \cite{HendriksMcCullough},  aimed to study  $\Diff_\partial(M)$ for reducible $M$ in terms of the irreducible prime pieces $P_i$ appearing in the decomposition of $M$. They describe $\Diff_\partial(M)$ as fitting into a fiber sequence with base space an embedding space of punctured 3-disks in $M$ that give a decomposition into (once-punctured) irreducible prime pieces and 1-handles, and with fiber the product of boundary-fixing diffeomorphism groups of these complementary pieces. We do not utilise these fiber sequences during our proofs, but these papers and Hatcher's unfinished draft \cite{HatcherReducible}  provided inspiration at the outset of this project, and in particular for our upcoming work which we now outline.

\subsection{Motivation and upcoming work}\label{subsection - upcoming work}

One motivation for this work came from previous work of the first two authors \cite{BoydBregman22}, in which the homotopy type of the embedding space of a split link is studied. Dropping the assumption that $M$ be compact, when the 3-manifold $M$ is the complement of a link $\rho$ in $\mathbb{R}^3$, our space $\sepNP(M)$ agrees with the space $\sep(\rho)$ defined in \cite{BoydBregman22}, and the associated semisimplicial space $\|\sep(\rho)_\bullet\|$ is shown to be contractible in \cite[Theorem A]{BoydBregman22}. 

Another motivation came from the perspective of modular $\infty$-operads, based on the third author's joint work with Barkan on $\infty$-properads \cite{equifibered-properads}.
Operadic structures have been used effectively to organise collections of moduli spaces, for example in Budney's work on long knots~\cite{Budney}, Costello's work on Riemann surfaces with boundary \cite{costello2004ainfinity}, and Giansiracusa's work on diffeomorphisms of handlebodies~\cite{Giansiracusa11}.
In upcoming work, we will explain how the moduli spaces of connected $3$-manifolds assemble into a modular $\infty$-operad whose composition operation is given by connected sum, 
and how the contractibility of $\|\sepNP_\bt(M)\|$ implies that this modular $\infty$-operad is freely generated by irreducible manifolds.

Further to the above, inspired by the program of \cite{CesardeSaRourke}, \cite{HendriksLaudenbach1984}, \cite{HendriksMcCullough}, and \cite{HatcherReducible} outlined above, in our upcoming work we will also construct a map on the level of classifying spaces
from~$\BDiff(M)$ to $\BDiff(\amalg_i P_i)$ where the~$P_i$ are the irreducible prime factors of~$M$. 
We describe
the fiber of this map as a homotopy colimit of certain framed configuration spaces on the~$P_i$, enabling effective computations. 
We illustrate this by computing the rational cohomology ring of~$\BDiff\left((S^1\times S^2)^{\# 2}\right)$. We will discuss past work on this topic in greater detail in this upcoming paper.

\subsection{Outline}
In Section~\ref{section: preliminaries} we give background on several topics, set up notation and prove general results that we quote in some specific settings later. This includes subsections on 3-manifolds, fiber sequences, the homotopical orbit-stabiliser theorem, homotopy finite spaces, and topological posets. We then introduce sphere systems in Section~\ref{section: systems of spheres} and prove in Theorem~\ref{thm: sepNP contractible} that the space of non-parallel separating systems is contractible. We also reprove Hatcher's result on connected sums of irreducible manifolds (Theorem~\ref{thm:two-prime-pieces}).
In Sections~\ref{section:strongsmale} and~\ref{section:decomposing irreducible}, we prove the base case Theorem~\ref{thm:sect4main} of the conjecture for an irreducible manifold with a disk removed. 
Following this, in Section~\ref{section: finiteness of BDiff} we prove Kontsevich's conjecture to be true in Theorem~\ref{thm:more general Kontsevich-finiteness-conjecture}, and prove $\BDiff(M)$ is finite type in \cref{thm:finite-type}.

\subsection{Acknowledgements}
We would like to thanks Oscar Randal-Williams for many insightful conversations, in particular at the outset of this project. We also extend thanks to Henry Wilton, specifically for discussions and suggestions relating to \cref{section:strongsmale}.
We would also like to thank the participants of the Glasgow reading group on mapping class groups of 3-manifolds, and the Copenhagen question seminar for conversations related to this paper. Finally we would like to thank Mark Powell and Nathalie Wahl for helpful comments on a draft.

During this work the first author was supported by ERC grant No.~756444, and EPSRC Fellowship EP/V043323/1 and EP/V043323/2. 
The second author was supported by NSF grant DMS-2401403. 
The third author was supported by the Independent Research Fund Denmark (grant no.~10.46540/3103-00099B)
and the Danish National Research Foundation through the ‘Copenhagen Centre for Geometry and Topology’ (grant no.~CPH-GEOTOP-DNRF151).
\section{Background on 3-manifolds and homotopy theory}\label{section: preliminaries}
In this section we set up notation concerning $3$-manifolds and notions of finiteness, recall some results about fiber sequences, topological posets, and the homotopy orbit construction, and prove a homotopical orbit stabiliser lemma.

\subsection{3-manifolds}
    All manifolds we consider are smooth.
    Throughout,~$M$ will be a compact, orientable 3-manifold, possibly with boundary, which we denote~$\partial M$. In general, we will not assume that~$M$ is connected and instead specify when we do. We will use the notation $\scl{M}$ to denote the manifold obtained from~$M$ by filling all boundary 2-spheres with 3-balls, and we call this the \emph{spherical closure} of $M$.

We say $M$ is \emph{irreducible} if every embedded 2-sphere in~${M}$ bounds a 3-ball. $M$ is \emph{reducible} if it is not irreducible.
A connected 3-manifold~$M$ without 2-spheres in its boundary is \emph{prime} if whenever we decompose~${M}$ as a nontrivial connected sum~${M}=M_1\# M_2$, then at least one of~$M_1$ or~$M_2$ is diffeomorphic to~$S^3$.
A well-known theorem due to Kneser \cite{Kneser1929} states that every connected oriented 3-manifold without 2-spheres in its boundary admits a connected sum decomposition $M=P_1\#\cdots\# P_n$ where the \emph{prime factors} $P_i$ are prime and not diffeomorphic to~$S^3$.  Moreover Milnor showed that the oriented prime factors appearing in this decomposition are uniquely determined up to reordering~\cite{Milnor62}. 
By convention, $S^3$ has 0 prime factors. Note that $S^1\times S^2$ is the unique connected 3-manifold without boundary which is prime but not irreducible. We can therefore refine the decomposition to be $M\cong P_1\#\cdots \#P_n \# (S^1\times S^2)^{\# g}$ where the $P_i$ are the prime factors which are also irreducible. 
The reader is referred to Hempel \cite{Hempel} for proofs of these and other foundational results on 3-manifolds. 

In this paper we often work in the setting where $M$ has spherical boundary components. 
In this case we will never talk about prime factors of $M$ but only discuss prime factors of the spherical closure $\scl{M}$. We also extend this definition to disconnected manifolds: in this case a prime factor of $M$ is a prime factor of one of the connected components.

If $\Sigma\subset \interior{M}$ is a codimension $1$ submanifold, such as a disjoint union of embedded 2-spheres, then $M\setminus \Sigma$ is diffeomorphic to the interior of a compact 3-manifold with boundary, which we denote as $M\ca \Sigma$. Intuitively, $M\ca \Sigma$ is the manifold obtained from \emph{cutting} $M$ \emph{along} $\Sigma$. 
In particular, $(M \ca \Sigma)^\circ = \interior{M} \setminus \Sigma$.
Note that the boundary of $M \ca \Sigma$ contains two spheres for each sphere in $\Sigma$. 
We let $2\Sigma \subset \partial(M \ca \Sigma)$ denote the newly created boundary components.
There is a canonical map $2\Sigma \to \Sigma$ that is a double cover, and for us this will always be a trivial double-cover because $\Sigma$ will be $2$-sided/coorientable.
A schematic is shown in Figure~\ref{fig: M minus Sigma}. 

\begin{figure}[ht]
    \centering
     \resizebox{\linewidth}{!}{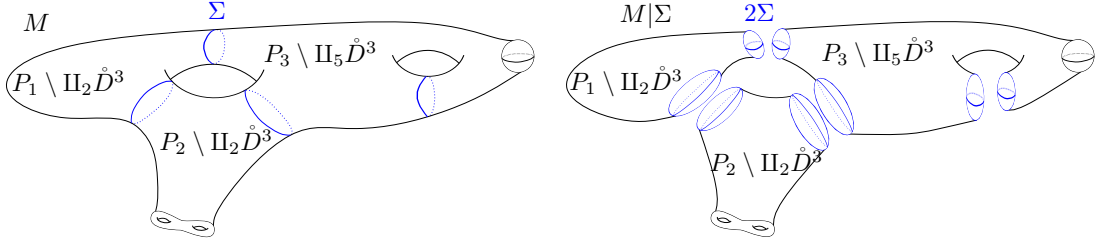}
    \caption{A disjoint union of embedded $2$-spheres~$\Sigma \subset M$ and the compact 3 manifold~$M\ca \Sigma$. We remark that the irreducible prime factors of both manifolds are the same.}
    \label{fig: M minus Sigma}
\end{figure}

\subsection{Fiber sequences of embedding spaces}
In this section we review results on fiber sequences of embedding spaces due to Palais \cite{Palais60} in the closed case and independently by Cerf \cite{Cerf} for manifolds with with corners and in greater generality (see also Lima \cite{Lima64}, who gave a simplified proof of Palais' result). The treatment we give here closely follows the exposition in \S 2 of Cantero Mor\'an--Randal-Williams \cite{CMRW17}.

For manifolds $M$ and $V$, we let $\emb(V,M)$ denote the space of embeddings of $V$ into $M$, equipped with the $C^\infty$-topology. 
The group of diffeomorphisms of $M$ is denoted $\Diff(M)$. If $V\subseteq M$ is a submanifold, we denote by $\Diff_V(M)$ the subgroup of $\Diff(M)$ consisting of diffeomorphisms which fix $V$ pointwise, and by $\Diff(M,V)$ the subgroup of $\Diff(M)$ consisting of diffeomorphisms which preserve $V$ setwise, \emph{i.e.}~$\varphi\in \Diff(M,V)$ if and only if $\varphi(V) = V$.
We denote the subgroup $\Diff_{\partial M} (M)$ by $\Diff_\partial(M)$.

Consider a compact manifold $M$ and a compact submanifold $V \subset \interior{M}$.
Both $M$ and $V$ may have boundary.
Palais and Cerf showed that there is a fiber sequence
\[\Diff_V(M)\hookrightarrow \Diff(M)\rightarrow \emb(V,\interior{M})\]
where the fibration is given by acting with $\Diff(M)$ on the standard embedding $V \subset M$ \cite{Palais60,Cerf}.
To prove this, it in fact suffices to construct a local section of the action map, which then can be used to obtain local trivializations.
This can be formalised using the concept of $G$-locally retractile spaces, as explained in \cite[\S2.3]{CMRW17}.
\begin{defn}\label{defn:locally-retractile}
    A space $X$ with an action of a topological group $G$ is \emph{$G$-locally retractile} 
    if for each $x \in X$ there is a neighbourhood $U \subset X$ and a map $\xi\colon U \to G$ such that $\xi(u).x = u$ holds for all $u \in U$.
\end{defn}

These $G$-locally retractile spaces will prove useful for showing that certain maps are locally trivial fiber bundles.
We will need the following elementary properties \cite[Lemmas 2.4, 2.5, 2.6]{CMRW17}.
\begin{lem}\label{lem:locally-retractile}
    Let $X$ be a $G$-locally retractile space.
    \begin{enumerate}[(1)]
        \item 
        Any $G$-equivariant map $f\colon Y \to X$ is a locally trivial fiber bundle.
        \item 
        If a $G$-equivariant map $g\colon X \to Z$ has local sections, then $Z$ is also $G$-locally retractile.
        \item If $X$ is locally path connected, then $X$ is also $G_0$-locally retractile for $G_0 \subset G$ the path component of the identity.
        \item If $G < H$ is a larger topological group such that the $G$-action on $X$ extends to an $H$-action, then $X$ is also $H$-locally retractile.
    \end{enumerate}
\end{lem}

Now Palais' and Cerf's theorem can be stated as follows.
As before, $M$ and $V$ are compact manifolds, possibly with boundary.
\begin{thm}\label{thm:emb-retractile}
    The embedding space $\emb(V, \interior{M})$ is $\Diff_c(\interior{M})$-locally retractile.
\end{thm}

Here $\Diff_c(\interior{M})$ is the group of diffeomorphisms of $\interior{M}$ with compact support, which we can equivalently think of as the subgroup of $\Diff_\partial(M)$ containing those diffeomorphisms that fix some (unspecified) neighbourhood of the boundary.
Using \cref{lem:locally-retractile}(4), we see that $\emb(V, \interior{M})$ is also $\Diff_\partial(M)$- and $\Diff(M)$-locally retractile.
Moreover, $\emb(V, \interior{M})$ is locally path connected, so we may restrict to the path component of the identity in these groups.
Cerf not only considers the case where $V$ is embedded in the interior of $M$, but also allows for more general ``face constraints'', see \cite[\S2.2 and Proposition 2.9]{CMRW17}.
The specific case we will need is the following:
\begin{thm}\label{thm:boundary-locally-retractile}
    The embedding space $\emb(V, \partial M)$ is $\Diff(M)$-locally retractile.
    In particular $\Diff(\partial M)$ is $\Diff(M)$-locally retractile.
\end{thm}

There is a natural right action of $\Diff(V)$ on $\emb(V,\interior{M})$ by precomposition.
This action is free and has the effect of reparametrising an embedding, hence the quotient 
\[\umb(V,\interior{M}):=\emb(V,\interior{M})/\Diff(V)\]
is the space of \emph{unparametrised embeddings} of $V$ into $N$.
As a set we can identify this with the set of submanifolds of $\interior{M}$ that happen to be diffeomorphic to $V$.
By Binz--Fisher\cite{BinzFischer78} and Michor \cite[13.11]{Michor80}, the quotient map $\emb(V,\interior{M})\rightarrow \umb(V,\interior{M})$ is a locally trivial fiber bundles with fiber $\Diff(V)$.
(This is proved only when $V$ is closed but, as pointed out in \cite[Propositon 2.11]{CMRW17}, the proofs adapt to the setting of manifolds with boundary.)
Crucially for us, this implies that the unparametrised embedding space is $\Diff(V)$-, and hence $\Diff(M)$-, locally retractile.
\begin{cor}\label{cor:umb-retractile}
    The space of unparametrised embeddings $\umb(V, \interior{M})$ is $\Diff_\partial(M)$-locally retractile.
\end{cor}

As before this implies that $\umb(V, \interior{M})$ is $\Diff(M)$ and also $\Diff_0(M)$-locally retractile.
The stabiliser groups of the action by $\Diff(M)$ are precisely $\Diff(M, V)$, so we have a fiber sequence
\[
    \Diff(M, V) \longrightarrow \Diff(M) \longrightarrow \umb(V, \interior{M}).
\]

We will also use the following standard results about the contractibility of collars.

\begin{thm}[{\cite[5.2.1, Corollaire 1]{Cerf}}]\label{thm:contractible-collars}
    The space of collars
    \[
        \emb_{\partial M \times \{0\}}(\partial M \times I, M) = \{ i \in \emb(\partial M \times [0,1], M) \;|\; i_{|\partial M \times \{0\}} = \mathrm{id}_{\partial M} \}
    \]
    is weakly contractible.
    It follows that the subgroup inclusion
    \[
        \Diff_U(M) \to \Diff_{\partial U}(M \setminus \interior{U})
    \]
    is a weak equivalence
    for any compact codimension $0$ submanifold $U \subset \interior{M}$.
\end{thm}

\subsection{The homotopy orbit construction}

Given a topological group $G$, we say that a $G$-space $X$ is principal if $X \to X/G$ is a $G$-principal bundle. 
For each $G$, we fix a weakly contractible $G$-principal space $EG$.
Then the \emph{homotopy orbit construction} of a $G$-space $X$ is defined as
\[
    X \hq G := (X \times EG)/G
\]
and we also refer to it as the \emph{homotopy quotient} of $X$ by $G$.
This fits into a fiber sequence
\[
    X \longrightarrow X\hq G \longrightarrow * \hq G = BG.
\]

\begin{defn}\label{defn: principal ses}
    A \emph{principal short exact sequence} is a short exact sequence of topological groups
    \[
        1 \to G_1 \to G_2 \to G_3 \to 1
    \]
    such that the map $G_2 \to G_3$ is a $G_1$-principal bundle.
    Equivalently, we can say that $G_3$ is $G_2$-locally retractile with stabiliser $G_1$.
\end{defn}

The following lemma allows us to combine principal short exact sequences of groups and equivariant fiber sequences to obtain a homotopy fiber sequence of homotopy quotients.
See Basualdo Bonatto \cite[Corollary 2.11]{bonatto2023} for a detailed proof.
\begin{lem}\label{lem:fiberseq/fiberseq}
    Suppose we have a principal short exact sequence of topological groups $G_i$ as above and let $S_i$ be a $G_i$-space such that there exists a fiber sequence of equivariant maps
    $S_1 \to S_2 \to S_3$.
    Then the maps on quotients form a homotopy fiber sequence
    \[
        S_1 \hq G_1 \longrightarrow S_2 \hq G_2 \longrightarrow S_3 \hq G_3.
    \]
\end{lem}

The following fact about principal bundles seems to be well-known, but we provide a proof as we could not find a reference.
\begin{lem}\label{lem:principal-bundle-subgroup}
    If $H \le G$ is a topological subgroup such that $G \to G/H$ is an $H$-principal bundle and $X$ is a principal $G$-space, then $X$ is also a principal $H$-space.
\end{lem}
\begin{proof}
    First we prove that $X$ is a principal $G$-space if and only if we can find a family of subsets $\{U_i \subset X\}_{i \in I}$ such that the action maps $U_i \times G \hookrightarrow X$ are open embeddings and jointly cover $X$.
    If $p\colon X \to X/G$ is a $G$-principal bundle, then there is an open covering $\{U_i'\}_{i \in I}$ of $X/G$ such that there are equivariant homeomorphisms $\alpha_i\colon U_i' \times G \cong p^{-1}(U_i')$ over $U_i'$. We can then set $U_i = \alpha_i(U_i' \times \{1\})$.
    Conversely, if we have $\{U_i \subset X\}_{i\in I}$, then $U_i' := p(U_i)$ is open because it is the image of the open map $U_i \times G \hookrightarrow X \to X/G$, and $p$ will be trivial over these $U_i'$.

    Applying this to $G \to G/H$ we find $\{V_j \subset G\}_{j \in J}$ such that $V_j \times H \hookrightarrow G$ are open embeddings covering $G$.
    Now we can set $W_{i,j}$ to be the image of the embedding $U_i \times V_j \subset U_i \times G \hookrightarrow X$ and these will be such that 
    \[
        W_{i,j} \times H \cong U_i \times V_j \times H
        \hookrightarrow U_i \times G \hookrightarrow X
    \]
    are open embeddings covering $X$.
    Therefore $X \to X/H$ is an $H$-principal bundle.
\end{proof}

Another useful source of homotopy fiber sequences will be the \emph{homotopical orbit stabiliser lemma}.
First we consider the case of a transitive action.
\begin{lem}\label{lem:orbit-stabiliser-transitive}
    If $X$ is a $G$-locally retractile space and the $G$ action on $X$ is transitive, 
    then 
    \[
        X\hq G \simeq B\mathrm{Stab}_G(x)
    \]
    for $\mathrm{Stab}_G(x) < G$ the stabiliser group of some $x \in X$.
    Moreover, there is a homotopy fiber sequence
    \[
        X \longrightarrow B\mathrm{Stab}_G(x) \longrightarrow BG.
    \]
\end{lem}
\begin{proof}
    For each orbit the map $G \to X$ defined by acting on any point $x \in X$ is a quotient map, so since the action is transitive we can identify $X$ with the left quotient $H\setminus G$ by the stabiliser group $H = \mathrm{Stab}_G(x)$.
    We now compute
    \[
        X \hq G = ((H \setminus G) \times EG)/G
        \cong H \setminus ((G \times EG)/G)
        \cong H \setminus EG
    \]
    Because $H = \mathrm{Stab}_G(x) < G$ is such that $G \to H \setminus G$ is a principal $H$-bundle and $EG \to G \setminus EG = BG$ is also a principal bundle, \cref{lem:principal-bundle-subgroup} says that $EG \to H \setminus EG$ is a principal bundle.
    Hence we may define $EH := EG$ such that $X \simeq H \setminus EG = BH =  B\mathrm{Stab}_G(x)$.
    The homotopy fiber sequence is obtained from $X \to X\hq G \to BG$ by rewriting the middle term.
\end{proof}

\begin{lem}[Homotopical orbit-stabiliser lemma]\label{lem:orbit-stabiliser-v2}
    Let $X$ be a $G$-locally retractile space 
    and assume that $X$ is locally path-connected.
    Then there is an equivalence
    \[
        X\hq G \simeq \coprod_{[x] \in X/G} B\mathrm{Stab}_G(x)
    \]
    where the coproduct runs over orbit representatives $[x]\in X/G$.
\end{lem}
\begin{proof}
    Because $X$ is $G$-locally retractile and locally path-connected, each orbit of the $G$-action is a union of path components.
    We further have $X/G = \pi_0(X)/\pi_0(G) = \pi_0(X\hq G)$.
    Writing $X = \coprod_i X_i$ as the disjoint union of its orbits we have $X \hq G = \coprod_i X_i \hq G$
    and the claim follows from \cref{lem:orbit-stabiliser-transitive}.
\end{proof}

\begin{lem}\label{lem:homotopy-orbits-functor}
    The homotopy orbits functor $
        (-)\hq G\colon G\mathrm{-Top} \longrightarrow \mathrm{Top}$
    \begin{enumerate}
        \item\label{it:hq-colim}
        commutes with all colimits,
        \item\label{it:hq-prod}
        commutes with products with trivial $G$-spaces:
        if $X$ is any $G$-space and $Y$ has the trivial $G$-action then $(X \times Y) \hq G \cong X \hq G \times Y$,
        \item\label{it:hq-weq}
        preserves weak equivalences,
        \item\label{it:htp-po}
        preserves homotopy pushout squares, and
        \item\label{it:geo-real}
        preserves (fat) geometric realisations.
    \end{enumerate}
\end{lem}
\begin{proof}
    All of these properties are standard, but we recall them here for the convenience of the reader.
    The homotopy orbits functor is the composite of the functor $(-) \times EG$ and the functor $(-)/G$.
    In a convenient category of topological spaces the cartesian product preserves all colimits in each variable, and the second functor preserves colimits because colimits commute, hence (\ref{it:hq-colim}) follows.
    For (\ref{it:hq-prod}) we note that $(-)\times Y$ certainly commutes with the first functor and it also commutes with the second functor because $(-) \times Y$ preserves all colimits and hence in particular preserves taking orbits by $G$.
    Item (\ref{it:hq-weq}) follows form the long exact sequence for the fiber sequences $X \to X \hq G \to BG$.
    Because we now know that $(-)\hq G$ preserves weak equivalences, to show (\ref{it:htp-po}) it will suffice to prove that the functor preserves the standard construction of the homotopy pushout as
    \[
        X \cup^h_Y Z = X \cup_Y (Y \times [0,1]) \cup_Y Z.
    \]
    This indeed is preserved because we already noted that $(-)\hq G$ preserves pushouts and products with trivial $G$-spaces.
    Finally, for (\ref{it:geo-real}), let $X_\bullet$ be semi-simplicial $G$-space.
    Then there is a homeomorphism $\|X_\bullet \hq G\| \cong \|X_\bullet\|\hq G$ because the fat geometric realisation is a colimit of products with the topological $n$-simplices, which are trivial $G$-spaces.
\end{proof}

\subsection{Homotopy finite spaces}\label{section: homotopy-finite-spaces}
We make precise the notion of finiteness that are used throughout the paper and prove the fiber sequence lemma that is our most-used tool for showing finiteness.

\begin{defn}
    We say that a topological space $X$ is \emph{homotopy finite} if there is a CW complex $C$ with finitely many cells and a weak equivalence $C \simeq X$.
\end{defn}

\begin{rem}
    If $X$ is homotopy finite, then $\pi_1(X)$ is finitely presented, $H_n(X)$ is finitely generated for all $n\ge 0$, and there is a $d \in \mathbb{N}$ such that $H_n(X) = 0$ for $n \ge d$.
    A discrete group is $G$ called \emph{type} $F$  if $BG$ is homotopy finite.
    In this case $G$ must be finitely presented and have finite cohomological dimension.
\end{rem}

Homotopy finite spaces are closed under finite disjoint union and  homotopy pushouts.
They are also closed under extensions in the following sense.

\begin{lem}\label{lem:fiber-sequence-finite}
    Let $p\colon E \to B$ be a map such that $B$ is homotopy finite and the homotopy fiber of $p$ at any $b \in B$ is homotopy finite.
    Then $E$ is homotopy finite.
\end{lem}
\begin{proof}
    Without loss of generality $B$ itself is a finite CW complex and $p$ is a fibration.
    We proceed by induction on the dimension of $B$ and the number of cells in $B$.
    If $B$ is a disk, then $E$ is equivalent to the homotopy fiber and the claim follows.
    Otherwise, we can find an open decomposition $B = U \cup V$ where $U$, $V$, and $U \cap V$ are all equivalent to finite CW complexes that have either fewer cells than $B$ or a lower dimension.
    (Let, for instance, $U \subset B$ be an open cell, $x \in U$ and $V = B \setminus \{x\}$.)
    By induction hypothesis, we see that $p^{-1}(U)$, $p^{-1}(V)$ and $p^{-1}(U \cap V)$ must all be equivalent to finite CW complexes, as they are fibrations over a homotopy finite space and have homotopy finite fibers.
    Now $E = p^{-1}(U) \cup_{p^{-1}(U \cap V)} p^{-1}(V)$ is a homotopy pushout of homotopy finite spaces and as such it is homotopy finite as claimed.
\end{proof}

\subsection{Topological posets}\label{section - discretization}

Recall that a topological poset $(P, \leq)$ is a poset together with a topology on the underlying set~$P$.
A map of topological posets $(P, \leq) \to (Q,\leq)$ is a continuous map $f\colon P \to Q$ such that $x \le y \Rightarrow f(x) \le f(y)$.
We let $N(P)$ denote the nerve, \emph{i.e.}~the simplicial space defined by 
    \[
        N_n(P) = \{p_0 \le \dots \le p_n\} \subset P^{n+1}
    \]
and we define the classifying space $BP := \|N_\bullet(P)\|$ to be the fat geometric realisation of the nerve.
We use the fat geometric realisation as opposed to the thin one as we will apply tools from Galatius--Randal-Williams \cite{GalatiusRandalWilliams} and Ebert--Randal-Williams \cite{EbertRandalWilliams}, who work with semi-simplicial spaces.
We will often use the following elementary technique for constructing equivalences of posets, see \cite[Lemma 3.4]{EbertRandalWilliams} for a proof.

\begin{lem}\label{lem:poset-natural-transformation=homotopy}
    If $f,g: P \to Q$ are maps of topological posets such that $f(x) \ge g(x)$ for all $x \in P$, then the induced maps on classifying spaces 
    $B(f), B(g): BP \to BQ$
    are homotopic.
\end{lem}

The topological posets we encounter usually arise from a specific construction.
Fix $X$ a Hausdorff topological space and let $\perp$ be a symmetric and anti-reflexive relation on $X$ such that the subspace
\[
    \{(x,y) \in X^2 \mid x \perp y\} \subset X^2
\]
defining $\perp$ is an open subset.
We say that $x$ and $y$ are \emph{orthogonal} when $x \perp y$ and  that an $n$-tuple of points $(x_1, \dots, x_n) \in X^n$ is \emph{orthogonal} if and only if $x_i \perp x_j \,\forall i\neq j$.
As an example, we will mainly be interested in the setting where $X = \umb(S^2, \interior{M})$ and $S_0 \perp S_1$ if and only if $S_0 \cap S_1 = \emptyset$.
In particular, $\perp$ will usually not be transitive.

\begin{defn}\label{defn - C/S(X, perp)}
    Let $S(X, \perp)$ denote the (discrete) simplicial complex where vertices are points in $X$ and $n+1$ vertices span an $n$-simplex if and only if they are orthogonal.
    Let $C(X, \perp)$ denote the space of non-empty unordered orthogonal configurations, 
    \emph{i.e.}~the subspace of $\coprod_{k\ge 0} X^k/\Sym_k$ of those $\ul{x} = [x_1,\dots,x_k]$ such that $x_i \perp x_j$ holds for all $i \neq j$.
    We make $C(X, \perp)$ into a topological poset with respect to the subset ordering $\subseteq$.
    For any orthogonal $\ul{y} = (y_1,\dots,y_n)$ we define $S(X, \perp)^{\perp \ul{y}} \subset S(X, \perp)$ as the full subcomplex on those vertices $x$ that are orthogonal to all $y_i$.
    Similarly, we define $C(X, \perp)^{\perp \ul{y}} \subset C(X,\perp)$ as the topological subposet containing those configurations that are orthogonal to $\ul{y}$. 
\end{defn}

The goal of this section is to prove the following useful tool, which is a variation of a technique by Galatius--Randal-Williams \cite{GalatiusRandalWilliams}.
\begin{prop} \label{prop - poset discretization}
    Suppose that $X$ is a Hausdorff space with an open, symmetric, and anti-reflexive relation $\perp$ as above.
    Suppose further that 
    \begin{enumerate}
        \item the simplicial complex $S(X, \perp)$ is $n$-connected, and
        \item for all orthogonal tuples $\ul{y} = (y_1, \dots, y_q)$ the simplicial complex $S(X, \perp)^{\perp \ul{y}}$ is $(n-1)$-connected. 
    \end{enumerate}
    Then the classifying space $B(C(X,\perp))$ of the poset $C(X,\perp)$ is $n$-connected.
\end{prop}
\begin{proof}
    Write $C := C(X, \perp)$ and $S := S(X, \perp)$,
    and let $C^\delta$ be $C$ equipped with the discrete topology.
    First, note that $C$ is the poset of simplices of $S$ and hence $B(C^\delta) = |N_\bullet C^\delta|$ is the barycentric subdivision of $|S|$.
    In particular they are homeomorphic, and similarly $|N_\bullet(C^{\perp \ul{y}})^\delta| \cong |S^{\perp \ul{y}}|$.
    From now on all geometric realisations are ``fat'', \emph{i.e.}~we do not quotient by the relations imposed by the degeneracy maps.
    For simplicial sets the fat and thin geometric realisations are always equivalent \cite[Lemma 1.7]{EbertRandalWilliams}.
    
    We define the following bisimplicial subspace of $(N_\bullet C)^\delta \times N_\bullet C$:
    \[
        D_{p,q} := \{ ((\ul{x}^0 \subseteq \dots \subseteq \ul{x}^p),(\ul{y}^0 \subseteq \dots \subseteq \ul{y}^q)) \in N_p C^\delta \times N_q C \;|\; \ul{x}^p \perp \ul{y}^q \}.
    \]
    Note that by $\ul{x}^p \perp \ul{y}^q$ we mean that $\ul{x}^i \perp \ul{y}^j$ for all $i$ and $j$.
    We will also need the diagonal of $D_{\bullet, \bullet}$, which is the simplicial space given by 
    $(\Delta D)_p := D_{p,p}$.
    Consider the diagram
    \[
        \begin{tikzcd}[row sep = small]
            && {\|N_\bullet C^\delta\|} \ar[dd, "\iota"] \\
            {\|(\Delta D)_\bullet\|} \ar[r, "\alpha"', "\simeq"] \ar[rru, bend left] \ar[rrd, bend right] & {\|\| D_{\bullet, \bullet} \|\|} \ar[dr, "\varepsilon"'] \ar[ur, "\gamma"] & \\
            && {\|N_\bullet C\|} 
        \end{tikzcd}
    \]
    where the map $\iota$ is the canonical map from the discretisation,
    $\gamma$ projects to the $\ul{x}^i$ and $\varepsilon$ projects to the $\ul{y}^j$.
    The map $\alpha$ is the inclusion of the diagonal and it is a weak equivalence by \cite[Theorem 7.1]{EbertRandalWilliams}.
    
    \noindent
    \textit{Claim:}
    The outside diagram commutes up to homotopy, \emph{i.e.}~$\iota \circ \gamma \circ \alpha \simeq \varepsilon \circ \alpha$.
    
    \textit{Proof of claim:}
    Consider the topological subposet $Q \subset C^\delta \times C$ containing those tuples $(\ul{x}, \ul{y})$ satsifying $\ul{x} \perp \ul{y}$.
    The nerve of $Q$ is isomorphic to the diagonal $\Delta D$:
    the homeomorphism $N_p Q \cong (\Delta D)_p = D_{p,p}$ is given by the tautological map
    \[
        ((\ul{x}^0, \ul{y}^0) \subseteq \dots \subseteq (\ul{x}^p,\ul{y}^p))
        \longmapsto
        ((\ul{x}^0 \subseteq \dots \subseteq \ul{x}^p), (\ul{y}^0 \subseteq \dots \subseteq \ul{y}^p)).
    \]
    The two maps $\iota \circ \gamma \circ \alpha$ and $\varepsilon \circ \alpha$ are obtained as the realisation of the two poset-maps 
    $\mathrm{pr}_1, \mathrm{pr}_2: Q \to C$ defined by $\mathrm{pr}_1(\ul{x}, \ul{y}) := \ul{x}$ and $\mathrm{pr}_2 (\ul{x}, \ul{y}) := \ul{y}$, respectively.
    We can also define a third map $q:Q \to C$ by $q(\ul{x}, \ul{y}) = \ul{x} \cup \ul{y}$, which is well-defined
    since $\ul{x} \perp \ul{y}$ in $Q$, and therefore it follows that $\ul{x}$ and $\ul{y}$ are disjoint as subsets of $X$ ($\perp$ is anti-reflexive).
    Now, since $\ul{x} \subseteq (\ul{x} \cup \ul{y}) \supseteq \ul{y}$, Lemma \ref{lem:poset-natural-transformation=homotopy} implies that
    \[
        \iota \circ \gamma \circ \alpha = B(\mathrm{pr}_1) \simeq B(q) \simeq B(\mathrm{pr}_2) = \varepsilon \circ \alpha, 
    \]
    which proves the claim.
    
    For fixed $q$ the projection map $\varepsilon_q: \|D_{\bullet,q}\| \to N_q C$ is a micro Serre fibration by \cite[Proposition 2.8]{GalatiusRandalWilliams}.
    (To apply this proposition observe that $N_q C^{\perp \ul{x}} \subseteq N_q C$ is open.)
    The fiber $\varepsilon^{-1}(\beta)$ of this at some $\beta = (\ul{y}^0 \subseteq \dots \subseteq \ul{y}^q) \in N_q C$ is isomorphic to the simplicial set $N_\bullet ( C^{\perp \ul{y}^q} )^\delta$.
    The realisation of this is homeomorphic to $B( (C^{\perp \ul{y}^q})^\delta ) \cong \|S^{\perp \ul{y}^q}\|$, which we assumed to be $(n-1)$-connected.
    Hence, by \cite[Proposition 2.6]{GalatiusRandalWilliams} the map $\varepsilon_q: \|D_{\bullet,q}\| \to N_q C$ is $n$-connected.
    Moreover, combining this with \cite[Proposition 2.7]{GalatiusRandalWilliams} we see that $\varepsilon: \|\|D_{\bullet,\bullet}\|\| \to \|N_\bullet C\|$ is $n$-connected.
    
    We have hence shown that the map $\varepsilon \circ \alpha: \|(\Delta D)_\bullet\| \to \|N_\bullet C\|$ induces a surjection on $\pi_k$ for $k \le n$.
    But this map is homotopic to $\iota \circ \gamma \circ \alpha$, 
    which factors through $\|N_\bullet C^\delta\| \cong \|S^\delta\|$, which we assumed to be $n$-connected.
    It follows that $\pi_k \|N_\bullet C\|$ is $0$ for $k \le n$,
    \emph{i.e.}~$\|N_\bullet C\|$ is $n$-connected.
\end{proof}

The following is a special case of Quillen's Theorem A for topological categories \cite[Theorem 4.8]{EbertRandalWilliams}.
\begin{cor}\label{cor:QuillenA}
    Let $(D, \le)$ be a topological poset and 
    $C \subset D$ a sub-poset such that:
    \begin{enumerate}
        \item For all $d \in D$ the sub-poset 
        $C_{d \le} := \{c \in C \;|\; d \le c\}$
        has a contractible classifying space.
        \item $D$ is left-fibrant, {i.e.}~the map
        \[
            \{(d_0, d_1) \in D^2 \;|\; d_0 \le d_1\} 
            \longrightarrow D, 
            \qquad
            (d_0, d_1) \longmapsto d_0
        \]
        is a Serre fibration.
        \item $C \subset D$ is a union of path components.
    \end{enumerate}
    Then the inclusion $BC \to BD$ is a weak equivalence.
\end{cor}
\begin{proof}
    We need to verify conditions (i) and (v) of \cite[Theorem 4.8]{EbertRandalWilliams}.
    Our condition (1) is exactly condition (i), since $N_\bullet(C_{d\le}) = (d/F)_\bullet$ for $F:C \to D$ the inclusion.
    To check condition (v) we may use \cite[Lemma 4.6.(vi)]{EbertRandalWilliams}.
    We need to show that $C$ and $D$ are left-fibrant, $C \to D$ is a Serre fibration and $D$ is unital.
    All posets are unital as $c \le c$ provides identity morphisms. 
    $C \to D$ is a Serre fibration since it is the inclusion of a union of path components.
    Condition (2) exactly tells us that $D$ is left-fibrant and it follows that $C$ is left-fibrant by restricting to path components.
\end{proof}

\section{Systems of spheres} \label{section: systems of spheres}
 In this section we introduce complexes and semi-simplicial spaces of embedded unparametrised spheres which cut a 3-manifold into irreducible pieces. In each subsection we introduce a variant with more constraints, and show it is contractible, using contractibility of the previously considered complexes. As our starting point we sketch the proof of a contractibility result due to Nariman~\cite{Nariman}, adjusted to our requirements. These contractibility results are a key input to the proof of our main theorem. We end the section by reproving a theorem of Hatcher, in the setting where $M$ is a connected sum of two irreducible manifolds.
 
\subsection{Single spheres} 
First we consider a complex where vertices are single spheres.

\begin{defn}\label{defn - essential sphere}
An embedded unparametrised sphere~$S\subset M$ is \emph{essential} if it represents a non-trivial class in~$\pi_2({M},p)$ for some $p\in M$. We call an essential sphere~$S \subset M$ \emph{reducing} if each connected component of the spherical closure $\scl{M \ca S}$ of $M \ca S$ has fewer prime factors than the spherical closure of $M$.
\end{defn}

\begin{rem}\label{rem - essential does not bound}
It follows from the Poincar\'e conjecture that an embedded sphere~$S\subset M$ is essential if and only if it does not bound~$B^3\subset M$. Therefore, a sphere~$S \subset M$ is reducing if and only if~$S$ is essential in the spherical closure~$\scl{M}$.
\end{rem}

\begin{defn}\label{defn: sphere complex S(M)}
    The \emph{sphere complex} $\mathcal{S}(M)$ of a compact $3$-manifold $M$  is the (discrete) simplicial complex where vertices are reducing spheres in $M$ and $p$-simplices are spanned by finite sets of vertices $\{S_0, \dots, S_p\}$ such that the $S_i$ are pairwise disjoint.
\end{defn}

Hatcher was the first to consider such a complex when he introduced and studied the complex of isotopy classes of essential spheres in $S^1\times S^2$ in \cite{Hatcher95}. He proved this complex was contractible \cite[Theorem 2.1]{Hatcher95} using a surgery argument that has since been adapted by Mann--Nariman \cite{MannNariman}, Nariman \cite[Proposition 2.7]{Nariman}, and the first and second authors~\cite[Theorem 3.7]{BoydBregman22} among others, in various settings where embedded spheres satisfying certain additional properties are no longer considered up to isotopy. Our next result is another such adaptation.

\begin{prop}\label{thm: S(M)-contractible}
    For every compact $3$-manifold $M$, the complex $\mathcal{S}(M)$ is either empty or contractible. 
\end{prop}
\begin{proof}
    The argument from the proof of Proposition 2.7 in \cite{Nariman} works  with some minor changes, since we consider reducing spheres but Nariman considers the larger class of essential spheres that are not boundary parallel. We sketch the proof with modifications here.

    Suppose~$\mathcal{S}(M)$ is nonempty and let~$\kappa \colon M\rightarrow \scl{M}$ be the natural inclusion which caps all 2-sphere boundary components of $M$. 
    We will show that any continuous map~$f\colon S^k\rightarrow \mathcal{S}(M)$ is nulhomotopic.  We may assume that~$f$ is simplicial with respect to some triangulation~$K$ of~$S^k$. As transverse embeddings are residual \cite[Theorem 2.1]{Hirsch}, after a simplicial homotopy we may assume that the images of the vertices of~$K$ under~$f$ are pairwise transverse. 
    Choose a vertex~$v\in K$.
    Then $f(v)$ is a reducing sphere in~$M$ which is transverse to all other spheres in~$f(K^{(0)})$.
    
    We now perform a sequence of surgeries on the spheres in $f(K^{(0)})$ to homotope the image of~$f$ into the star of~$f(v)$. 
    For every other vertex $w \in K$ the intersection $f(v) \cap f(w)$ is a union of circles.
    Let~$N$ be total number of intersection circles in $f(v)$, i.e.~the number of circles in $f(v)\cap f(w)$, summed over all $w \neq v$.
    Consider all such circles, take a maximal disjoint family, and consider an innermost circle $C$ in this family. (Here, innermost means that~$C$ bounds a disk in~$f(v)$ containing no other circle in the maximal collection.) 
    Then $C$ is an intersection between $f(v)$ and a reducing sphere~$f(w)$ for $w \in K$.
    For the surgery step, we need to show that if we surger~$f(w)$ at $C$ with respect to the transverse sphere~$f(v)$ to obtain two spheres~$S_0$ and~$S_1$ (see \cite[Figure 2]{Nariman}) then at least one of~$S_0$ or~$S_1$ is still reducing.
    Since~$f(w)$ and~$f(v)$ lie in the interior of~$M$, we may perform the surgery entirely in~$\interior{M}$. Passing to~$\scl{M}$ via~$\kappa$, we may perform the same surgery away from the added 3-balls. By assumption,~$f(w)$ and~$f(v)$ are reducing, hence essential in~$\scl{M}$ by Remark \ref{rem - essential does not bound}. 
    In $\pi_2(\scl{M},p)$ we have $[S_0]+[S_1]=[f(w)]$. As $[f(w)]$ is nontrivial, at least one of $[S_0]$ or $[S_1]$ is nontrivial. 
    Thus we can pick $i \in \{0,1\}$ such that the sphere $S_i$ is still essential in~$\scl{M}$ and therefore reducing in $M$.
    
    The map $f$ is now replaced with a homotopic map $f'$
    that agrees with $f$ on all vertices except $w$, but sends $w$ to $S_i$ instead of $f(w)$.
    (The original map $f$ sends all vertices in $\mathrm{link}_K(w)$ to spheres disjoint from $S_i$ since $C$ was chosen to be innermost,
    and after a small translation $S_i$ is also disjoint from $f(w)$, so $f'$ sends all of $\mathrm{star}_K(w)$ to the star of $f(w)$ in $\mathcal{S}(M)$.
    This is what allows us to find this homotopy --
    see \cite[Figure 4]{BoydBregman22} for a schematic.)
    Since $f'(w) = S_i$ no longer intersects $f(v)$ at $C$, for this new map $f'$ the number of intersections $N$ has been reduced by one, so by induction we can reduce $N$ to zero i.e.~until $f$ sends all of $K$ into the star of $f(v)$. 
    Hence, $f$ is nullhomotopic and~$\mathcal{S}(M)$ is contractible as required.
\end{proof}

\subsection{Systems of spheres}
We now consider posets of \emph{systems} of spheres which satisfy certain properties.
\begin{defn}
    A \emph{cut system} for a $3$-manifold $M$ is a (possibly empty) collection of spheres $\Sigma \in \umb(\amalg_k S^2, M)$, for~$k\in \mathbb{N}$,  such that every sphere in~$\Sigma$ is reducing.

    We let $\cut(M)$ denote the space of cut systems, topologised as a subspace
    \[
        \cut(M) \subset \coprod_{k \ge 0} \umb(\amalg_k S^2, M).
    \]
    This is a topological poset with respect to the partial order
    $\subseteq$ defined by inclusion.
\end{defn}

Note that $\Sigma$ is allowed to be empty, which occurs when~$k=0$. The empty cut system is minimal with respect to the partial order. If $\Sigma\neq \emptyset \in \cut(M)$ then each component of $\scl{M \ca \Sigma}$ has strictly fewer prime factors than $\scl{M}$. Let $\cutNE(M)$ be the subspace of~$\cut(M)$ consisting of non-empty cut systems.

\begin{defn}
    Let $\cut_\bullet(M)$ and $\cutNE_\bullet(M)$ denote the simplicial spaces obtained as the nerve of the poset $(\cut(M), \subseteq)$ and $(\cutNE(M), \subseteq)$ respectively. 
    Concretely, the space of $p$-simplices in each case is
    \[
        \cut_p(M) = \{\Sigma_0 \subseteq \dots \subseteq \Sigma_p \in \cut(M) \} \subset \cut(M)^{\times p+1}, \text{ and }
    \]\[
        \cutNE_p(M) = \{\emptyset\neq\Sigma_0 \subseteq \dots \subseteq \Sigma_p \in \cut(M) \} \subset \cutNE(M)^{\times p+1}.
    \]
\end{defn}

Since~$\Sigma=\emptyset$ is minimal with respect to the partial order, it acts as a cone point for the realisation of~$\cut(M)$, and therefore~$\|\cut_\bullet(M)\|\simeq *$. 
The set of reducing spheres in $M$ is empty if and only if $\scl{M}$ is irreducible, so in this case $\cut(M)=\{\emptyset\}$, and $\cutNE(M)$ is empty.

\begin{prop}\label{prop:ess-red-contratible}
    Suppose~$\scl{M}$ is not irreducible. 
    Then the classifying space of the poset of non-empty cut systems is contractible, i.e.~$\|\cutNE_\bullet(M)\| \simeq *$.
\end{prop}

\begin{proof}
    $\cutNE_\bullet(M)$ is non-empty, since we assumed that~$\scl{M}$ was not irreducible.
    Following the framework of Section \ref{section - discretization}, let~$X$ be the subspace of $\umb(S^2, M)$ consisting of reducing spheres, and~$\perp$ be the symmetric and anti-reflexive relation on~$X$ given by~$S_0\perp S_1$ iff~$S_0\cap S_1=\emptyset$. 
    Then $\cutNE_\bullet(M)$ is the nerve of the associated topological poset $C(X,\perp)$ from Definition~\ref{defn - C/S(X, perp)}. 
    We now apply the adaptation of the discretisation technique of Galatius and Randal-Williams \cite{GalatiusRandalWilliams} given in~Proposition \ref{prop - poset discretization}. First note that $S(X,\perp)$ (Definition \ref{defn - C/S(X, perp)}) is contractible since it is precisely the complex $\mathcal{S}(M)$, which is contractible by Theorem~\ref{thm: S(M)-contractible}. Hence condition~$(1)$ of Proposition~\ref{prop - poset discretization} is satisfied for all~$n$. 
    For condition~$(2)$ we fix a cut system $\Sigma \in \cutNE_\bullet(M)$ (which will function as~$\ul{y}$ in the proposition) and consider the simplicial complex~$S(X,\perp)^{\perp \Sigma}$. This complex is contractible -- for any map~$f:S^k \to S(X,\perp)^{\perp \Sigma}$, there exists a sphere~$S$ disjoint from all vertices in $\im (f)$ (for example, take $S$ to be a parallel copy of a sphere in~$\Sigma$). $S$ is therefore a vertex in $S(X,\perp)^{\perp \Sigma}$ which acts as a cone point for~$\im (f)$ and thus $S(X,\perp)^{\perp \Sigma}\simeq *$. 
    Therefore $(2)$ is also satisfied for all $n$. Applying Proposition~\ref{prop - poset discretization} gives that~$\|\cutNE_\bullet(M)\|$ is $n$-connected for all $n$, and thus contractible.
\end{proof}

\subsection{Separating systems}
Now we are ready to restrict to cut systems that are \emph{separating}.

\begin{defn}
    A cut system $\Sigma$ on a $3$-manifold $M$ is called \emph{separating}
    if each component of $\scl{M \ca \Sigma}$ is irreducible. Let $\sep(M) \subseteq \cut(M)$ denote the subspace of separating cut systems.
    
    For $q \ge 1$ we say that a cut system $\Sigma$ is \emph{$q$-separating}
    if each component of $\scl{M \ca \Sigma}$ has at most $q$ prime factors. 
    Let $\qsep{q}(M) \subseteq \cut(M)$ denote the subspace of $q$-separating cut systems.
    
    Then $\sep(M)$ and $\qsep{q}(M)$ are posets with respect to $\subseteq$ and we denote their nerves by $\sep_\bullet(M) \subset \cut_\bullet(M)$ and $\qsep{q}_\bullet(M) \subset \cut_\bullet(M)$.
\end{defn}

    There are inclusions 
    \[
        \sep(M) \subseteq \qsep{1}(M) \subset \dots 
        \subset \qsep{q-1}(M) \subset \qsep{q}(M) = \cut(M)
    \]
    where $q$ is the number of prime factors of $M$.
    Note that the first inclusion is strict exactly when~$\scl{M}$ has~$S^1\times S^2$ as a prime factor, since~$S^1\times S^2$ is the only closed 3-manifold which is prime but not irreducible. Since every reducing sphere cuts~$M$ into components with strictly fewer prime factors than~$M$, it follows that $\qsep{q-1}(M)=\cutNE(M)$.

\begin{ex}
    When~$\scl{M}=S^1\times S^2$,~$\sep(M)=\cutNE(M)$ and $\qsep{1}(M) = \cut(M)$.
\end{ex}

We record the following for later use.
\begin{lem}\label{lem:face-map-is-finite-covering}
    The map
    \[
        \varphi_n\colon\sep_n(M) \longrightarrow \sep_0(M) = \sep(M),
        \qquad
        (\Sigma_0 \subseteq \dots \subseteq \Sigma_n) \mapsto \Sigma_n
    \]
    is a finite covering.
\end{lem}
\begin{proof}
    The map is a locally trivial fiber bundle because it is $\Diff(M)$-equivariant and $\sep(M)$ is locally $\Diff(M)$-retractile, so apply \cref{lem:locally-retractile}(1).
    Therefore it will suffice to check that the fiber $\varphi_n^{-1}(\Sigma)$ is a finite set. (It then automatically has the discrete topology because $\sep_n(M)$ is Hausdorff.)
    Indeed, this fiber is the set of sequences of cut systems
    $\Upsilon_0 \subseteq \dots \subseteq \Upsilon_n = \Sigma$
    such that each $\Upsilon_i$ is a separating system.
    As there are only finitely cut systems contained in $\Sigma$, this set is finite.
\end{proof}

\begin{prop}\label{prop:sep-p-contractible}
    The spaces $\|\qsep{p}_\bullet(M)\|$ are contractible for all $p \ge 1$ and all $M$.
\end{prop}
\begin{proof} 
    Let $q$ be the number of prime factors in $M$. 
    For $p \ge q$ the poset $\qsep{p}(M)$ contains the empty cut system, which is initial and hence acts as a cone point for $\|\qsep{p}_\bullet(M)\|$.
    We show that $\|\qsep{p}_\bullet(M)\| \to \|\qsep{p+1}_\bullet(M)\|$ is a homotopy equivalence for all $p\ge 1$.
    
    We invoke Quillen's Theorem A for topological categories \cite[Theorem 4.8]{EbertRandalWilliams} as recalled in Corollary \ref{cor:QuillenA},  noting that \ref{cor:QuillenA}(3) is true since connected components correspond to isotopy classes of cut systems, and \ref{cor:QuillenA}(2) is true
    because the map
    \[
        \cut_1(M) = \{\Sigma_0 \subseteq \Sigma_1 \in \cut(M)\} \to \cut(M),
        \qquad
        (\Sigma_0 \subseteq \Sigma_1) \mapsto \Sigma_0
    \]
    is a $\Diff(M)$-equivariant map into a $\Diff(M)$-locally retractile space and hence a fiber bundle by \cref{lem:locally-retractile}(1).
    
    It therefore suffices to show \ref{cor:QuillenA}(1): the slice category of the inclusion functor at every object $\Upsilon \in \qsep{p+1}(M)$ has a contractible classifying space.
    In the case at hand this slice category is the poset
    \[
        \qsep{p}(M)_{\Upsilon \subseteq} = \{ \Sigma \in \qsep{p}(M) \;|\; \Upsilon \subseteq \Sigma\}.
    \]
    Any separating system $\Sigma$ containing $\Upsilon$ decomposes as the disjoint union $\Sigma = \Upsilon \sqcup \Upsilon^c$ and it will suffice to keep track of $\Upsilon^c$.
    Note that while $\Upsilon^c$ is $p$-separating in $M\ca \Upsilon$, it may contain spheres which are not reducing in $M \ca \Upsilon$ (for example, they may be boundary parallel, if they were parallel to a sphere in~$\Upsilon$).  
    
    We introduce the following modification of~$\qsep{p}(M)$ to account for this.  Let $N$ be a 3-manifold and let $B\subset \partial N$ be a union of 2-sphere boundary components.  Let $N_{\scl{B}}$ denote be the manifold obtained by filling each 2-sphere of $B$ with a disk. Define $C^{\mathrm{pf}\leq p}(N,B)$ to be collections of disjointly embedded spheres in $N$ that are essential in~$N_{\scl{B}}$ and~$p$-separating, and let $C_\bullet^{\mathrm{pf}\leq p}(N,B)$ be the nerve of the associated poset. 
    In particular, $\qsep{p}(N) = C^{\mathrm{pf} \le p}(N, \partial N) \subseteq C^{\mathrm{pf}\leq p}(N,B)$ but the essential spheres appearing in collections in $C^{\mathrm{pf}\leq p}(N, B)$ are not necessarily reducing.

    Let $U$ be a connected component of $M\ca \Upsilon$. Each boundary sphere of $M$ lies in a unique such component $U$. Let $B_U$ be the union of all 2-spheres in $\partial U$ that come from $\partial M$. Then there is a homeomorphism
    \[
        \qsep{p}(M)_{\Upsilon \subseteq}
        \cong \prod_{U \in \pi_0(M\ca \Upsilon)} C^{\mathrm{pf}\leq p}(U,B_U).
    \]
    
    Since~$\Upsilon \in \qsep{p+1}(M)$, the components of $\scl{M \ca \Upsilon}$ have at most~$p+1$ prime factors. Therefore, for each $U$ there are two possibilities:
    \begin{itemize}
        \item if $\scl{U}$ has at most $p$ prime factors, then $C^{\mathrm{pf}\leq p}({U},B_U)$ contains the empty sphere system, which is an initial object and hence $\|C_\bullet^{\mathrm{pf}\leq p}({U}, B_U)\|$ is contractible.
        \item If $\scl{U}$ has exactly $p+1$ prime factors, then every~$\Sigma \in C^{\mathrm{pf}\leq p}({U},B_U)$ must contain at least one reducing sphere. 
        Therefore the inclusion $\cutNE({U})\subseteq C^{\mathrm{pf}\leq p}({U},B_U)$ admits a retraction $F: C^{\mathrm{pf}\leq p}(U,B_U)\to \cutNE({U})$ that forgets spheres which are not reducing. 
        By Lemma~\ref{lem:poset-natural-transformation=homotopy} this is homotopic to the identity map, so $\cutNE_\bullet({U})$ is a deformation retract of $C^{\mathrm{pf}\leq p}({U},B_U)$. Since $\|\cutNE_\bullet({U})\|$ is contractible by Proposition~\ref{prop:ess-red-contratible}, it follows that $\|C_\bullet^{\mathrm{pf}\leq p}({U},B_U)\|$ is contractible.
    \end{itemize}
    Therefore, $\|C_\bullet^{\mathrm{pf}\leq p}({U},B_U)\|$ is always contractible.
    The fat geometric realisation of a product of simplicial spaces is equivalent to the product of the fat geometric realisations \cite[Theorem 7.1]{EbertRandalWilliams}, so $\|\qsep{p}_{\bullet}(M)_{\Upsilon\subseteq}\|$ is contractible, and applying \cref{cor:QuillenA} completes the proof.
\end{proof}

We now extend this one step further.

\begin{prop}\label{prop:sep-contractible}
    The space $\|\sep_\bullet(M)\|$ is contractible for all $M$.
\end{prop}
\begin{proof} 
    Paralleling \cref{prop:sep-p-contractible}, we will show that $\|\sep_\bullet(M)\| \to \|\qsep{1}_\bullet(M)\|$ is weak equivalence by 
    proving that for every $\Upsilon \in \qsep{1}(M)$ the poset
    \[
        \sep(M)_{\Upsilon \subseteq} = \{ \Sigma \in \sep(M) \;|\; \Upsilon \subseteq \Sigma\}.
    \]
    has a contractible classifying space.

    We define $C_\bullet^{\sep}(N, B) \subset C_\bullet^{\mathrm{pf}\le 1}(N, B)$ as the subspace of those $\Sigma \subset N$ such that every component of $\scl{N\ca \Sigma}$ is irreducible.
    In particular $\sep(N)\subseteq C^{\sep}(N,B)$ but the latter contains systems with essential spheres that are not necessarily reducing.
    As before we have a product
    \[
        \sep(M)_{\Upsilon \subseteq} 
        \cong \prod_{U \in \pi_0(M\ca \Upsilon)} C^{\sep}({U},B_U).
    \]
    Now, since~$\Upsilon \in \qsep{1}(M)$, each ${U}$ has at most one prime factor.
    There are two cases:
    \begin{itemize}
        \item if $U$ admits no reducing spheres then $\scl{U}$ must be irreducible and $C^{\sep}({U},B_U)$ contains the empty sphere system, which is an initial object and hence $\|C^{\sep}_\bullet({U},B_U)\|$ is contractible.
        \item If $U$ admits a reducing sphere, then because $\scl{U}$ is prime, $U$ must be $S^1 \times S^2 \setminus \amalg_k \interior{D}^3$.
        As in the proof of  Proposition \ref{prop:sep-p-contractible}, $C^{\sep}({U},B_U)$ has a deformation retraction to $\sep(U)$.
        But in this case $\sep(U)=\cutNE(U)$, and thus $\|C^{\sep}_\bullet({U},B_U)\|$ is contractible by Proposition \ref{prop:ess-red-contratible}. 
    \end{itemize}
    As in Proposition \ref{prop:sep-p-contractible}, assembling these results together completes the proof.
\end{proof}

\subsection{Removing parallel spheres}\label{subsection:removing parallel spheres}

We now turn our attention to the presence of unnecessary `parallel' spheres in our separating systems, and show these can be removed.
\begin{defn}\label{defn:parallel}
    We say that two spheres $S_0, S_1 \subset M$ are \emph{parallel} if they bound an embedded $S^2 \times I$,
    \emph{i.e.}~if there is an embedding $e: S^2 \times [0,1] \hookrightarrow M$ with $e(S^2 \times \{i\}) = S_i$.

    Let $\sepNP(M) \subset \sep(M)$ denote the sub-poset consisting of those separating systems $\Sigma \in \sep(M)$ such that no two spheres in $\Sigma$ are parallel.
\end{defn}

Note that there is a maximum possible number $N$ of spheres in any given separating system $\Sigma \in \sepNP(M)$. (In \cref{lem:bounding-graphs} we give an upper bound for $N$ in terms of the prime decomposition of $M$.)
Therefore all chains of length~$> N$ in the poset must include an equality, \emph{i.e.}~when~$i>N$, all simplices in~$\sepNP_i(M)$ are degenerate. This is our motivation for passing to the space $\sepNP(M)$.

It is useful to have the following alternative characterisation of parallel spheres.

\begin{lem}\label{lem:parallel-spheres-bound-cylinder}
    Two essential spheres $S_0, S_1 \subset M$ are parallel if and only if they are disjoint and (unparametrised) isotopic.
\end{lem}
\begin{proof} 
    By \cite[Lemma 2.1]{Laudenbach73}, if two disjoint essential spheres are homotopic, they bound an $h$-cobordism in $M$.
    As a consequence of the Poincar\'e conjecture any such $h$-cobordism is trivial, \emph{i.e.}~an embedded $S^2 \times [0,1]$, so the spheres are parallel.
\end{proof}

The goal of this section is to prove the following result.

\begin{prop}\label{prop - removing parallel spheres is an equivalence}
    For $M \not\cong S^1 \times S^2$ the inclusion $\sepNP(M) \subset \sep(M)$ induces an equivalence 
    \[
    \|\sepNP_\bullet(M)\| \xrightarrow{\ \simeq\ } \|\sep_\bullet(M)\|.
    \]
\end{prop}

In order to show that $\|\sep_\bullet^\nparallel(M)\| \to \|\sep_\bullet(M)\|$ is an equivalence we will need to choose a way of ordering parallel spheres. 
To do so we make a non-canonical choice and use this to define some auxiliary posets encoding the ordering. To show this choice is well-defined we include the following lemma.

We would like to thank the Copenhagen ``homotopy theory problem seminar'' who worked out the following proof and showed more generally that if $M$ is a closed oriented $n$-manifold ($n\neq 4$) such that there is an $S^{n-1} \subset M$ that admits an orientation-reversing isotopy, then $M$ is an exotic $n$-sphere.

\begin{lem}\label{lem:reversing-coorientation}
    Essential spheres do not admit orientation-reversing self-isotopies:
    if $i,j\colon S^2 \hookrightarrow M$ are isotopic embeddings of an essential sphere such that $j(S^2) = i(S^2)$, then $j^{-1} \circ i \in \Diff(S^2)$ is orientation preserving.
\end{lem}
\begin{proof}
    Suppose such an orientation-reversing isotopy exists. 
    By isotopy extension, we can find an isotopy from $\textrm{Id}_M$ to a diffeomorphism $h$ reversing the coorientation on $S$. Since $M$ is orientable, $h$ must reverse the orientation on $S$ as well. 
    We first claim $S$ separates $M$. If not, there is a simple closed curve $\gamma\subset M$ that intersects $S$ in exactly one point. Hence $\gamma$ and $S$ are duals for the intersection pairing on $M$, so $[S]\in H_2(M)$ generates a nontrivial free factor. But then $h_*[S]=-[S]$, contradicting $h\simeq\textrm{Id}_M$.
 
    So we may assume $S$ separates $M$. Since $h$ reverses the coorientation on $S$, it exchanges the two (necessarily homeomorphic) components $M_0$ and $M_1$ of $M\ca S$.  If neither  $M_0$ nor $M_1$ are simply connected, then $\pi_1(M_0)\cong \pi_1(M_1)\cong G$ for some nontrivial $G$. Moreover, $\pi_1(M)\cong G*G$ and $h_*$ exchanges the two factors. Since $h_*$ preserves the kernel of the natural homomorphism $G*G\rightarrow G\times G$, it induces an automorphism of $G\times G$, which cannot be not inner, since no nontrivial element of $G\times\{1\}$ is conjugate to an element of $\{1\}\times G$. Thus, $h_*\colon \pi_1(M) \to \pi_1(M)$ is not inner, again contradicting $h\simeq \textrm{Id}_M$. 
    Finally, if $M_0$ and $M_1$ are simply connected, they are each homeomorphic to $S^3\setminus (\coprod_k\interior{D}^3)$, and since $S$ is essential, we must have $k\geq 2$. 
    But then $h$ would have to permute the boundary spheres, and it cannot because it is isotopic to the identity.
\end{proof}

    For each essential sphere $S \subset M$ we now choose a coorientation and use it to define (locally around $S$) a left side and a right side of $S$.
    To be precise, we make this choice once for one representative in each isotopy class, and then define the coorientation on any essential sphere by transporting it along an isotopy from the representative in its isotopy class.
    By \cref{lem:reversing-coorientation} self-isotopies of essential spheres cannot reverse coorientation, so this procedure is well-defined.

    For two parallel spheres $S_0, S_1 \subset M$ we say that $S_0$ is \emph{to the left of $S_1$} if the part of $M \ca (S_0 \cup S_1)$ that is right of $S_0$ and left of $S_1$ is diffeomorphic to $S^2 \times [0,1]$.
    Conversely we define \emph{to the right of}.
Note that if $S_0$ is both left of and right of $S_1$, then $M \cong S^2 \times S^1$. From now on we will assume that $M \not\cong S^2 \times S^1$. 

Within an isotopy class of reducing spheres, we will need to understand the subspace of all spheres to the left of a fixed sphere $S_0$. Equivalently, we can regard this as a space of spheres parallel to a fixed boundary component of $M\ca S_0$. For this, we have the following lemma. 

\begin{lem}\label{lem:boundary-parallel-spheres}
    Let $M$ be a compact $3$-manifold and $S_0 \subset \partial M$ a preferred boundary sphere.
    Then the space of unparametrised embeddings 
    \[
         P(S_0; M) := \{ S \subset \interior{M} \;|\; S \text{ is parallel to } S_0 \}
         \subset 
         \umb(S^2, \interior{M})
    \]
    is contractible.
\end{lem}
\begin{proof}
    The $\Diff_\partial(M)$-equivariant map
    \[
        p\colon \emb_{S_0\times\{0\}}( S_0 \times [0,1], M)  \to
        \umb(S_0, \interior{M}), \qquad
        i \mapsto i(S_0 \times \{1\})
    \]
    is a fiber bundle because $\umb(S_0, M)$ is $\Diff_\partial(M)$-locally retractile by \cref{cor:umb-retractile}.
    Moreover, by the definition of parallel, the image of this map is exactly the connected component $P(S_0; M) \subset \umb(S_0, \interior{M})$.
    Let $S_1 \in P(S_0; M)$ be some unparametrised sphere parallel to $S_0$.
    The fiber $p$ at some unparametrised sphere $S_1 \in P(S_0; M)$, which is parallel to $S_0$,
    is the space of those embeddings
    $\iota\colon S_0 \times [0,1] \hookrightarrow M$
    that are the identity on $S_0 \times \{0\}$ and send $S_0 \times \{1\}$ to $S_1$.
    The image of such an embedding is always exactly the connected component $U \subset M \ca S_1$ that lies between $S_0$ and $S_1$. $U$ is necessarily diffeomorphic to $S^2 \times [0,1]$.
    Specifying another embedding $\iota' \in p^{-1}(S_1)$ is hence equivalent to specifying a diffeomorphism $S_0 \times [0,1] \cong U$ that is the identity on $S_0 \times \{0\}$.
    Therefore we have a fiber sequence
    \[
        \Diff_{S_0\times\{0\}}(S_0 \times [0,1]) \to \emb_{S_0 \times \{0\}}(S_0 \times [0, 1], M) \to P(S_0; M).
    \]
    The left term is the group of pseudo-isotopies of $S^2$, which is contractible by Hatcher's proof of the Smale conjecture \cite{Hatcher}.
    The middle term is the space of collars of $S_0$ in $M$, which is also contractible by \cite{Cerf}, as recalled in \cref{thm:contractible-collars}.
    Therefore $P(S_0; M)$ is contractible.
\end{proof}

\begin{defn}\label{defn:L}
    For a sphere system $\Sigma$ we let $L(\Sigma) \subseteq \Sigma$ denote the sphere system that selects the left-most sphere from each isotopy class represented in $\Sigma$.
    Then the topological poset $\sepL(M)$ is defined as follows.
    Its underlying space is the subspace 
    $\sepL(M) \subseteq \sep(M)$
    of those separating systems where each isotopy class contains either no sphere or at least two spheres.
    The relation is defined as $\Sigma \le \Sigma'$ if and only if $\Sigma \subseteq \Sigma'$ and $L(\Sigma) \subseteq L(\Sigma')$.
    Define $\sepLtwo(M) \subset \sepL(M)$ as the full subposet on those separating systems where each isotopy class has either no sphere or exactly two spheres.
\end{defn}

By construction, $L$ defines a map of posets $
    L\colon \sepL(M) \longrightarrow \sepNP(M),
$
and we can also define 
$
    R\colon \sepL(M) \longrightarrow \sep(M),$ via $
    R(\Sigma) := \Sigma \setminus L(\Sigma).
$
Thus the map $R$ removes the left-most sphere from each isotopy class.
This is still a separating system as we assumed that every non-empty isotopy class of $\Sigma \in \sepL(M)$ contains at least two spheres. The following shows that forgetting the left-most sphere does not change the homotopy type of the space of separating systems.
\begin{lem}\label{lem:forget-left-most-sphere}
    Assume $M \not\cong S^2 \times S^1$.
    Then the map
    \[
        R\colon \sepL(M) \longrightarrow \sep(M)
    \]
    that discards the left-most sphere from each isotopy class is a trivial Serre fibration.
    As a consequence, the maps
    $R\colon \|\sepL_\bullet(M)\| \to \|\sep_\bullet(M)\|$ and
    $L\colon \|\sep^{\rm L,2}_\bullet(M)\| \to \|\sepNP_\bullet(M)\|$
    are weak equivalences.
\end{lem}
\begin{proof}
    By \cref{cor:umb-retractile}, $\sep(M)$ is a locally $\Diff(M)$-retractile space and so by \cref{lem:locally-retractile}(1) any $\Diff(M)$-equivariant map into it is a locally trivial fibration.

    For an isotopy class $a \in \pi_0\umb(S^2, M)$, let $\sep^a(M) \subset \sep(M)$ denote the subspace of those separating systems $\Sigma$ such that the number of spheres of $\Sigma$ that are in the isotopy class $a$ is at least $2$. This only depends on the isotopy class of $\Sigma$, hence $\sep^a(M)$ is a union of path components of $\sep(M)$.
    Then we can define a map
    \[
        R_a\colon \sep^a(M) \to \sep(M) 
    \]
    that forgets the left-most sphere from the isotopy class $a$. 
    This map is again a fiber bundle because $\sep(M)$ is locally $\Diff(M)$-retractile. Each path component of $\sepL(M)$ belongs to $\sep^a(M)$ for only finitely many $a$, hence restricted to each path component the map $R$ is equal to a finite composite of maps of the form $R_a$. Thus, it will suffice to show that each $R_a$ is a trivial Serre fibration.

    We therefore have to show that the fiber $R_a^{-1}(\Sigma)$ is contractible at any $\Sigma \in \sep(M)$ in the image of $R_a$.
    A point in this fiber is of the form $\Sigma \sqcup S$ for some $S \in \umb(S^2, M)$ such that $S$ is disjoint from $\Sigma$, $S$ is in the isotopy class $a$, and (within $\Sigma \sqcup S$) $S$ is the left-most sphere in its isotopy class.
    Let $S_0$ be the left-most sphere in $\Sigma$ that is in the isotopy class $a$.
    Then we can equivalently write the fiber as the subspace  of those unparametrised spheres that are parallel to $S_0 \subset \partial (M\ca \Sigma)$.
    (Here there are two ways of thinking of $S_0$ as a sphere in the boundary of $M\ca \Sigma$. We choose the one that corresponds to the left side of $S_0 \subset M$.)
    \[
        R_a^{-1}(\Sigma) \cong 
        \{S \in \umb(S^2, (M \ca \Sigma)^\circ) \;|\; S \text{ parallel to } S_0\}
    \]
    Contractibility now follows from \cref{lem:boundary-parallel-spheres}.

    To prove the final claim, we first consider the commutative square
    \[\begin{tikzcd}
        \sepL_n(M) \ar[r, "R_n"] \ar[d] & \sep_n(M) \ar[d] \\
        \sepL_0(M) \ar[r, "R_0"] & \sep_0(M)
    \end{tikzcd}\]
    where the vertical maps send $(\Sigma_0 \subseteq \dots \subseteq \Sigma_n)$ to $\Sigma_n$.
    These vertical maps are finite coverings by \cref{lem:face-map-is-finite-covering} and the horizontal map induces a bijection on fibers, so the square is a pullback and hence $R_n$ is also a trivial Serre fibration. 
    Fat geometric realisation preserves weak equivalences \cite[Theorem 2.2]{EbertRandalWilliams} so $\|R\|$ is a weak equivalence.
    The proof for $L$ is entirely analogous.
\end{proof}

Using these intermediate spaces we can now prove \cref{prop - removing parallel spheres is an equivalence}.

\begin{proof}[Proof of \cref{prop - removing parallel spheres is an equivalence}]
Consider the following (non-commutative!) diagram of maps of topological posets.
\[
\begin{tikzcd}[column sep = large]
    \sepLtwo(M) \ar[r, "I"] \ar[d, "L"'] & \sepL(M) \ar[dl, "L"'] \ar[d, "R"] \\
    \sepNP(M) \ar[r, "I"] & \sep(M) 
\end{tikzcd}
\]
Here the maps labelled $I$ are inclusions and $L$ and $R$ are as defined above.
The top triangle in this diagram commutes.
Consider the three maps
\[
    L, F, R \colon \sepL(M) \longrightarrow \sep(M)
\]
where $F$ is the forgetful map (\emph{i.e.}~the inclusion on underlying spaces).
For every $\Sigma$ we have inclusions 
$L(\Sigma) \subset F(\Sigma) = \Sigma \supset R(\Sigma)$.
Therefore Lemma \ref{lem:poset-natural-transformation=homotopy} provides us with homotopies $|L| \simeq |F| \simeq |R|$, so the following diagram commutes up to homotopy.
\[
\begin{tikzcd}[column sep = large]
    {\|\sepLtwo_\bullet(M)\|} \ar[r, "{\|I\|}"] \ar[d, "{\|L\|}"'] & 
    {\|\sepL_\bullet(M)\|} \ar[dl, "{\|L\|}"'] \ar[d, "{\|R\|}"] \\
    {\|\sepNP_\bullet(M)\|} \ar[r, "{\|I\|}"] & 
    {\|\sep_\bullet(M)\|}
\end{tikzcd}
\]
By \cref{lem:forget-left-most-sphere} the two vertical maps in the above square are weak equivalences.
Now it follows from the $2$-out-of-$6$-property that all maps in the diagram are weak equivalences.
(Alternatively, one can apply the functor $\pi_k(-)$ to the diagram and then show that the diagonal map is both injective and surjective.)
\end{proof}

By combining \cref{prop - removing parallel spheres is an equivalence} and \cref{prop:sep-contractible} we obtain the following theorem, which we will use in the proof of our main theorem in Section~\ref{section: finiteness of BDiff}.

\begin{thm}\label{thm: sepNP contractible}
    Suppose $M\not\cong S^1 \times S^2$. Then $\|\sepNP_\bullet(M)\|\simeq *$.
\end{thm}

\subsection{The connected sum of two irreducible manifolds}\label{sec:M connected sum of two irreducible}

Theorem~\ref{thm: sepNP contractible} allows us to reprove the following result of Hatcher's, which is given in the final remark of~\cite{Hatcher1981}, and also discussed after the main theorem in a revision of that same article \cite{Hatcher1981revised}.

\begin{thm}[{\cite{Hatcher1981}}]\label{thm:two-prime-pieces}
    Let $M\cong M_1 \# M_2$, where the $M_i$ are irreducible and not~$S^3$. Let~$S\subset M$ be a reducing 2-sphere. Then the inclusion $\Diff(M,S)\hookrightarrow \Diff(M)$ is an equivalence.
\end{thm}
    \begin{proof}
        Note that since the $M_i$ are irreducible, the prime decomposition of $M$ contains no $S^1 \times S^2$ factors and therefore any essential sphere is separating and induces a connected sum decomposition of $M$ into two pieces.
        We claim that any two disjoint reducing spheres in $M$ must be parallel.
        Indeed if $S_0$ and $S_1$ are disjoint embedded spheres, then they induce a connected sum decomposition of $M$ into three pieces: $M\cong N_1 \#_{S_0} N_2 \#_{S_1} N_3$.
        Since $M$ has two prime factors, it follows that one of the $N_i$ is homeomorphic to $S^3$, and it cannot be $N_1$ or $N_3$, since the $S_i$ are reducing.
        Therefore $N_2$ is homeomorphic to $S^3$, and $N_2\setminus (S_0\sqcup S_1)\cong S^2 \times I$ determines an embedding with boundary components $S_0$ and $S_1$, so $S_0$ and $S_1$ are parallel.
        Consequently, any $\Sigma \in \sepNP(M)$ consists of a single sphere, in other words $\sepNP(M) \subset \umb(S^2, \interior{M})$ is the space of essential spheres. 
        In particular, there are no non-degenerate chains in~$(\sepNP(M), \subseteq)$ and it follows that~$\|\sepNP_\bullet(M)\|\simeq \sepNP(M)$, so Theorem~\ref{thm: sepNP contractible} implies that $\sepNP(M)\simeq *$. 

        By the discussion following \cref{cor:umb-retractile} we have the fibration sequence
        \[
            \Diff(M,S)\hookrightarrow \Diff(M) \rightarrow \umb(S^2,\interior{M}).
        \]
        The fibration lands in the subspace $\sepNP(M) \subset \umb(S^2, \interior{M})$ of essential spheres (because the action of $\Diff(M)$ preserves these) which we have shown is contractible.
        Therefore the inclusion of $\Diff(M, S) \to \Diff(M)$ must be an equivalence as claimed.
    \end{proof}

 The above theorem is a powerful computational tool, because when the $M_i$ are distinct we can describe $\Diff(M, S)$ using the homotopy fiber sequence
       \[
            \Diff_\partial(M_1 \setminus \interior{D}^3) \longrightarrow \Diff(M, S) \longrightarrow \Diff(M_2\setminus \interior{D}^3),
       \]
which can be used to describe $\Diff(M)$ in terms of $\Diff(M_i)$ for the irreducible pieces $M_i$.

\section{Homotopy finiteness via the (strong) generalised Smale conjecture} \label{section:strongsmale}

In Section \ref{section: finiteness of BDiff} we will use contractibility of $\|\sepNP_\bullet(M)\|$ to prove Theorem \ref{thm:Kontsevich-finiteness-conjecture} via an inductive argument. 
The base case of our argument is homotopy finiteness for $\BDiff_\partial(M)$ where $\scl{M}$ is an irreducible manifold. 
In the case that $M$ itself is irreducible (\emph{i.e.}~has no spherical boundary components) this is a theorem of Hatcher--McCullough \cite{HatcherMcCullough}, as recalled in \cref{thm: HatcherMcCullough finiteness}. We will show in \cref{thm:prime-finiteness} that the general case can be reduced to either Hatcher--McCullough's result or the following theorem, which is equivalent to the case in which $M$ has one fixed spherical boundary component. 

\begin{thm}\label{thm:sect4main}
    Let $M$ be an irreducible 3-manifold with either empty or incompressible toroidal boundary, and let $D^3\subset \interior{M}$ be an embedded disk. 
    Then $\BDiff_{D^3}(M)$ has the homotopy type of a finite CW complex.
\end{thm}

The proof of this theorem will occupy Sections \ref{section:strongsmale} and \ref{section:decomposing irreducible}.  For its proof, we exploit the  JSJ and geometric decompositions of $M$, which we recall now. Let $M$ be a prime orientable 3-manifold with empty or incompressible toroidal boundary.
Then there exists a minimal collection $T_{\JSJ}$ of embedded, incompressible tori and annuli such that each of the components of $M\setminus T_{\JSJ}$ are either Seifert-fibered or atoroidal, and this collection is unique up to isotopy. (Recall that an irreducible manifold is called \emph{atoroidal} if every essential immersed torus is homotopic to the boundary.)
This decomposition of $M$ is known as the \emph{JSJ decomposition}, having been discovered independently by Jaco--Shalen \cite{JacoShalen1976} and Johannson \cite{Johannson1979b}. Moreover, since $\partial M$ consists only of tori,  no annuli appear in this decomposition (see e.g.~\cite{Jaco1980}). 
 
\begin{defn}\label{defn:TrivialJSJ} A prime 3-manifold with empty or incompressible toroidal boundary has \emph{trivial JSJ decomposition} if it is either Seifert fibered or atoroidal (\emph{i.e.}~if $T_{\JSJ}$ is empty).  
\end{defn}

On the other hand, Thurston's geometrisation conjecture states that  there exists a collection $T_G$ of incompressible 2-sided tori and 1-sided Klein bottles embedded in $M$ such that each component of $M\setminus T_G$ admits a finite-volume Riemannian metric locally isometric to one of 8 model geometries: $\bbS^3$, $\bbE^3$, $\bbH^3$, $\bbS^2\times \bbR$, $\bbH^2\times \bbR$, $\Nil$, $\Sol$, or $\widetilde{\PSL}_2(\bbR)$, unless $M=T^2\times I$ or $\Kx$, the twisted $I$-bundle over the Klein bottle. 
This collection $T_G$ is unique up to isotopy.
Except for $\bbH^3$ and $\Sol$, the remaining geometries are all Seifert-fibered. Thurston proved his conjecture in the case of Haken 3-manifolds, showing in particular that every closed  atoroidal Haken 3-manifold admits a finite-volume metric based on $\bbH^3$ \cite{ThurstonAtoroidal}. (Recall an orientable manifold is \emph{Haken} if it is irreducible and contains an orientable incompressible surface.) The full conjecture was proven by Perelman \cite{Perelman:2002-1, Perelman:2003-1, Perelman:2003-2} (see also \cite{CaoZhu,BessieresEtAl, MorganTian}).

\begin{defn}\label{defn:Geometric}
   A compact, connected 3-manifold $M$ is called \emph{geometric} if $\interior{M}$ admits a complete, finite-volume metric locally modeled on one of the 8 Thurston geometries. 
\end{defn}

Geometric 3-manifolds necessarily have incompressible boundary, which is diffeomorphic to a (possibly empty) disjoint union of tori (see, \emph{e.g.}~\cite[\S 2]{Bonahon-Geometric-Structures}). The decomposition of $M$ along $T_G$ is called the \emph{geometric decomposition} of $M$, and is nearly the same as the JSJ decomposition, except that $T_G$ may contain 1-sided Klein bottles, and that closed manifolds supporting $\Sol$-geometry have a trivial geometric decomposition but a nontrivial JSJ decomposition. 

If $M$ is an orientable $T^2$-bundle over $S^1$ that is not Seifert-fibered, then $M$ is called an \emph{Anosov bundle} and admits $\Sol$-geometry. Although not all $\Sol$-manifolds are of this form, such bundles will be a special case of Theorem \ref{thm:sect4main} that we will not be able to treat in the same way we treat other manifolds with nontrivial JSJ decomposition. Conversely, if $M$ is orientable, irreducible and has trivial JSJ-decomposition then $M$ is geometric unless $M=T^2\times I$ or $M=K\tilde{\times}I$. 
Thus when decomposing manifolds, we cannot completely exclude the case of manifolds with non-trivial JSJ decomposition nor the case of non-geometric manifolds.  However, if $M$ has nontrivial JSJ decomposition, and is not an Anosov bundle, the components of $M|T_{\JSJ}$ all have nonempty toroidal boundary and thus are either Haken Seifert-fibered (and not $T^2\times I$) or hyperbolic.

The flowchart in \cref{fig:flowchartirreducible} explains how the proof of \cref{thm:sect4main} breaks up into cases depending on the geometry of $M$. Each leaf contains a reference to the result in which we prove $\BDiff_{D^3}(M)$ is homotopy finite for such $M$. Arrows correspond to straightforward case distinctions, except in one case where a proposition is required to reduce to the case where $M$ has trivial JSJ decomposition. In order to prove \cref{thm:sect4main} when $M$ has nontrivial JSJ decomposition and is not an Anosov bundle, we will need a stronger finiteness condition on $\BDiff$ for components of $M\ca T_{\JSJ}$. We introduce this condition, called \emph{hereditary finiteness} (\cref{defn:hereditarily-finite}) in the next section and prove it holds in the (fiber-rigid) Haken Seifert-fibered and hyperbolic cases, as only these can arise in $M\ca T_{\JSJ}$. The leaves for which we prove hereditary finiteness are coloured in blue in \cref{fig:flowchartirreducible}.

\begin{figure}
\resizebox{\textwidth}{!}{
\begin{tikzpicture}[node distance=7mm and 7mm]
	\tikzstyle{box} = [rectangle, draw, text width=9em, text centered, rounded corners, minimum height=3em]
    \tikzstyle{leafbox} = [rectangle, draw, text width=9em, text centered, rounded corners, minimum height=3em, fill=orange!10]
    \tikzstyle{leafboxHF} = [rectangle, draw, text width=9em, text centered, rounded corners, minimum height=3em, fill=teal!10]
	\tikzstyle{thmbox} = [text width=8em, text centered, minimum height=2em]
	\tikzstyle{decoration}=[draw=none, anchor=west, font=\small]
	\tikzstyle{arrow}=[->]
	 
	\node (m minus ball) [rectangle, draw, text width=13em, text centered, rounded corners, minimum height=3em] {\cref{thm:sect4main}:\\ $\BDiff_{D^3}(M)$ homotopy finite} ;
	
	\node (anasov) [leafbox, below right = and -10mm of m minus ball] { Anosov $T^2$ \\bundle over $S^1$ \\ \cref{thm:Anosov-T2-bundles}};
	
	\node (thm jsj) [thmbox,below left = and -10mm of m minus ball] {\cref{prop:cut-along-tori}} ;
	\node (trivial jsj) [box,below left= and -10mm of thm jsj] {Trivial JSJ \\decomposition};
        \node (sf) [box, below right = and -10mm of trivial jsj] { Seifert-fibered};
	\node (sf non haken) [leafbox,below left = and -10mm of sf] {Seifert-fibered, non-Haken \\ \cref{cor:NonHaken-and-Hyperbolic}} ;
	\node (hyperbolic) [leafboxHF, below left= and -10mm of trivial jsj] { Hyperbolic \\ \cref{cor:NonHaken-and-Hyperbolic}} ;
	
	\node (sf haken) [box, below right= and -10mm of sf] {Seifert-fibered Haken} ;
	\node (flexible) [box, below left= 10mm and 0 mm  of sf haken] {Flexible};
	\node (t3) [leafbox, below  = 10mm and 0 mm of flexible] { $T^3$ \\ \cref{prop:T3-Finiteness}};
	\node (mhw) [leafbox, below right = 10mm and 5 mm  of flexible] { $M_{HW}$\\ \cref{prop:Hantzsche-Wendt}};
 \node (M is t2 x i) [leafbox, below left = 10mm and 5 mm of flexible] {$T^2 \times I$ \\ \cref{prop: finiteness sphere or torus x I}};
	\node (fiber rigid) [box, below right = 10mm and 10mm of sf haken] {Fiber-rigid};
        \node (singular) [leafboxHF,below right = 10 mm and 5mm of fiber rigid ] {Singular \\ \cref{prop:Haken-SF-Finiteness}};
        \node (non singular) [leafboxHF,  below = 10mm of fiber rigid ] {Non-singular\\ \cref{lem:Haken-SF-Nonsingular}};
        
	\draw[]  (m minus ball) -- (thm jsj);
	\draw[arrow]  (m minus ball) -- (anasov);
	\draw [arrow] (flexible) -- (M is t2 x i);
	\draw[arrow]  (thm jsj) -- (trivial jsj);
	\draw[arrow]  (trivial jsj) -- (hyperbolic);
	\draw[arrow]  (trivial jsj) -- (sf);
	\draw[arrow]  (sf) -- (sf non haken);
	\draw[arrow]  (sf) -- (sf haken);
	\draw[arrow] (sf haken) -- (flexible);	
	\draw[arrow] (sf haken) -- (fiber rigid);	
	\draw[arrow] (flexible) -- (t3);	
	\draw[arrow] (flexible) -- (mhw);	
	\draw[arrow] (fiber rigid) -- (singular);
	\draw[arrow] (non singular) -- (singular);
	\draw[arrow] (fiber rigid) -- (non singular);
\end{tikzpicture}
}
    \caption{Flowchart of the proof of \cref{thm:sect4main}}
    \label{fig:flowchartirreducible}
\end{figure}
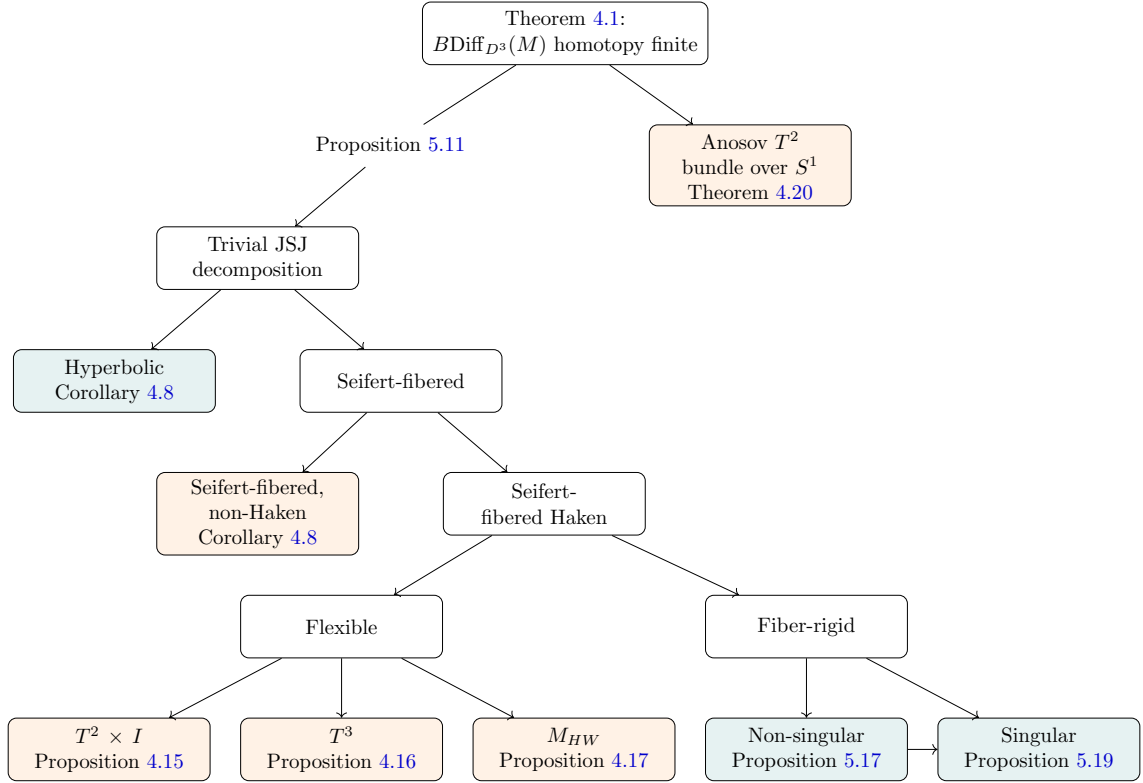

Recall that we have a fiber sequence 
\begin{equation*}
    \Diff_{D^3}(M)\longrightarrow \Diff(M)\longrightarrow \emb(D^3,\interior{M})
\end{equation*}
whence $\BDiff_{D^3}(M)$ is equivalent to a connected component of $\emb(D^3,\interior{M})\hq \Diff(M)$.
The quotient $\emb(D^3,\interior{M})\hq \Diff(M)$ is connected if $M$ admits an orientation reversing diffeomorphism, and otherwise it has two connected components.

Let $\Fr(M)$ denote the frame bundle of $M$. 
Since $M$ is parallelisable, $\Fr(M)\cong M\times\GL_3(\bbR)$. 
We will also consider the oriented orthonormal frame bundle $\Fr^\perp(M)$, which is homotopy equivalent to $\Fr(M)$.  
Again by parallelisability, $\Fr^\perp(M)\cong M\times \Or(3)$. 
Since $\emb(D^3,\interior{M})\simeq \Fr(M)$, we obtain
that $\BDiff_{D^3}(M)$ is equivalent to a connected component of $\Fr(M) \hq \Diff(M)$.
The overall strategy for proving \cref{thm:sect4main} when $M$ has trivial JSJ decomposition will be to exploit the geometry on $M$ to find a suitable subgroup $G< \Diff(M)$ such that $G\hookrightarrow \Diff(M)$ is a homotopy equivalence and $\Fr(M)\hq G$ is homotopy finite.
In this case the basepoint component of $\Fr(M) \hq G$, which we can also write as $\Fr^+(M)\hq G^+$ by restricting to oriented frames and oriented diffeomorphisms, 
will be a model for $\BDiff_{D^3}(M)$. 
The simplest case will be when $G$ is a Lie group acting freely on $\Fr(M)$. For example, if $G=\Isom(M)$ for some Riemannian metric on $M$, we learnt the idea of the following proof from Sam Nariman.

\begin{lem}\label{lem:Riemannian-Finite-Quotient}
    Suppose $(M,g)$ is a Riemannian 3-manifold such that the inclusion $\Isom(M)\rightarrow \Diff(M) $ is a homotopy equivalence.  Then $\BDiff_{D^3}(M)$ is homotopy finite.  
\end{lem}
\begin{proof} 
    As argued above $\BDiff_{D^3}(M)$ is a connected component of 
\[\Fr(M)\hq \Diff(M)\simeq\Fr(M)\hq \Isom(M)\simeq \Fr^\perp(M)\hq\Isom(M).
\]
The isometry group $\Isom(M)$ is a compact Lie group acting freely on the compact manifold $\Fr^\perp(M)$, so $\Fr^\perp(M)\hq\Isom(M) \simeq \Fr^\perp(M)/\Isom(M)$ and this quotient is a compact manifold \cite[Theorem 7.10]{Lee}, and hence is homotopy finite.
\end{proof}

Recall from the introduction that we refer to the statement that $\Isom(M)\hookrightarrow \Diff(M)$ is a homotopy equivalence as the \emph{strong} generalised Smale conjecture, and to the statement that $\Isom_0(M)\hookrightarrow \Diff_0(M)$ is a homotopy equivalence as the \emph{weak} generalised Smale conjecture. 
For the rest of \cref{section:strongsmale}, we will prove all cases of \cref{thm:sect4main} (corresponding to leaves in the flowchart in \cref{fig:flowchartirreducible}) where there exists a Lie group $G< \Diff(M)$ acting freely on $\Fr(M)$ and such that $G\hookrightarrow \Diff(M)$ is a homotopy equivalence. Except for $T^3$ and $T^2\times I$, the strong generalised Smale conjecture will hold, allowing us to take $G=\Isom(M)$ and apply \cref{lem:Riemannian-Finite-Quotient}.

\subsection{The generalised Smale conjecture}\label{sec:GenSmaleFiniteness}

The strong form of the generalised Smale conjecture is false, even among irreducible geometric 3-manifolds. However, for irreducible geometric manifolds, the weak form always holds.

\begin{thm}[Weak generalised Smale conjecture]\label{thm:Gen-Smale} Let $(M,g)$ be an irreducible geometric 3-manifold.
Then the inclusion $\Isom_0(M)\hookrightarrow \Diff_0(M)$ is a homotopy equivalence.
\end{thm}

The proof of Theorem \ref{thm:Gen-Smale} combines the results of several authors spanning decades. For Haken geometric 3-manifolds, it follows from results of Hatcher \cite{Hatcher76} and, independently, Ivanov \cite{Ivanov76} on spaces of incompressible surfaces. Hatcher and Ivanov's work covers all Haken cases and thus all cases when $M$ has non-empty boundary, all cases with $\Sol$-geometry, and many of the hyperbolic and Seifert-fibered cases except possibly when the base-orbifold is homeomorphic to $S^2$. 

For $M=S^3$, Theorem \ref{thm:Gen-Smale} follows from Hatcher \cite{Hatcher}. 
It was proven for quotients of $S^3$ containing one-sided incompressible Klein bottles (prismatic and quaternionic) and lens spaces other than $\RP^3$ using topological techniques \cite{Ivanov79,Ivanov82,HKMR12}. Bamler--Kleiner \cite{BK23, BK19} used Ricci flow techniques to prove the remaining cases (tetrahedral, octahedral, icosahedral, and $\RP^3$) and gave a unified proof for all spherical 3-manifolds. In particular, their approach offers a new proof of the Smale conjecture.

In the $\bbH^3$-case, Gabai \cite{Gabai01} proved that $\Diff_0(M)$ is contractible, and the techniques of \cite{BK23} also reprove this case. Since $(M,g)$ is negatively curved, $\Isom(M)$ is discrete (see {e.g.}~\cite[Corollary 5.4]{KobayashiNomizu}) so Theorem \ref{thm:Gen-Smale} follows. The non-Haken cases of $\bbH^2\times\bbR$- and $\widetilde{\PSL}_2(\bbR)$-geometry were proven by McCullough and Soma \cite{McCulloughSoma13}. Recently, Bamler--Kleiner \cite{BamlerKleiner2023} proved the last remaining non-spherical case, namely closed manifolds admitting $\Nil$-geometry and whose base-orbifold is homeomorphic to $S^2$.

We note that in each case, the actual homotopy type of $\Diff_0$ is known. For irreducible 3-manifolds that are spherical, see Tables 1 and 2 in \cite{HKMR12}. For non-spherical geometric 3-manifolds, the homotopy type of $\Diff_0$ is simpler to describe, and since we will refer to it later we state it here. 
\begin{cor}[Haken \cite{Hatcher76,Ivanov76,Ivanov79}, Non-Haken \cite{McCulloughSoma13,BamlerKleiner2023}]\label{cor:Diff0-list}
    Let $M$ be an irreducible geometric 3-manifold that is not covered by $S^3$.
    \begin{enumerate}
        \item If $M$ is not Seifert-fibered, then $\Diff_0(M)\simeq *$.
        \item If $M$ is Seifert-fibered  then $\displaystyle\Diff_0(M)\simeq \prod_{b}\SO(2)$, where $b$ is the rank of the center $Z(\pi_1(M))$.
    \end{enumerate}
\end{cor}

Theorem \ref{thm:Gen-Smale} tells us that $\Isom(M)$ and $\Diff(M)$ can only differ with respect to $\pi_0$, the group of path components. The group $\pi_0\Diff(M)$ is known as the \emph{mapping class group} of $M$. We briefly survey some results on mapping class groups of geometric manifolds. The computation of the mapping class group for $S^3$ is a consequence of work of Cerf \cite{Cerf-Gamma4}, for $S^1\times S^2$ it is due to Gluck \cite{Gluck1962}, and for lens spaces it is due to Bonahon \cite{Bonahon1983}, although some cases were known \cite{Rubinstein1979,Ivanov79,Ivanov82,BirmanRubinstein1984}. More generally, the results for all spherical 3-manifolds are surveyed in \cite{HKMR12}, and include work of Asano \cite{Asano1978}, Rubinstein \cite{Rubinstein1978, Rubinstein1979}, Boileau--Otal \cite{BoileauOtal1986}, and Birman--Rubinstein \cite{BirmanRubinstein1984}.  Johannson \cite{Johannson1979} showed that Haken 3-manifolds with incompressible boundary that are atoroidal and anannular have finite mapping class groups, and McCullough \cite{McCullough1991} showed more generally that mapping class groups of Haken 3-manifolds are finitely presented, and investigated other finiteness properties.

In many cases, the strong form of the generalised Smale conjecture is also known to hold.
\begin{thm}\label{thm:Isom-Equivalence}
    Suppose $M$ is an irreducible geometric 3-manifold that is either hyperbolic or non-Haken Seifert-fibered.  Then the inclusion $\Isom(M)\hookrightarrow \Diff(M)$ is a homotopy equivalence. 
\end{thm}
The non-Haken Seifert-fibered manifolds in the above theorem include all those with $\bbS^3$-geometry completed by \cite{HKMR12,BK19,BK23}, see references therein for a comprehensive history.  All other irreducible non-Haken Seifert-fibered cases carry either the geometry of  $\bbH^2\times \bbR$, $\widetilde{\PSL}_2(\bbR)$, or $\Nil$, and fiber over the 2-sphere with 3 cone points. Theorem \ref{thm:Isom-Equivalence} was proven for the $\bbH^2\times \bbR$ and $\widetilde{\PSL}_2(\bbR)$ cases in \cite{McCulloughSoma13}, and for the $\Nil$ case in \cite{BamlerKleiner2023}.
(In fact, Theorem \ref{thm:Isom-Equivalence} also holds when $M$ fibers over $S^2$ with 3 cone points and carries $\bbE^3$-geometry \cite{McCulloughSoma13}, although these are Haken.)
The hyperbolic case follows from Mostow--Prasad rigidity (see e.g.~\cite{Gromov1979}) and the work of Gabai \cite{Gabai01} and Gabai--Meyerhoff--Thurston \cite{GMT03}.
Combining Theorem \ref{thm:Isom-Equivalence} with Lemma \ref{lem:Riemannian-Finite-Quotient} we obtain the following corollary.

\begin{cor}\label{cor:NonHaken-and-Hyperbolic}
    Let $M$ be an irreducible geometric 3-manifold that is either hyperbolic or non-Haken Seifert-fibered. Then $\BDiff_{D^3}(M)$ is homotopy finite.
\end{cor}

In Proposition \ref{prop:Hantzsche-Wendt} and Theorem \ref{thm:Anosov-T2-bundles} we show that the generalised Smale conjecture also holds for the Hantzsche-Wendt manifold $M_{\rm HW}$ (see \cref{sec:Flexible} below) and $T^2$-bundles over the circle that have Anosov monodromy.

\subsection{Haken Seifert-fibered manifolds -- overview}\label{sec:HakenSF}

The strong form of the generalised Smale conjecture often does not hold for Haken Seifert-fibered 3-manifolds, essentially because the mapping class group of the base surface may be infinite. Our proof of Theorem \ref{thm:sect4main} for these manifolds will rely on understanding the space of all Seifert fiberings.  

We say two Seifert fiberings of $M$ are  \emph{equivalent} if there is a diffeomorphism from $M$ to itself taking the fibers of one to the fibers of the other. 
Given a fixed Seifert fibering $\varphi$ of $M$, let $\Diff^f(M;[\varphi])$ denote the subgroup of \emph{fiber-preserving diffeomorphisms} of $\varphi$, \emph{i.e.}~those diffeomorphisms which map fibers of $\varphi$ diffeomorphically to fibers. On the other hand, in \cref{section:decomposing irreducible} we also study the subgroup of \emph{vertical diffeomorphisms} $\Diff^v(M;[\varphi])$, consisting of diffeomorphisms $f \in \Diff(M)$ that satisfy $\varphi \circ f = \varphi$, \emph{i.e.}~they preserve each fiber setwise.

If $\varphi$ and $\varphi'$ are equivalent, then $\Diff^f(M;[\varphi])$ and $\Diff^f(M;[\varphi'])$ are conjugate subgroups of $\Diff(M)$. Following \cite[Definition 3.13]{HKMR12}, we define the \emph{space of Seifert fiberings isomorphic to $\varphi$} to be the coset space $\SF(M;[\varphi])=\Diff(M)/\Diff^f(M,[\varphi])$. Theorems 3.12 and 3.14 of \cite{HKMR12} state that the quotient map $\Diff(M)\rightarrow \SF(M;[\varphi])$ is a locally trivial fiber bundle and that the components of $\SF(M;[\varphi])$ are contractible.

For many Haken Seifert-fibered 3-manifolds, $\SF(M;[\varphi])$ is also connected, which is to say that any diffeomorphism of $M$ is isotopic to a fiber-preserving diffeomorphism. 
Indeed, by results of Waldhausen \cite[Theorem 10.1]{Waldhausen67} (see also Jaco \cite[Theorem VI.18]{Jaco1980} for the case where $M$ has boundary), there are only finitely many orientable Haken Seifert-fibered 3-manifolds for which there does not exist a unique Seifert fibering up to isotopy. These six manifolds, which we call \emph{exceptional}, are:
\begin{itemize}
\item $S^1\times D^2$,
\item $T^3$, the 3-torus,
\item $T^2\times I$, the trivial $I$-bundle over the 2-torus,
\item $M_{HW}$, the Hantzsche-Wendt manifold (first studied in \cite{HantzscheWendt}),
\item $\Kx$, the orientable $I$-bundle over the Klein bottle $K$, and 
\item $D(\Kx)$ the double of $\Kx$.
\end{itemize}
We remark that although $S^1\times D^2$ and its fiberings will play a major role in \cref{section:decomposing irreducible}, we do not need to prove \cref{thm:sect4main} for $S^1\times D^2$, as it has compressible boundary. Indeed, $S^1\times D^2$ is the only Haken Seifert-fibered 3-manifold with compressible boundary.

\begin{defn}\label{def:fiberrigid-flexible}
    A Haken Seifert-fibered manifold $M$ is \emph{fiber-rigid} if it admits a Seifert fibering $\varphi\colon M\rightarrow S$ such that the inclusion $\Diff^f(M;[\varphi])\hookrightarrow \Diff(M)$ is a homotopy equivalence.
    (Or equivalently such that $\SF(M; [\varphi])$ is contractible.)
    Otherwise we say $M$ is \emph{flexible}.
\end{defn}

\begin{thm}\label{thm:Unique-fibering}
    Let $M$ be orientable and Haken Seifert-fibered. If $M$ is not $S^1 \times D^2$, $T^3$, $T^2\times I$, or $M_{HW}$, then $M$ is fiber-rigid.
\end{thm}
\begin{proof} 
    By Theorem 3.15 of \cite{HKMR12} all orientable Haken Seifert-fibered manifolds that are not exceptional are fiber-rigid.
    In fact, their theorem says that for non-exceptional cases $\SF(M;[\varphi])$ is contractible for any choice of $\varphi$.
    Thus it remains to prove that for $\Kx$ and $D(\Kx)$ there exists a fibering $\varphi$ for which $\Diff^f(M; [\varphi]) \to \Diff(M)$ is surjective on $\pi_0$.
    
If $M=\Kx$, then $M\simeq K$ so $\pi_1(M)\cong \pi :=\langle a,t\mid tat^{-1}=a^{-1}\rangle$. 
By Waldhausen's theorem \cite{Waldhausen68}, $\pi_0\Diff(M)$ is isomorphic to the subgroup of $\Out(\pi)$ preserving $\pi_1(\partial M)$.
In the case at hand $\pi_1(\partial M) \cong \bbZ^2$ 
is the subgroup $\langle a, t^2\rangle \leq \pi$, which is characteristic, so $\pi_0\Diff(M)\cong \Out(\pi)\cong \Z/2\times\Z/2$. 
The generators of $\Out(\pi)$ are $\sigma_1\colon t\mapsto t^{-1}$, $a\mapsto a$ and $\sigma_2\colon t\mapsto ta$, $a\mapsto a$.
There are two distinct Seifert fiberings $\varphi_1$, $\varphi_2$ of $M$ corresponding to the two maximal normal infinite cyclic subgroups of $\pi$, namely $\langle a\rangle$ and $\langle t^2\rangle$, respectively. 
The base surface of $\varphi_1$ is M\"obius strip and there are no singular fibers. 
The base orbifold of $\varphi_2$ is a disk with two cone points of order 2. 
These two cone points correspond to the conjugacy classes of $t$ and $ta$ respectively. 
Both $\sigma_1$ and $\sigma_2$ can be realised by diffeomorphisms of $\Kx$ that preserve both fiberings: $\sigma_1$ is a reflection that reverses the orientation of each fiber of $\varphi_2$, preserving each setwise, while $\sigma_2$ rotates each fibers of $\varphi_1$ by $180^\circ$ and permutes the fibers of $\varphi_2$, exchanging the two singular fibers (see Figure \ref{fig:Twisted Klein}). Thus $\pi_0\Diff^f(M;[\varphi_i])$ surjects onto $\pi_0\Diff(M)$ for $i=1,2$. Using either $\varphi_1$ or $\varphi_2$ proves the theorem in this case.

Finally, consider $M=D(\Kx)$, the double of $\Kx$.  Then $M$ supports a flat metric, and by \cite[Theorem 1]{CharlapVasquez}, the group of affine automorphisms $\Aff(M)$ fits into a short exact sequence\[1\rightarrow \Aff_0(M)\rightarrow \Aff(M)\rightarrow \Aff(M)/\Aff_0(M)\rightarrow 1,\] where $\Aff_0(M)= \Isom_0(M)$, and $\Aff(M)/\Aff_0(M)\cong \Out(\pi_1(M))$.  Therefore by \cref{thm:Gen-Smale} and Waldhausen's theorem \cite{Waldhausen68}, the inclusion $\Aff(M)\rightarrow \Diff(M)$ is a homotopy equivalence.
On the other hand, the quotient map  $\varphi_0\colon M\rightarrow M/\Aff_0(M)$ is a Seifert fibering with base surface a sphere with 4 cone points of order 2. Since $\Aff_0(M)$ is normal in $\Aff(M)$, every element of $\Out(\pi_1(M))\cong \Aff(M)/\Aff_0(M)$ is realised by a fiber-preserving diffeomorphism of $M$. Thus, $\pi_0\Diff(M;[\varphi_0])\rightarrow \Out(\pi_1(M))\cong \pi_0\Diff(M)$ is surjective. \end{proof}
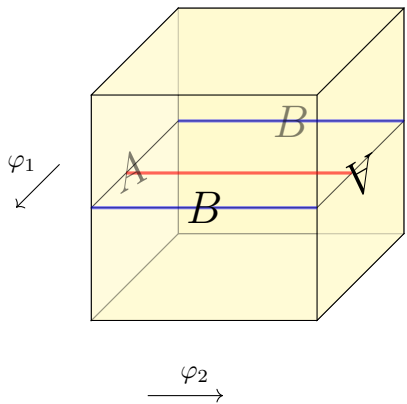
\begin{figure}[ht]
    \centering
    \begin{tikzpicture}
        \coordinate (O) at (0,0,0);
\coordinate (A) at (0,3,0);
\coordinate (B) at (0,3,3);
\coordinate (C) at (0,0,3);
\coordinate (D) at (3,0,0);
\coordinate (E) at (3,3,0);
\coordinate (F) at (3,3,3);
\coordinate (G) at (3,0,3);

\draw[fill=yellow!20] (O) -- (C) -- (G) -- (D) -- cycle;
\draw[fill=yellow!30] (O) -- (A) -- (E) -- (D) -- cycle;
\draw[fill=yellow!10] (O) -- (A) -- (B) -- (C) -- cycle;
\draw[fill=yellow!20,opacity=0.8] (D) -- (E) -- (F) -- (G) -- cycle;
\draw[fill=yellow!20,opacity=0.6] (C) -- (B) -- (F) -- (G) -- cycle;
\draw[fill=yellow!20,opacity=0.8] (A) -- (B) -- (F) -- (E) -- cycle;

\draw[opacity=.7] (0,1.5,0)--(3,1.5,0)--(3,1.5,3)--(0,1.5,3)--(0,1.5,0);
\draw[very thick,red,opacity=.6](0,1.5,1.8)--(3,1.5,1.8);
\draw[blue,very thick,opacity=.5](0,1.5,0)--(3,1.5,0);

\draw[blue, very thick,opacity=.5](0,1.5,3)--(3,1.5,3);

\draw[->] (-1,1.5,1.5)--(-1,1.5,3);
\node at(-1,2,2.8){$\varphi_1$};

\draw[->] (.75,-1,3)--(1.75,-1,3);
\node at(1.3,-.8,2.8){$\varphi_2$};

\draw[opacity=.5](0,1.5,1.8) node{\rotslant{30}{10}{\huge $A$}};
\draw(3.2,1.5,1.8) node{\rotslant{210}{10}{\huge $A$}};

\draw(1.5,1.5,3) node{\rotslant{0}{0}{\huge $B$}};
\draw[ opacity=.5](1.5,1.5,0) node{\rotslant{0}{0}{\huge $B$}};

    \end{tikzpicture}
    \caption{$\Kx$ is formed by identifying sides of a cube as indicated above.  The vertical axis is the $I$-direction while the central square is the core  $K$. Projecting onto the $B$-face is the nonsingular fibering $\varphi_1$ over the M\"obius band, and the projection $\varphi_2$ onto the $A$-face has two singular fibers of order 2, shaded in blue and red respectively. Reflection in the plane perpendicular to the red line induces $\sigma_1$, while the rotation of $K$ that exchanges the blue and red curves extends to $\Kx$ as $\sigma_2$. The top and bottom faces of the cube glue up to form a torus $T^2$, and the projection onto the core $K$ is the orientation double cover $T^2\to K$.}
    \label{fig:Twisted Klein}
\end{figure}

\begin{rem}\label{rem:most-non-singular-are-rigid}
    For a Seifert fibering $\varphi\colon M \to S$ where $M$ has non-trivial boundary and $M \not\cong T^2 \times I$ and $M \not\cong S^1 \times D^2$, we have $\Diff^f(M; [\varphi]) \simeq \Diff(M)$.
    This is because any irreducible $M$ with boundary is Haken in which case the only exceptional cases we need to consider are $T^2 \times I$ and $K \tilde{\times} I$, but for the latter we just checked that both the non-singular fibering and singular fibering is rigid.
\end{rem}

For the sake of completeness, in the next subsection we will show 
that the converse to Theorem \ref{thm:Unique-fibering} is also true.  That is, the four excluded manifolds $S^1\times D^2$, $T^3$, $T^2\times I$ and $M_{HW}$ are the only flexible Haken Seifert-fibered 3-manifolds. As alluded to above, the proof of Theorem \ref{thm:sect4main} for $T^3$, $T^2\times I$ and $M_{HW}$ fits into the same schema as those addressed in \cref{sec:GenSmaleFiniteness}, while we deal with the fiber-rigid case in \cref{section:decomposing irreducible}.

\subsection{Flexible Seifert-fibered manifolds}\label{sec:Flexible} 

We now deal with $T^3$, $T^2\times I$,  and $M_{HW}$. First, we recall basic facts about their fiberings and show that all three of these are flexible. For $M=T^2\times I$ or $M=T^3$, there infinitely many Seifert fiberings, each one corresponding to an infinite cyclic free factor of $\pi_1(M).$ For $T^3$, the base surface is always $T^2$ and there are no singular fibers, while for $T^2\times I$ the base surface is always $S^1\times I$, again without singular fibers. In each case, any two fiberings are equivalent, since the mapping class group acts transitively on free factors of $\pi_1(M)$.  
Thus, for any fibering $\varphi$ the components of $\SF(T^3;[\varphi])$ (respectively $\SF(T^2\times I;[\varphi])$) are in one-to-one correspondence with the cosets $\GL_3(\Z)/\GL_3^{\langle e_1\rangle}(\Z)$ (respectively $\GL_2(\Z)/\GL_2^{\langle e_1\rangle}(\Z)$
), where $\GL_n^{\langle e_1\rangle}(\Z)$ denotes the $\GL_n(\Z)$-stabiliser of the subgroup $\langle e_1:=(1,0,\ldots,0)\rangle\leq\Z^n$. In particular, both $T^3$ and $T^2\times I$ are flexible.

To see that $M_{HW}$ is flexible, we first recall its construction (see \cite[Section 3.9]{HKMR12} for more details). The $I$-bundle $K\tilde{\times} I$ can be thought of as the mapping cylinder of the oriented double cover $T^2\rightarrow K$, and as such its boundary is $T^2$ (see \cref{fig:Twisted Klein}). Take two copies of $\Kx$ and let $T_1$ and $T_2$ be the two boundary tori. While the double $D(K\tilde{\times} I)$  identifies longitude to longitude and meridian to meridian,  $M_{HW}$ is formed by identifying the meridian of $T_1$ with the longitude of $T_2$, and vice versa.  There are two Seifert fiberings of $M_{HW}$, one where the meridians of $T_1$ are fibers, and the other where the longitudes of $T_1$ are fibers.  However, there is a diffeomorphism of $M_{HW}$ which exchanges the two copies of $\Kx$ and simultaneously swaps the meridian and longitude of $T_1$.  This is not fiber-preserving, so $M_{HW}$ is flexible.

To prove $\BDiff_{D^3}(M)$ is homotopy finite for these three manifolds, we exhibit a subgroup $G\leq \Diff(M)$ such that the inclusion $G \to \Diff(M)$ is a homotopy equivalence and $G$ acts freely on $\Fr(M)$ with homotopy finite quotient. 
When $M$ is $M_{HW}$, we will exhibit a metric on $M$ such that $G= \Isom(M)$, and apply Lemma \ref{lem:Riemannian-Finite-Quotient}. When $M$ is $T^2\times I$ or $T^3$, $\pi_0\Diff(M)$ is infinite, so no such metric exists; nevertheless, $\pi_0\Diff(M)$ does lift to a group of diffeomorphisms.

Let us first proceed with the proof for $T^2\times I$. Note that the results in this case will also be needed in \cref{section:decomposing irreducible}. We begin with some finiteness results for diffeomorphisms of surfaces.

\begin{lem}\label{lem:Surface-Fix-Disk-finite}
    Let $S$ be a surface (not necessarily orientable).  Then $\BDiff_{D^2}(S)$ is homotopy finite.
\end{lem}
\begin{proof}
    Let $S'=S\setminus \interior{D}^2$, and let $b$ be the boundary component of $S'$ corresponding to $\partial D^2$. Then  $\BDiff_{D^2}(S)\simeq \BDiff_b(S')$. By \cite[Proposition 6]{Gramain}, the connected components of $\BDiff_b(S')$ are contractible, so it suffices to show that $\pi_0\Diff_b(S')$ has finite classifying space.  This is proved in  \cite[Lemma 1.2]{HatcherMcCullough}.
\end{proof}

Given a closed, compact manifold $M$,  the \emph{pseudo-isotopy group} is defined to be \[\PI(M) := \Diff_{M\times \{0\}}(M\times I).\] For surfaces, Hatcher proved the following.

\begin{thm}[$S\neq S^2$ \cite{Hatcher76}, $S=S^2$ \cite{Hatcher}]\label{thm:Pseudo-Isotopy-Surfaces}Let $S$ be an orientable surface.  Then $\PI(S)$ is contractible.
\end{thm}

\begin{cor}\label{cor:I-bundle-Diff}
    Let $S$ be an orientable surface.  Then $\Diff(S\times I)\simeq \Diff(S)\times \Z/2$.
\end{cor}
\begin{proof} 
Let $S_0=S\times \{0\}$ and $S_1=S\times \{1\}$. $\Diff(S\times I)$ may permute the boundary components, so we have principal short exact sequence 
\[1\rightarrow \Diff(S\times I,S_0)\rightarrow \Diff(S\times I)\rightarrow \Z/2\rightarrow 1.\] 
Letting $r\colon S\times I\rightarrow S\times I$ be the reflection $r(x,t)=(x,1-t)$ we see this sequence splits. On the other hand, by restricting to $S_0$ we get another (split) principal short exact sequence 
\[1\rightarrow \PI(S\times I)\rightarrow \Diff(S\times I, S_0)\rightarrow \Diff(S_0)\rightarrow 1.\]  
with splitting given by 
$f\mapsto f \times \id_I$.
By Theorem \ref{thm:Pseudo-Isotopy-Surfaces}, the fiber is contractible, so $\Diff(S\times I,S_0)\simeq \Diff(S_0)=\Diff(S)$. 
Finally, $r$ commutes with $f\times \id_I$, so the inclusion $\Diff(S)\times \Z/2\rightarrow \Diff(S\times I)$ is a homotopy equivalence. 
\end{proof}

We apply the previous corollary to the case where $S=T^2$.

\begin{prop}\label{prop: finiteness sphere or torus x I}
    Let $D^3\subset (T^2\times I)^\circ$. Then $\BDiff_{D^3}(T^2\times I)$ is homotopy finite.
\end{prop}
\begin{proof}
    We know that $\Diff(T^2 \times [0,1]) \simeq \bbZ/2 \times \Diff(T^2)$, so we can write
    \[
        \BDiff_{D^3}(T^2 \times [0,1]) \simeq \Fr(T^2 \times [0,1]) \hq \Diff(T^2 \times [0,1])
        \simeq \Fr(T^2 \times [0,1]) \hq (\bbZ/2 \times \Diff(T^2)).
    \]
    There also is a homotopy fiber sequence
    \[  
        \Fr(T^2) \times \bbZ/2 \longrightarrow
        \Fr(T^2 \times [0,1]) \longrightarrow
        \RP^2
    \]
    where the fibration is defined by recording the line spanned by last vector in the frame.
    The first map is $\Diff(T^2) \times \bbZ/2$-equivariant and the second map $\Diff(T^2) \times \bbZ/2$-invariant, so by \cref{lem:fiberseq/fiberseq} we get a homotopy fiber sequence
    \[  
        (\Fr(T^2) \times \bbZ/2)\hq (\Diff(T^2) \times \bbZ/2) \longrightarrow
        \Fr(T^2 \times [0,1])\hq (\Diff(T^2) \times \bbZ/2) \longrightarrow
        \RP^2.
    \]
    The base is a finite CW complex and the fiber is homotopy equivalent to $\Fr(T^2) \hq \Diff(T^2) \simeq \BDiff_{D^2}(T^2)$, which is homotopy finite by Lemma \ref{lem:Surface-Fix-Disk-finite}.
\end{proof}

To prove Theorem \ref{thm:sect4main} for $T^3$, we will use the fact that $\pi_0\Diff(T^3)\cong \GL_3(\Z)$ acts on $T^3$ by diffeomorphisms.

\begin{prop}\label{prop:T3-Finiteness}
    Let $D^3\subset \interior{T}^3$ be an embedded disk.  Then $\BDiff_{D^3}(T^3)$ is homotopy finite.
\end{prop}
\begin{proof}
    Since $T^3$ is closed and Haken, Waldhausen's theorem \cite{Waldhausen68} implies that the natural map $\pi_0\Diff(T^3)\rightarrow \Out(\pi_1(T^3))\cong \GL_3(\Z)$ is an isomorphism. The standard action of $\GL_3(\bbZ)$ on $\bbR^3$ preserves the $\bbZ^3$ lattice, hence descends to a group of diffeomorphisms of $T^3=\bbR^3/\bbZ^3$.
    By  \cref{thm:Gen-Smale}, the group of translations of $T^3$, which is equal to $\Isom_0(T^3)$ and which we denote by $\bbT^3$, is homotopy equivalent to $\Diff_0(T^3)$. Hence the inclusion $\bbT^3\rtimes\GL_3(\bbZ)\hookrightarrow \Diff(T^3)$ is a homotopy equivalence.
    
    Let $\Fr(T^3)$ be the space of frames on $T^3$. There is a diffeomorphism $\Fr(T^3)\cong T^3\times \GL_3(\bbR)$, and using the standard flat metric on $T^3$, this diffeomorphism is compatible with the action of $\bbT^3\rtimes \GL_3(\Z)$: the translation action is trivial on the $\GL_3(\bbR)$ factor and freely transitive on $T^3$, while $\GL_3(\Z)$ acts on $\GL_3(\bbR)$ via the standard right action. In particular, the action is free with quotient $\GL_3(\bbR)/\GL_3(\Z)$. 
    By the preceding discussion, we have 
    \begin{align*}\BDiff_{D^3}(T^3)\simeq \Fr(T^3)\hq \Diff(T^3)&\simeq (T^3\times \GL_3(\bbR))\hq (\bbT^3\rtimes \GL_3(\Z))\\ &\cong\GL_3(\bbR)/\GL_3(\Z).\end{align*}
    To see that the latter is homotopy finite, recall that there is a fiber bundle $\Or(3)\hookrightarrow\GL_3(\bbR)\rightarrow \mathcal{Q}_3$, where $\mathcal{Q}_3$ is the space of positive definite $3\times3$ matrices over $\bbR$. By a result of Soul\'e \cite{Soule78}, $\mathcal{Q}_3$ admits a $\GL_3(\Z)$-equivariant deformation retraction onto a cell complex $X_3\subset \mathcal{Q}_3$ on which the action of $\GL_3(\Z)$ is cocompact.  Lifting this deformation retraction to $\GL_3(\bbR)$, we find a subspace $\widetilde{X}_3\subset \GL_3(\bbR)$ on which the action of $\GL_3(\Z)$ is free, properly discontinuous and cocompact.    
\end{proof}

To complete the flexible cases of \cref{thm:sect4main}, for $M_{HW}$ we will show there exists a metric satisfying the strong generalised Smale conjecture.
\begin{prop}\label{prop:Hantzsche-Wendt} Let $D^3\subset M_{HW}$ be an embedded disk.  Then $\BDiff_{D^3}(M)$ is homotopy finite.
\end{prop}

\begin{proof}
    Let $M=M_{HW}$ and $\pi=\pi_1(M_{HW})$. Recall that $M$ admits a flat metric, and let $\Aff(M)$ denote the group of affine diffeomorphisms of $M$.  By \cite[Theorem 1,  p474]{CharlapVasquez}, the identity component $\Aff_0(M)$ equals $\Isom_0(M)\cong \prod_b \SO(2)$ for any flat metric on $M$, where $b$ is the rank of the center of $\pi$. Moreover, $\Aff(M)/\Aff_0(M)\cong \Out(\pi)$. The center of $\pi$ is trivial, hence $\Aff(M)\cong \Out(\pi)$, which is a finite group \cite[Example 1]{CharlapVasquez} (see \cite{ZimmermannHW} for a correct calculation of $\Out(\pi)$). Moreover, by Corollary \ref{cor:Diff0-list} and Waldhausen's theorem, this implies that $\Aff(M)\hookrightarrow \Diff(M)$ is a homotopy equivalence. 
    
    By a result of Zieschang--Zimmermann \cite[Satz 3.17]{ZieschangZimmermann}, for any flat manifold $M$ and any finite subgroup $F\leq\Aff(M)$ there exists a flat metric on $M$ realising $F$ by isometries.  Applying this to $\Out(\pi)$ itself, we find a flat metric on $M$ such that $\Isom(M)\cong \Aff(M)\simeq \Diff(M)$. Applying Lemma \ref{lem:Riemannian-Finite-Quotient} completes the proof.
\end{proof}

 \subsection{Anosov torus bundles over the circle}
 As noted above, some $T^2$-bundles over $S^1$ are not Seifert-fibered, and hence have non-trivial JSJ decomposition. Such manifolds are geometric and admit $\Sol$-geometry. 
 We will prove that these manifolds satisfy the generalised Smale conjecture. Although this result seems to be known to experts (see the remarks following Problem 3.47 in \cite{KirbyProblems})  we could not find it written down explicitly, so we included it for the sake of completeness.  

 If $M$ is any orientable $T^2$-bundle over $S^1$ then, by the  clutching construction, the diffeomorphism type of $M$ is completely determined by its monodromy $A\in  \SL_2(\Z)\cong \pi_0\Diff^+(T^2)$. In this case, $\SL_2(\Z)$ acts on $T^2$ by diffeomorphisms, and we let $M_A$ denote the $T^2$-bundle with monodromy $A$. Then $M_A$ is aspherical and $\pi_1(M_A)$ is the semi-direct product $\pi_1(M_A):=G_A=\Z^2\langle x,y\rangle\rtimes_A\Z\langle t\rangle,$ where conjugating $\Z^2$ by $t$ acts via $A$.  We obtain a presentation
 \begin{equation}\label{eqn:Torus-bundle}
     G_A=\langle x,y,t\mid [x,y]=1,~txt^{-1}=Ax,~tyt^{-1}=Ay\rangle.
 \end{equation}

Recall that a matrix $A\in \SL_2(\Z)$ is \emph{Anosov} if $|\tr(A)|>2$.
When $A$ is Anosov, $A$ has two real irrational roots $\lambda, \frac{1}{\lambda}$, and thus no eigenvectors in $\bbQ^2$. In particular, $(I-A)$ is invertible over $\bbQ$ and $(I-A)\cdot \Z^2$ has finite index in $\Z^2$. It follows that the group of co-invariants $F_A:=\Z^2/((I-A)\cdot \Z^2)$ is finite.

\begin{lem}\label{lem:Characteristic-Abelian}Let $G_A=\Z^2\rtimes_A \Z$ with $A\in \SL_2(\Z)$ an Anosov matrix. Then $\Z^2$ is characteristic in $G_A.$  
\end{lem}
\begin{proof}
    From the presentation in Equation (\ref{eqn:Torus-bundle}), the abelianisation of $G_A$ is isomorphic to $F_A\oplus \Z$. Thus, since~$F_A$ is finite, $\Z^2$ is the kernel of any nontrivial homomorphism $G_A\rightarrow \Z$, and as such is characteristic.
\end{proof}

Let $f=I-A$ be regarded as an endomorphism $f_{\bbT}\colon \bbT^2 \to \bbT^2$ of the group of translations of $T^2$.  The kernel $\ker(f_\bbT)$ is the subgroup of translations that commute with $A$, or equivalently the group of fixed points for the action of $A$.
Applying the snake lemma to
\[\begin{tikzcd}
	1 & {\bbZ^2} & {\bbR^2} & {\mathbb{T}^2} & 1 \\
	1 & {\bbZ^2} & {\bbR^2} & {\mathbb{T}^2} & 1 
	\arrow["f", hook, from=1-2, to=2-2]
	\arrow["f", "\cong"', from=1-3, to=2-3]
	\arrow["f_\bbT", two heads, from=1-4, to=2-4]
	\arrow[from=2-1, to=2-2]
	\arrow[from=2-2, to=2-3]
	\arrow[from=2-3, to=2-4]
	\arrow[from=2-4, to=2-5]
	\arrow[from=1-1, to=1-2]
	\arrow[from=1-2, to=1-3]
	\arrow[from=1-3, to=1-4]
	\arrow[from=1-4, to=1-5]
\end{tikzcd}\]
we see that $F_A \cong \ker(f_\bbT)$.

Let $N_A$ be the normaliser of $A$ in $\GL_2(\Z)$. Since $\langle A\rangle$ is infinite cyclic, $BAB^{-1}=A^{\pm 1}$ for every $B\in N_A$. Each element of $N_A$ also acts on $T^2$ and preserves $F_A$, hence we get a semi-direct product $F_A\rtimes N_A.$ Since $\langle A\rangle$ is normal in $N_A$ and acts trivially on $F_A$, it is also normal in $F_A\rtimes N_A$.  Define $\overline{N}_A:=N_A/\langle A\rangle$ to be the quotient and $I_A:=F_A\rtimes \overline{N}_A=(F_A\rtimes N_A)/\langle A\rangle$.

We claim that the groups $\overline{N}_A$ and $I_A$ defined above are both finite. This follows from the fact that $\GL_2(\Z)$ is virtually free and $A$ generates a $\Z$ subgroup and so has finite index in its normaliser, since the same is true for free groups.

\begin{lem}\label{lem:Sol-Outer-Automorphism}
    Let $A\in \SL_2(\Z)$ be Anosov, and let $G_A$ and $I_A$ be as above. Then $\Out(G_A)\cong I_A.$
\end{lem}
\begin{proof} 
    Let $\Inn(G_A)$ denote the group of inner automorphisms.
    Observe that $\Z^2$ injects into $\Aut(G_A)$ as the subgroup $\langle \ell_v \mid v\in \Z^2\rangle$, where $\ell_v$ acts trivially on $\Z^2$ and sends $t\mapsto vt$. Now let $\varphi\in \Aut(G_A)$ be any automorphism. By Lemma \ref{lem:Characteristic-Abelian}, $\varphi$ induces $\overline{\varphi}\colon\Z\rightarrow \Z$ and $\varphi'\colon \Z^2\rightarrow \Z^2$. Since $\overline{\varphi}(t)=t^{\pm1}$, we may assume after composing with $\ell_v$ that $\varphi(t)=t^{\pm1}$. Represent $\varphi'$ by the matrix $B$ with respect to the generators $x,y$. From the presentation in Equation (\ref{eqn:Torus-bundle}), we see that:
    \begin{enumerate}
        \item if $\varphi(t)=t$ then $B$ satisfies $BAB^{-1}=A$.
        \item If $\varphi(t)=t^{-1}$ then $B$ satisfies $BAB^{-1}=A^{-1}$.
    \end{enumerate} Conversely, any $B$ satisfying $(1)$ or $(2)$ induces an automorphism of $G_A$ ($t$ gets inverted according to whether $B$ satisfies $(1)$ or $(2)$). The collection of $B\in \GL_2(\Z)$ satisfying $(1)$ or $(2)$ is exactly $N_A$. 

    Thus, $N_A$ injects into $\Aut(G_A)$, and $\langle \ell_v\rangle\cap N_A=1$ since $\langle \ell_v\rangle$ acts as the identity on $\Z^2\leq G_A$ but $N_A$ does not.  Moreover, $N_A$ normalises $\langle \ell_v\rangle$ since $B\ell_vB^{-1}=\ell_{Bv}$ for every $B\in N_A$. Hence $\Aut(G_A)= \langle \ell_v\rangle\rtimes N_A$. On the other hand, conjugation by $v\in \Z^2$ acts trivially on $\Z^2$ and sends $t\mapsto vtv^{-1}=v(tv^{-1}t^{-1})t=((I-A)v) t$, while conjugation by $t$ acts trivially on $t$ and by $A$ on $\Z^2.$ Therefore
    $\langle \ell_v\rangle\cap \Inn(G_A)= \langle\ell_{(I-A)v}\mid v\in \Z^2\rangle\cong (I-A)\cdot\Z^2$ while $N_A\cap \Inn(G_A)=\langle A\rangle$.  We conclude that $\Out(G_A)\cong (\Z^2/(I-A)\cdot \Z^2)\rtimes (N_A/\langle A\rangle)=F_A\rtimes \overline{N}_A=I_A$ as desired. 
\end{proof}

Recall that $\Sol$ is a 3-dimensional solvable Lie group which decomposes as a semi-direct product $\bbR^2\rtimes \bbR$, where the $\bbR$ factor acts on $\bbR^2$ via the 1-parameter family of diagonal matrices of the form 
\[\left(\begin{array}{cc}e^t&0\\0&e^{-t}\end{array}\right).\]
Choosing a left-invariant metric at the identity for which the standard basis vectors of $\bbR^3$ are orthonormal,  $\Isom(\Sol)\cong \Sol \rtimes D_8$, where $\Sol$ acts on itself by left translations, and $D_8$ is the dihedral group of order 8 \cite{Scott83}. In coordinates $(x,y,t)$ where $(x,y,0)$ span the $\bbR^2$-factor and $(0,0,t)$ spans the $\bbR$-factor, $D_8$ is the group of 8 elements $R(\epsilon_1,\epsilon_2)$, $S(\epsilon_1,\epsilon_2)$ for $\epsilon_i=\pm1$ acting as
$R(\epsilon_1,\epsilon_2)\colon (x,y,t)\mapsto(\epsilon_1\cdot x,\epsilon_2\cdot y,t)$, and $S(\epsilon_1,\epsilon_2)\colon (x,y,t)\mapsto(\epsilon_1\cdot y,\epsilon_2\cdot x,-t).
$

If $A$ is Anosov, then $M_A$ admits $\Sol$-geometry. Let $\Isom(M_A)$ be the group of isometries of $M_A$ with respect to this geometric structure. 

\begin{thm}\label{thm:Anosov-T2-bundles}
Let $M_A$ be a $T^2$-bundle over $S^1$ with Anosov monodromy $A\in \SL_2(\Z)$. The inclusion $\Isom(M_A)\rightarrow \Diff(M_A)$ is a homotopy equivalence.  In particular, $\BDiff_{D^3}(M_A)$ is homotopy finite. 
\end{thm}
\begin{proof}
    Let $\rho\colon M_A\rightarrow S^1$ be the bundle projection. 
    Write $G_A=\pi_1(M_A)=\Z^2\langle x,y\rangle\rtimes_A \Z\langle t\rangle$. Combining Corollary \ref{cor:Diff0-list} part (1), which states that $\Diff_0(M_A)\simeq *$, with Waldhausen's theorem \cite{Waldhausen68} that $\pi_0\Diff(M_A)\cong \Out(G_A)$, we conclude that $\Diff(M_A)\rightarrow \Out(G_A)$ is a homotopy equivalence.  We will show that $\Isom(M_A)\rightarrow \Out(G_A)$ is an isomorphism.

    By Lemma \ref{lem:Sol-Outer-Automorphism}, $\Out(G_A)\cong I_A = (F_A\rtimes N_A)/\langle A\rangle$. Recall that $I_A$ is finite. We claim $I_A$ can be realised as a finite group of diffeomorphisms of $M_A$, \emph{i.e.}~we construct a splitting of $\Diff(M_A)\rightarrow \Out(G_A)$. Let $\smash{\widetilde{M}_A}\to M_A$ be the regular cover of induced by $\rho_*\colon G_A\to \Z$. Then $\smash{\widetilde{M}_A\cong T^2\times \bbR}$ and projection onto $\bbR$ is the lift $\tilde{\rho}$ of $\rho$. 
    Recall that $F_A$ is the group of fixed points of $A$ acting on $T^2$. For any $f\in F_A$ and $v\in T^2$, we have $A(v+f)=Av+f$.  Hence acting trivially on $\bbR$ and translating of each fiber of $\tilde{\rho}$ by $f$ gives a well-defined diffeomorphism of $\smash{\widetilde{M}_A}$ that commutes with the deck group. Any $B\in N_A$ acts on $T^2$ as a linear map in the standard way, and on $\bbR$ via the action of $N_A$ on the axis $\gamma\subset \bbH^2$. The action of $N_A$ normalises the translations $F_A$ so we obtain an action of $F_A\rtimes N_A$ on $\smash{\widetilde{M}_A}$.  Since $F_A\rtimes N_A$ normalises $A$, this action normalises the deck group $\langle A\rangle$ action hence descends to an action of $I_A$ on the quotient $M_A$. Thus we have constructed our splitting and $\Out(G_A)$ acts on $M_A$ by diffeomorphisms. By a result of Meeks--Scott \cite[Theorem 8.2]{Meeks-Scott}, finite group actions on $\Sol$ manifolds preserve the geometric structure. As $I_A$ is finite, $\Out(G_A)$ can be realised on $M_A$ by isometries in the $\Sol$-metric. This implies $\Isom(M_A)\rightarrow \Out(G_A)$ is onto.

    For injectivity, let $\phi\in \Isom(M_A)$ and suppose that $\phi_*\in \Out(G_A)$ is trivial.  Lift $\phi$ to $\smash{\widetilde{\phi}}\colon \Sol\rightarrow \Sol.$ We can choose $\smash{\widetilde{\phi}}$ so that $g\circ \smash{\widetilde{\phi}}=\smash{\widetilde{\phi}}\circ g$ for all $g\in G_A,$ which implies that $\smash{\widetilde{\phi}}$ centralises $G_A$. Since $G_A\cap \bbR^2$ is a lattice and no element of $D_8$ commutes with an entire lattice in $\bbR^2$, we must have  $\smash{\widetilde{\phi}}\in \Isom_0(\Sol)\cong \Sol$. But the centraliser of any element of $\bbR^2$ is $\bbR^2$, while the centraliser of any element projecting nontrivially to $\bbR$ is the one 1-parameter subgroup containing it.  Since $G_A$ contains elements of both types, its centraliser is trivial.  Thus $\smash{\widetilde{\phi}}$ is the identity, and it follows that $\Isom(M_A)\cong \Out(G_A)$. We conclude that $\BDiff_{D^3}(M_A)$ is  homotopy finite by \cref{lem:Riemannian-Finite-Quotient}.
\end{proof}

\section{Decomposing irreducible \texorpdfstring{$3$}{3}-manifolds}
\label{section:decomposing irreducible}

In \cref{section:strongsmale} we used the strong generalised Smale conjecture and results of a similar nature to prove \cref{thm:sect4main} for certain geometric manifolds, appearing in the geometric or JSJ decomposition of $M$. In this section we complete our proof of homotopy finiteness of $\BDiff_{D^3}(M)$ for $M$ irreducible. One key step of our proof will be to glue $M$ together from its JSJ pieces, and we use a similar gluing argument when dealing with singular Seifert fibered manifolds. 
In order to assemble the proof of homotopy finiteness for these pieces into a proof for $M$, we first introduce a stronger notion of finiteness, where we require that the finiteness is inherited by classifying spaces of certain subgroups.

In general, we wish to start with a fiber sequence of the form
\[  
    \Diff_U(M) \longrightarrow \Diff(M, U) \longrightarrow \Diff(U),
\]
and use \cref{lem:fiberseq/fiberseq} to get a homotopy fiber sequence of classifying spaces.
However, a fundamental problem is that the above sequence need not be a principal short exact sequence (\cref{defn: principal ses}) because the fibration might not be surjective.
Its image is the subgroup $\Diff'(U) \subseteq \Diff(U)$ of those diffeomorphism of $U$ that can be extended to a diffeomorphism of $M$. Restricting to $\Diff'(U)$ in the base, the map from $\Diff(M,U)$ is indeed a $\Diff_U(M)$-principal bundle, and so we can apply \cref{lem:fiberseq/fiberseq}.
The homotopy fiber sequence that we get on classifying spaces is
\[
    \BDiff_U(M) \longrightarrow \BDiff(M, U) \longrightarrow \BDiff'(U).
\]
Therefore to deduce homotopy finiteness of the middle term we will need to know homotopy finiteness of this base, rather than of $\BDiff(U)$.
It can be a little annoying that $\Diff'(U)$ depends on $M$ and not just on $U$, so we instead describe a suitable class of subgroups of $\Diff(U)$ whose finiteness we might need to know.
Whether or not a diffeomorphism $\varphi$ of $U$ can be extended to $M$ only depends on the mapping class $[\varphi_{|\partial}] \in \pi_0\Diff(\partial U)$ on the boundary, so we restrict ourselves to considering subgroups of $\pi_0\Diff(\partial U)$.

\begin{defn}\label{defn:Diff-Gamma}
    For $N$ a $3$-manifold and $\Gamma < \pi_0\Diff(\partial N)$ a subgroup of the mapping class group of its boundary we let $\Diff^\Gamma(N) \subset \Diff(N)$ denote the subgroup of those diffeomorphisms $\varphi\colon N \to N$ for which $[\varphi_{|\partial N}]$  is in $\Gamma$.
\end{defn}

\begin{defn}\label{defn:hereditarily-finite}
    Let $(M, F)$ be a pair consisting of a compact $3$-manifold and a (possibly empty) union of boundary components $F \subset \partial M$.
    We say that $(M, F)$ is \emph{hereditarily finite} if for any subgroup $\Gamma \le \pi_0\Diff(\partial M)$ the space
    $\BDiff_F^\Gamma(M)$ 
    is homotopy finite.
\end{defn}

\begin{rem}\label{remark:suffices to show HF for special cases}
    Once we have checked that $\BDiff_F^\Gamma(M)$ is homotopy finite for some $\Gamma$, it follows that $\smash{\BDiff_F^{\Gamma'}(M)}$ is also homotopy finite for all finite index subgroups $\Gamma' < \Gamma$ because the map
    \[
        \BDiff_F^{\Gamma'}(M) \longrightarrow
        \BDiff_F^{\Gamma}(M)
    \]
    is a finite covering (this is a special case of \cref{lem:principal-bundle-subgroup} where $G/H$ is finite). 
    Moreover, without loss of generality it suffices to check finiteness of $\BDiff_F^\Gamma(M)$ only for those subgroups $\Gamma$ that are contained in the image of
    \[
        \pi_0\Diff_F(M) \longrightarrow \pi_0\Diff(\partial M).
    \]
    Since this map factors through $\pi_0\Diff(M)$, if $\pi_0\Diff_F(M)$ or $\pi_0\Diff(M)$ is finite, to show hereditary finiteness it suffices to check the case where $\Gamma$ is the entire group.
    In particular, if $\partial M$ is a disjoint union of spheres then $\pi_0(\Diff(\partial M))$ is the finite group $(\Z/2) \wr \Sym_n$, so it suffices to check that $\BDiff_F(M)$ is homotopy finite.
\end{rem}

In their paper \cite{HatcherMcCullough} on the homotopy finiteness of $\BDiff_F(M)$ for irreducible $M$, Hatcher and McCullough in fact show hereditary finiteness.
While we believe that their argument can be used to show hereditary finiteness for any irreducible manifold $M$, we will only need it for the cases that appear as part of the JSJ decomposition.
\begin{thm}[Hatcher--McCullough]\label{thm:HM-hereditary}
    Let $M$ be an irreducible $3$-manifold that is either hyperbolic or Seifert-fibered.
    Let $F \subset \partial M$ be a non-empty union of boundary components, including all the compressible ones. 
    Then $(M, F)$ is hereditarily finite.
\end{thm}
\begin{proof}
    The main theorem of \cite{HatcherMcCullough} states that $\BDiff_F(M)$ is homotopy equivalent to an aspherical finite CW complex.
    In the hyperbolic case $\pi_0 \Diff_F(M)$ is a finitely generated free abelian group \cite[Proposition 3.2]{HatcherMcCullough}, so any $\pi_0 \Diff_F^\Gamma(M)$ must also be a finitely generated free abelian group and hence must have a finite classifying space.

    In the case of a Seifert-fibered manifold $\varphi\colon M \to S$ the authors argue  \cite[p108]{HatcherMcCullough} that $\pi_0\Diff_F(M)$ must have a finite classifying space. During this proof, they construct an intermediate group $G\leq \pi_0\Diff_F(M)$ and prove that $G$ has finite classifying space, using only that $\pi_0 \Diff_\partial(M) \leq G$. Therefore their proof applies to $\pi_0 \Diff_F^\Gamma(M)$ for any $\Gamma\le\pi_0\Diff(\partial M)$.\qedhere
\end{proof}

The argument that the strong generalised Smale conjecture for $M$ implies finiteness of $\BDiff_{D^3}(M)$ can easily be improved to also show hereditary finiteness.
\begin{lem}\label{lem:Riemannian-hereditarily-finite}
    Suppose $(M,g)$ is a Riemannian 3-manifold such that the inclusion $\Isom(M)\rightarrow \Diff(M) $ is a homotopy equivalence.
    Then $(M \setminus \interior{D}^3, S^2)$ is hereditarily finite.
\end{lem}
\begin{proof} 
    This works analogously to \cref{thm:contractible-collars} since the subgroup $\Isom^\Gamma(M) \le \Isom(M)$ is still a compact Lie group acting freely on the closed manifold $\Fr^\perp(M)$.
\end{proof}

There are three cases in which we need to show hereditary finiteness, shown as blue leaves in \cref{fig:flowchartirreducible}.
The case of hyperbolic manifolds is dealt with by combining \cref{lem:Riemannian-hereditarily-finite} above with \cref{thm:Isom-Equivalence},
so it remains to consider fiber-rigid Haken Seifert-fibered  manifolds, both with and without singularities.

\begin{ex}\label{ex:not-hereditarily-finite}
    The pair $(T^2 \times I\setminus \interior{D}^3, S^2)$ is not hereditarily finite.
    Let $\Gamma < \GL_2(\bbZ)$ be any subgroup, which we may regard as a  subgroup of the diagonal in $\pi_0 \Diff(\partial (T^2 \times I)) \cong \GL_2(\bbZ) \times \GL_2(\bbZ)$.
    By the argument of \cref{prop: finiteness sphere or torus x I} there is a homotopy fiber sequence
    \[
        \Fr(T^2)\hq \Diff^\Gamma(T^2) \longrightarrow
        \BDiff_{D^3}^\Gamma(T^2 \times I) \longrightarrow \RP^2
    \]
    and furthermore the left term is equivalent to $\GL_2(\bbR) \hq \Gamma$.
    If we let $\Gamma < \GL_2(\bbZ)$ be some infinite rank free group, then one can check using the long exact sequence on homotopy groups that $\pi_0 \Diff_{D^3}^\Gamma(T^2 \times I)$ cannot be finitely generated.
    In particular $\BDiff_{D^3}^\Gamma(T^2 \times I)$ cannot be equivalent to a finite CW complex.
\end{ex}

\subsection{Canonical submanifolds}\label{subsec:canonical}
We now explain the general strategy for deducing finiteness of $\BDiff_{D^3}(M)$ by decomposing $M$ along certain \emph{canonical} submanifolds $N \subset M$.
These are analogous to characteristic subgroups $H < G$, which are preserved by all group automorphisms.

\begin{defn}\label{defn:canonical}
    We say that a submanifold $N \subset M$ is \emph{canonical} if the inclusion
    \[
        \Diff(M, N) \to \Diff(M)
    \]
    is a homotopy equivalence.
\end{defn}

    Let us discuss some examples. If $M$ is a manifold with boundary and $N \subset \partial M$ is the union of all spherical boundary components then $N$ is a canonical submanifold because every diffeomorphism must preserve it, and in fact $\Diff(M, N) = \Diff(M)$.
    However, canonical submanifolds need not always be strictly preserved.
    For example, $\{0\} \subset D^n$ is a canonical submanifold because there is a fiber sequence $\Diff(D^n, \{0\}) \to \Diff(D^n) \to \interior{D}^n$.
    To give a more sophisticated example, let $P_1$ and $P_2$ be two irreducible $3$-manifolds, neither of which is $S^3$.
    Then any essential sphere $S^2 \subset P_1 \# P_2$ is canonical by a theorem of Hatcher \cite{Hatcher1981}, which we reproved in Theorem \ref{thm:two-prime-pieces}.

From now on, we let $M$ be a compact connected oriented $3$-manifold, $F \subset \partial M$ a non-empty union of boundary components and $T \subset \interior{M}$ a closed orientable codimension $1$ submanifold (not necessarily connected). Similar to our notation $2\Sigma$, which we use to denote new boundary created in $M$ after cutting along $\Sigma$, let $2T = \partial (M \ca T) \setminus \partial M$ denote the new boundary components created after cutting along $T$.

\begin{prop}[Cutting along submanifolds]\label{prop:cut-along-sub}
    Let $M$, $F\subset \partial M$, $T\subset \interior{M}$, and $2T\subset \partial (M\ca T)$ be as above. 
    Suppose that 
    \begin{enumerate}
        \item[$(\dagger_0)$]
        there is a component $N_0 \subset M \ca T$ such that $N_0 \cap F \neq \emptyset$ and $(N_0, N_0 \cap F)$ is hereditarily finite.
        \item[$(\dagger)$]
        For every other component $N \subset M \ca T$ and any non-empty union of components $\emptyset \neq F' \subset \partial N$ containing $N \cap F$, $(N, F')$ is hereditarily finite.
    \end{enumerate}
    Then $\BDiff_F^{\Gamma}(M, T)$ is homotopy finite for all $\Gamma \le \pi_0\Diff(\partial M)$.
    In particular, if $T \subset M$ is canonical, then $(M, F)$ is hereditarily finite.
\end{prop}

Note that in the proposition, if we know $(\dagger)$ for every component $N \subset M \ca T$, then we also know $(\dagger_0)$. 
This is because $(\dagger_0)$ is a weaker condition on $N_0$ than $(\dagger)$ would be.

\begin{proof}
    The boundary $2T$ created when cutting $M$ along $T$ is diffeomorphic to two copies of $T$. In this proof it will be helpful to utilise this and hence we write $2T$ as $T^+ \sqcup T^-$.
    (This decomposition into $T^+$ and $T^-$ is non-canonical and depends on a coorientation of $T$ that we choose arbitrarily.)
    Hence $\partial (M \ca T) = \partial M \sqcup T^+ \sqcup T^-$.

    We proceed by induction over the number of connected components of $M \ca T$, starting with the case where $M \ca T$ is connected, so $N_0 = M \ca T$.
    There is a canonical injective group homomorphism
    \[
        \Diff(T) \hookrightarrow \Diff(T^+ \sqcup T^-), \qquad
        \varphi \mapsto \varphi \sqcup \varphi
    \]
    and we define the space of matchings $\calM_{T}$ as the coset space $\Diff(T^+ \sqcup T^-)/\Diff(T)$. We have  $$\calM_{T} \cong\coprod_{\mu_T} \prod_{A \in \pi_0(T)} \Diff(A)$$ where $\mu_T$ is the finite set of matchings of $\pi_0(T)$ (such that matched components are diffeomorphic) and the product runs over all connected components $A \subset T$. Given a matching in $\mu_T$, the product measures the difference between diffeomorphisms on each matched pair of manifolds: there is one such pair for each $A\in \pi_0(T)$. 
    In particular, the connected components of $\calM_{T}$ are homotopy finite because $\Diff_0(A)$ is either $\SO(3)$, $S^1 \times S^1$, or a point, depending on the genus of $A$.
    The group $\Diff_F(M \ca T)$ acts on $\calM_T$ via the group homomorphism
    \[  
        \Diff_F(M \ca T, 2T) \longrightarrow \Diff(T^+ \sqcup T^-)
    \]
    given by restricting to the boundary.
    Using \cref{thm:boundary-locally-retractile} one can show that $\calM_T$ is $\Diff_F(M \ca T, 2T)$ locally retractile.
    The stabiliser of this action is the subgroup of $\Diff_F(M \ca T, 2T)$ consisting of those diffeomorphism that can be glued to a homeomorphism of $M$.
    Because the space of collars is contractible (\cref{thm:contractible-collars}), this stabiliser is equivalent to the smaller subgroup of those diffeomorphisms that can be glued to a diffeomorphism of $M$ and this group is isomorphic to $\Diff_F(M, T)$.
    By \cref{lem:orbit-stabiliser-transitive} we obtain a homotopy fiber sequence
    \[
        X \longrightarrow \BDiff_F(M, T) \longrightarrow \BDiff_F(M \ca T, 2T)
    \]
    where $X \subset \calM_{T}$ is the orbit of the basepoint of $\calM_{T}$ under the action of $\BDiff_F(M \ca T, 2T)$.
    In fact, we would like to study the smaller group $\Diff_F^\Gamma(M\ca T, 2T)$ for some $\Gamma \le \pi_0\Diff(\partial M)$.
    Let $\Gamma' \le \pi_0\Diff(\partial (M \ca T))$ be the subgroup defined as the image of the composite map
    \[
        \Diff_F^\Gamma(M, T) \to \Diff_F(M \ca T) \to \pi_0 \Diff(\partial (M \ca T)).
    \]
    This in particular ensures that elements of $\Diff_F^{\Gamma'}(M \ca T)$ preserve $2T \subseteq \partial (M \ca T)$ setwise.
    Then $\Diff_F^\Gamma(M, T) \to \Diff_F^{\Gamma'}(M \ca T)$ hits all components.
    So if we let $X' \subset \calM_{T}$ be the orbit of the base point under the $\Diff_F^{\Gamma'}(M \ca T)$-action, then $X'$ is connected and as it is a connected component of $\calM_{T}$, it is homotopy finite.
    Now we apply \cref{lem:orbit-stabiliser-transitive} again to get the homotopy fiber sequence
    \[
        X' \longrightarrow \BDiff_F^{\Gamma}(M, T) \longrightarrow \BDiff_F^{\Gamma'}(M \ca T).
    \]
    We have already argued that $X'$ is homotopy finite, and the base is homotopy finite by assumption $(\dagger_0)$ (with $N_0=M\ca T$).
    Therefore the total space is homotopy finite, which proves the base of the induction.

    Now suppose that $M \ca T$ has $n$ connected components and that the proposition holds for manifolds $T'\subset \interior{M}'$ such that $M'\ca T'$ has $n'$ connected components and $n'<n$.
    From $(\dagger_0)$ there exists a component $N_0 \subset M \ca T$, which by assumption satisfies that $F_0 = \partial N \cap F$ is non-empty. 

    We can decompose $T = T_0 \sqcup T_1 \sqcup T_2$ where $T_i \subset T$ is the union of those connected components that contribute $i$ boundary components to $N_0$.
    Let $N$ be the image of $N_0$ under the map $M\ca T \to M$ that glues $2T$ back together.
    (So $T_2 \subset \interior{N}$, $T_1 \subset \partial N$, and $T_0 \subset M \setminus \interior{N}$.)
    Then there is a fiber sequence
    \[
        \Diff_{(F \setminus F_0) \sqcup N}(M, T_0) \longrightarrow
        \Diff_F(M, T) \xrightarrow{\;r\;}
        \Diff_{F_0}(N, T_2).
    \]
    Given $\Gamma \le \pi_0\Diff(\partial M)$ we can choose $\Gamma' \le \pi_0\Diff(\partial N)$ such that 
    $r(\Diff_F^\Gamma(M, T))$ is precisely $\smash{\Diff_{F_0}^{\Gamma'}(N, T_2)}$, so that we get a homotopy fiber sequence
    \[
        \BDiff_{(F \setminus F_0) \sqcup N}^{\Gamma}(M, T_0) \longrightarrow
        \BDiff_F^\Gamma(M, T) \xrightarrow{\;r\;}
        \BDiff_{F_0}^{\Gamma'}(N, T_2).
    \]
    The base is homotopy finite by the base of the induction, as $N \ca T_2$ is connected.
    The fiber is equivalent to 
    \[
        \BDiff_{(F \setminus F_0) \sqcup T_1}^{\Gamma''}(M \setminus \interior{N}, T_0)
    \]
    where $\Gamma'' \le \pi_0\Diff(\partial (M \setminus \interior{N}))$ is chosen by restricting $\Gamma$.
    The manifold $(M \setminus \interior{N})\ca T_0$ has strictly fewer connected components than $M\ca T$ by construction, so by the induction hypothesis this space is homotopy finite.
    (Note that because we introduce new boundary components coming from $T_1$ that must be fixed, we need the stronger statement $(\dagger)$ to ensure that this new manifold still satisfies $(\dagger_0)$.)
    Consequently, by \cref{lem:fiber-sequence-finite} the above homotopy fiber sequence shows that $\BDiff_F^\Gamma(M, T)$ is homotopy finite, concluding the proof.
\end{proof}

When cutting along spheres the same argument also applies without the assumption of hereditary finiteness.
\begin{cor}[Cutting along spheres]\label{cor:cut-along-spheres}
    Let $M$ be a compact connected $3$-manifold, $F \subset \partial M$ a non-empty union of boundary components and $\Sigma \subset M$ a sphere system. 
    Suppose that 
    \begin{enumerate}
        \item[$(\ddagger)$]
        for every component $N \subseteq M \ca \Sigma$ and every union of boundary components $F' \subseteq \partial N$ satisfying $F' \neq \emptyset$ and $F' \supseteq \partial N \cap F$ we know that  $ \BDiff_{F'}(N) $ is homotopy finite.
    \end{enumerate}
    Then $\BDiff_F(M, \Sigma)$ is homotopy finite.
\end{cor}
\begin{proof}
    Inspecting the proof of \cref{prop:cut-along-sub} we see that if we start out with $\Gamma = \pi_0 \Diff(\partial M)$, then each subsequent $\Gamma'$ can be chosen as a finite index subgroup because it is obtained by imposing a restriction on the spherical part of the boundary and $\smash{\pi_0 \Diff(\coprod_k S^2) \cong  (\bbZ/2)\wr\Sym_k} $ is a finite group.
    Since hereditary finiteness for finite index subgroups is automatic, the proof indeed works without the assumption that $(N, F')$ is hereditarily finite.
    As remarked below \cref{prop:cut-along-sub} we do not need to require an analogue of $(\dagger_0)$, if we simply require the stronger condition $(\ddagger)$ for all components.
\end{proof}

As we are about to cut along 2-tori we will need a form of hereditary finiteness for them.
Note that the following is in contrast with \cref{ex:not-hereditarily-finite}.
\begin{lem}\label{lem:2torus-finite}
    Let $\Gamma < \GL_2(\bbZ)$ be a finite subgroup.
    Then the induced subgroup $\Diff_{D^2}^\Gamma(T^2) < \Diff_{D^2}(T^2)$ has a finite classifying space.
\end{lem}
\begin{proof}
    The Teichm\"uller space of marked, flat 2-tori can be identified with the upper half plane $\bbH^2$, equipped with its natural hyperbolic metric. $\GL_2(\Z)$ acts properly discontinuously on $\bbH^2$ by isometries, altering the marking by a change of basis. Since $\bbH^2$ is complete, simply connected, and nonpositively curved, any finite subgroup $\Gamma< \GL_2(\Z)$ has a global fixed point. This corresponds to a flat metric on $T^2$ which realises $\Gamma$ by isometries. Since the identity component of $\Diff(T^2)$ is homotopy equivalent to $\bbT^2$, the group of translations, $\Diff^\G(T^2)$ is homotopy equivalent to a subgroup $\bbT^2\rtimes \G$ that acts freely on $\Fr^\perp(T^2)$. Lemma \ref{lem:Riemannian-hereditarily-finite} now completes the proof.
\end{proof}

In the next proof we decompose the frame bundle $\Fr(M)$ as a pushout according to the JSJ decomposition. A version of this idea appeared in Nariman's preprint \cite{Nariman}, though we choose a different decomposition.

\begin{prop}\label{prop:cut-along-tori}
    Let $M$ be a compact connected $3$-manifold and $T \subset \interior{M}$ a union of tori.
    Assume that $T$ is canonical, i.e.~the inclusion $\Diff(M, T) \hookrightarrow \Diff(M)$ is an equivalence.
    Suppose that any connected component $N$ of $M \ca T$ satisfies:
    \begin{enumerate}
        \item 
        $(N \setminus \interior{D}^3, S^2)$ is hereditarily finite.
        \item
        For any non-empty union of boundary components $F \subset (2T \cap \partial N)$,
        $(N, F)$ is hereditarily finite. 
        \item The image of $\Diff(M, T) \to \pi_0(\Diff(T))$ is a finite group.
    \end{enumerate}
    Then $(M \setminus \interior{D}^3, S^2)$ is hereditarily finite.
\end{prop}
\begin{proof}
    Let $\Fr_3(T) \subset \Fr(M)$ denote the subspace of those framed points based in $T$.
    An element of $\Fr_3(T)$ is a tuple of a point $p \in T$ and a frame for $T_p(T) \oplus \bbR$, so in particular $\Fr_3(T)$ is different from $\Fr(T)$ .

    Recall that $2T = \partial (M \ca T) \setminus \partial M$.
    This is the co-orientation double cover of $T$ and
    we can picture this as being obtained from $T$ by translating it infinitesimally in both normal directions.
    Since $T$ is two-sided (\emph{i.e.}~co-oriented in $M$) we know that $2T \cong T \sqcup T$, but this identification is not canonical as it requires choosing a coorientation for each component, which we do not do in this proof as the action of $\Diff(M,T)$ will not necessarily preserve the decomposition.

    We can re-glue $M$ as $(M\ca T) \cup_{2T} T$, and similarly we can write the frame bundle on $M$ as a pushout 
    \[
        \Fr(M) = \Fr(M \ca T) \cup_{\Fr_3(2T)} \Fr_3(T).
    \]
    This is in fact a homotopy pushout because submanifold inclusions are cofibrations.
    It is moreover equivariant for the group $\Diff^\Gamma(M, T)$, so by \cref{lem:homotopy-orbits-functor} we obtain the following homotopy pushout square.
    \[\begin{tikzcd}
        \Fr_3(2T)\hq \Diff^\Gamma(M, T) \ar[r] \ar[d] & 
        \Fr(M \ca T)\hq \Diff^\Gamma(M, T) \ar[d] \\
        \Fr_3(T)\hq \Diff^\Gamma(M, T) \ar[r] & 
        \Fr(M)\hq \Diff^\Gamma(M, T)
    \end{tikzcd}\]
    The bottom right term is $\Fr(M) \hq \Diff^\Gamma(M, T) \simeq \Fr(M) \hq \Diff^\Gamma(M) \simeq \Diff_{D^3}^\Gamma(M)$, so to prove the claim it will suffice to show that the three other terms in the square are homotopy finite.

    We begin with the top right term $\Fr(M \ca T) \hq \Diff^\Gamma(M, T)$.
    We can equivalently replace $\Fr(M \ca T)$ by $\emb(D^3, (M \ca T)^\circ)$.
    The space $\emb(D^3, (M\ca T)^\circ)$ is $\Diff_0((M \ca T)^\circ)$-locally retractile by combining \cref{thm:emb-retractile}, with \cref{lem:locally-retractile}(3) and (4). We conclude it is locally retractile for the action of the larger group $\Diff^\Gamma(M, T)$ by \cref{lem:locally-retractile}(4).
    We may then apply the homotopy orbit stabiliser lemma (\cref{lem:orbit-stabiliser-v2}) to compute
    \[
        \emb(D^3,(M \ca T)^\circ)\hq \Diff^\Gamma(M, T) 
        \simeq \coprod_{[i\colon D^3 \hookrightarrow (M \ca T)^\circ]} B\Diff_{i(D^3)}^\Gamma(M, T).
    \]
    Here the coproduct runs over the set of orbits of the action of $\Diff^\Gamma(M, T)$ on $\pi_0\emb(D^3, (M \ca T)^\circ)$, which is finite because $\pi_0\Fr(M \ca T)$ is finite.
    In each orbit we pick an embedded $i\colon D^3 \hookrightarrow (M \ca T)^\circ$ as a representative and the corresponding term in the coproduct is the classifying space of the stabiliser group  $\Diff_{i(D^3)}^\Gamma(M, T)$.
    This is equivalent to $\BDiff_{\partial D^3}^\Gamma(M \setminus \interior{D}^3, T)$, which is homotopy finite by applying \cref{prop:cut-along-sub} to $T \subset M \setminus \interior{D}^3$ with $F=S^2$.
    (Condition $(\dagger_0)$ is satisfied for the component where we removed the disk by $(1)$ and condition $(\dagger)$ is satisfied by $(2)$.)

    We still have to deal with the two left hand terms in the homotopy pushout square, 
    but the left vertical map in the square is a double-covering, so it suffices to show finiteness of $\Fr_3(T) \hq \Diff^\Gamma(M,T)$.
    Choose at each point $p \in T$ a normal vector up to sign $\pm \xi_p$ smoothly depending on $p$.
    This is a contractible space of choices, so the subgroup $\Diff_{\pm \xi}^\Gamma(M, T) \subset \Diff^\Gamma(M, T)$ of those diffeomorphism $\varphi$ satisfying $D\varphi(\xi_p) = \pm \xi_{\varphi(p)}$ is equivalent to the entire group.
    We then have a fiber sequence
    \[
        \Fr_{\pm\xi}(T) \longrightarrow \Fr_3(T) \longrightarrow (\bbR^3\setminus 0)/\pm
    \]
    where the fibration sends a framed point $(p \in T, \theta\colon \bbR^3 \cong T_pM)$ to $\theta^{-1}(\pm \xi_p)$.
    The fiber $\Fr_{\pm\xi}(T)$ is the subspace of those framed points on $T$ whose first vector is $\xi_p$ or $-\xi_p$.
    There is a homotopy equivalence $\Fr(2T) \simeq \Fr_{\pm\xi}(T)$ defined by extending the $2$-framing to a $3$-framing by adding $\pm \xi_p$ with the sign depending on the coorientation.
    Taking homotopy orbits via \cref{lem:fiberseq/fiberseq} applied to the short exact fiber sequence $\Diff_{\pm\xi}^\Gamma(M, T)\to \Diff_{\pm\xi}^\Gamma(M, T)\to *$, we get a homotopy fiber sequence
    \[
        \Fr(2T)\hq \Diff_{\pm\xi}^\Gamma(M, T) \longrightarrow
        \Fr_3(T)\hq \Diff_{\pm\xi}^\Gamma(M, T) \longrightarrow (\bbR^3\setminus 0)/\pm
    \]
    where the base is equivalent to the finite CW complex $\RP^2$.
    By the homotopy orbit stabiliser lemma (\cref{lem:orbit-stabiliser-v2}), the fiber has finitely many connected components (since $\pi_0(\Fr(2T))$ is finite) each of which is equivalent to
    \[
        \BDiff_{D^2}^{\Gamma, +}(M, T)
    \]
    where we fix some $2$-disk in $T$ as well as its co-orientation in $M$. (Together with the fixed 2-disk, the latter is equivalent to being orientation-preserving.) 
    Now there is a homotopy fiber sequence
    \[
        \BDiff_{T_0}^{\Gamma,+}(M, T) \longrightarrow 
        \BDiff_{D^2}^{\Gamma, +}(M, T)
        \longrightarrow \BDiff_{D^2}^{\Gamma'}(T_0)
    \]
    where $T_0 \subset T$ is the connected component containing the disk and $\Gamma' < \pi_0(\Diff(T_0))$ consists of those isotopy classes that can be extended to a diffeomorphism in $\Diff^{\Gamma,+}(M, T)$.
    The base is homotopy finite by \cref{lem:2torus-finite} because $\Gamma'$ is finite by assumption (3).
    The fiber is equivalent to $\BDiff_{2T_0}^{\Gamma}(M\ca T_0, T \setminus T_0)$
    where we write $2T_0 \subset \partial(M \ca T_0)$ for the two new tori created by cutting along $T_0$.
    This space is homotopy finite by \cref{prop:cut-along-sub} applied to $T \setminus T_0 \subset M \ca T_0$ with $F = T_0$. 
    ($M \ca T_0$ might not be connected, but because the components cannot be permuted, we can apply \cref{prop:cut-along-sub} to each component individually.)
    Applying \cref{lem:fiber-sequence-finite} completes the proof.
\end{proof}

\subsection{Seifert-fibered solid tori}\label{subsection: SF solid torus}
In order prove hereditary finiteness for $(M \setminus \interior{D}^3, S^2)$ when $M$ is Haken Seifert-fibered and not exceptional, in the presence of singular fibers we will require results about fibered solid tori. Of course, there are many Seifert fiberings of $S^1\times D^2$, and we consider all of them.

Any Seifert fibering of a solid torus is determined by a pair $(p,q)$ of coprime integers with $0<q<p$, from which it can be constructed as follows. 
(We additionally allow $(p,q) = (1,0)$ to get the non-singularly fibered torus.)
Equip $S^1$ and $D^2$ with their standard Euclidean metrics and let $\Z/p$ act on the product $S^1\times D^2$ diagonally by $(x,y)\mapsto (\zeta_p\cdot x,\zeta_p^q\cdot y)$, where $\zeta_p=e^{2\pi i/p}$ is a primitive $p^{th}$ root of unity.
This action is by isometries and permutes the fibers of the trivial fibering $S^1\times D^2\rightarrow D^2$.
Thus the quotient, which is still diffeomorphic to $S^1\times D^2$, inherits a Euclidean metric and a Seifert fibering $\varphi=\varphi_{(p,q)}\colon S^1\times D^2\rightarrow D^2/(\Z/p)$.
With respect such a fibering $\varphi$, we have the diffeomorphism groups 
$\Diff^v(S^1 \times D^2; [\varphi]) \le \Diff^f(S^1\times D^2;[\varphi]) \le \Diff(S^1 \times D^2)$.
To simplify notation, we will suppress writing $[\varphi]$ unless it is unclear which fibering we are discussing.

\begin{lem}\label{lem:Diffv-solid-torus}
    Let $\varphi\colon M = S^1 \times D^2 \to D = D^2/(\bbZ/p)$ be any Seifert-fibering of the solid torus.
    If $p\ge 3$ and $0<q<p$, then the inclusion
    \[
        \SO(2) \hookrightarrow \Diff^v(S^1 \times D^2; [\varphi])
    \]
    defined by applying the same rotation to each non-singular fiber is a homotopy equivalence.
    If $p=1$ or $(p,q) = (2,1)$, then we can also reflect fibers and have $\Or(2) \simeq \Diff^v(S^1 \times D^2)$.
\end{lem}
\begin{proof}
    We may assume that the fibering is of the standard form $\varphi_{(p,q)} \colon M = (S^1 \times D^2)/(\bbZ/p) \to D^2/(\bbZ/p)$ described above.
    We have subgroup inclusions $\bbZ/p \le \SO(2) \le \Diff(S^1)$ such that the generator of $\bbZ/p$ is sent to the diffeomorphism that rotates by $\zeta_p = e^{2\pi i/p}$.
    We let $\Diff(S^1)^{\bbZ/p}$ denote the fixed points of the conjugation action by $\zeta$, or in other words the centraliser of $\bbZ/p$.
    There is a group homomorphism
    \[
        \alpha\colon \Diff(S^1)^{\bbZ/p} \longrightarrow \Diff^v(M), \qquad
        f \longmapsto \left( [x, y] \mapsto [f(x), y]\right)
    \]
    which is well defined because $[\zeta_p x, \zeta_p^q y] \mapsto [f(\zeta_p x), \zeta_p^q y] = [\zeta_p f(x), \zeta_p^q y]$.
    A $\bbZ/p$-equivariant diffeomorphism $f \in \Diff(S^1)^{\bbZ/p}$ induces a diffeomorphism $\bar{f} \in \Diff(S^1/(\bbZ/p))$, 
    and $\alpha(f)$ is the diffeomorphism of $M$ that is $\bar{f}$ on the central fiber and $f$ on the non-singular fibers.
    In particular, we can define a one-sided inverse to $\alpha$ by restricting a vertical diffeomorphism $g \in \Diff^v(M)$ to any non-singular fiber, and therefore $\alpha$ is an embedding.
    The image of $\alpha$ consists of those $g$ that induce the same diffeomorphism on all non-singular fibers.
    We can construct a deformation retraction to this subgroup by setting
    \[
        H\colon \Diff^v(S^1 \times D^2) \times (0,1] \longrightarrow \Diff^v(S^1 \times D^2), \qquad
        g \longmapsto \left( [x,y] \mapsto [g|_{S^1}([x,ty]), y]\right)
    \]
    and then continuously extending this to $[0,1]$.
    Alternatively $H(g, t)([x,y])$ can be obtained by multiplying the second coordinate of $g([x,ty])$ by $\tfrac{1}{t}$.
    This $H$ defines a homotopy such that $H(-,1)$ is the identity and $H(-,0)$ is a retraction onto the subspace that is the image of $\Diff(S^1)^{\bbZ/p}$.
    
    We will construct a $\bbZ/p$-equivariant deformation retraction from $\Diff(S^1)$ to $\Or(2)$.
    Any orientation-preserving diffeomorphism $f \in \Diff^+(S^1)$ can be represented by a $\bbZ$-equivariant $\tilde{f} \in \Diff^+(\bbR)$ and two such lifts represent the same element of $\Diff^+(S^1)$ if they differ by a constant integer.
    We define the average offset of $\tilde{f}$ as $a = \int_0^1 \tilde{f}(x) - x\; \mathrm{dx}$ and then construct the affine isotopy
    \[
        \tilde{f}_t(x) := (1-t) \cdot \tilde{f}(x) + t \cdot (x + a).
    \]
    This is still $\bbZ$-equivariant and the equivalence class $[\tilde{f}_t] \in \Diff^+(S^1)$ is independent of the choice of $\tilde{f}$.
    Therefore $h\colon [0,1] \times \Diff^+(S^1) \to \Diff^+(S^1)$ given by $h(t,f)= [\tilde{f}_t]$ defines a deformation retraction of $\Diff^+(S^1)$ onto $\SO(2)$, which is equivariant for the $\bbZ/p$-action given by conjugation with $\zeta$.
    (Indeed, the deformation retraction is equivariant for the conjugation action by $\SO(2)$ because it comes from an $\bbR$-equivariant deformation retraction on $\Diff(\bbR)$.)
    The orientation-reversing case is treated similarly.

    In summary, $\Diff(S^1)$ deformation retracts $\bbZ/p$-equivariantly to $\Or(2)$ and passing to fixed points we see that $\Diff(S^1)^{\bbZ/p} \simeq \Diff^v(S^1 \times D^2)$ is equivalent to the space of $\bbZ/p$-fixed points on $\Or(2)$.
    For $p=1$ this is $\Or(2)$.
    For $p=2$ it is also $\Or(2)$ because $\zeta_2 = -1$ is central in $\Or(2)$.
    For $p\ge 3$ (and $0<q<p$) it is $\SO(2)$ because this is the centraliser of $\zeta$ in $\Or(2)$.
\end{proof}

The non-trivial vertical mapping class $\alpha \in \Diff^v(S^1 \times D^2; [\varphi_{(p,q)}])$ that exists for $(p,q) = (2,1)$ or $(1,0)$, can be described as $\alpha([x,y]) = [\overline{x}, y]$ using the complex conjugation on $S^1 \subset D^2 \subset \mathbb{C}$.
For general $(p,q)$ we also have the fiberwise diffeomorphism $\beta \in \Diff^f(S^1\times D^2; [\varphi_{(p,q)}])$ defined by $\beta([x,y]) = [\overline{x}, \overline{y}]$.

\begin{cor}\label{lem:solid-torus-facts}
    For $\varphi\colon M = S^1 \times D^2 \to D = D^2/(\bbZ/p)$ any Seifert-fibering of the solid torus, the following statements hold.
    \begin{enumerate}
        \item\label{it:Diffv-contractible} $\Diff_\partial^v(M)$ is contractible.
        \item\label{it:Difff-contractible} $\Diff_\partial^f(M)$ is contractible.
        \item\label{it:MCGf} $\pi_0\Diff^f(M) \cong \{ [\id], [\alpha], [\beta], [\alpha\beta]\}$ for $p=1,2$ and 
        $\pi_0\Diff^f(M) \cong \{ [\id], [\beta]\}$ for $p\ge 3$.
        \item\label{it:extend} 
        If $\varphi \in \Diff^f(\partial M)$ can be extended to a diffeomorphism on $M$, then it can be extended to a fiberwise diffeomorphism.
    \end{enumerate}
\end{cor}
\begin{proof}
    By \cite[Corollary 3.8(ii)]{HKMR12} there is a locally trivial fiber sequence
    \[
        \Diff_\partial^v(M) \longrightarrow 
        \Diff^v(M) \longrightarrow 
        \Diff^v(\partial M).
    \]
    Since the boundary $\partial M \to \partial D$ of the Seifert fibering is always diffeomorphic to the trivial fibration $S^1 \times S^1 \to S^1$,
    a diffeomorphism in $\Diff^v(\partial M)$ is simply a smooth $S^1$-family in $\Diff(S^1)$.
    Therefore $\Diff^v(\partial M) \simeq \Map(S^1, \Or(2)) \simeq \bbZ \times \Or(2)$.
    Combining this with \cref{lem:Diffv-solid-torus} we see that the map $\Diff^v(M) \to \Diff^v(\partial M)$ is equivalent to either $\SO(2) \to \bbZ \times \Or(2)$ or $\Or(2) \to \bbZ \times \Or(2)$.
    In either case it is a homotopy equivalence onto the components it hits and hence its fiber $\Diff_\partial^v(M)$ is contractible, proving (\ref{it:Diffv-contractible}).

    Claim (\ref{it:Difff-contractible}) is now a consequence of combining (\ref{it:Diffv-contractible}) with the fiber sequence
    \[
        \Diff_\partial^v(M) \to 
        \Diff_\partial^f(M) \to 
        \Diff_\partial(D)
    \]
     from \cite[Theorem 3.9]{HKMR12}. The base and fiber are contractible, and hence so is the total space.

    To show (\ref{it:MCGf}), note that by \cref{lem:Diffv-solid-torus} $\pi_0\Diff^v(M)$ is trivial if $p\ge 3$ and $\{[\id], [\alpha]\}$ if $p \le 2$.
    Inspecting the long exact sequence of the fiber sequence (again from \cite[Theorem 3.9]{HKMR12})
    \[
        \Diff^v(M) \to 
        \Diff^f(M) \to 
        \Diff(D)
    \]
    we see that because $\pi_0\Diff(D) = \bbZ/2$ generated by the (complex) reflection, $\pi_0\Diff^f(M)$ is indeed generated by $\beta$ (and $\alpha$ if $p \le 2$), proving (\ref{it:MCGf}).

    For (\ref{it:extend}), we first note that 
    the map $\pi_0\Diff^f(\partial M) \to \GL_2(\bbZ)$,
    which records the action on first homology, is injective and its image consists of those matrices that preserve the homology class of the fiber up to sign.
    If $\varphi \in \Diff^f(\partial M)$ can be extended to a diffeomorphism of $M$,
    then $\varphi$ it also preserves the homology class of the meridian (which bounds in $M$) up to sign.
    A basis for $H_1(\partial M) \cong \bbZ^2$ is given by the fiber $(p,q)$ and the meridian $(0,1)$.
    Then we have that $\varphi_*(p,q) = \pm (p,q)$ and $\varphi_*(0,1) = \pm(0,1)$.
    If $(p,q) = (1,0)$ or $(2,1)$, then there are four elements of $\GL_2(\bbZ)$ satisfying these equations and if $p \ge 3$, then only $\mathrm{id}$ and $-\mathrm{id}$ satisfy them.
    Therefore $\varphi$ must be isotopic (in $\Diff^f(\partial M)$) to one of $\{\id, \beta_{\partial M}, \alpha_{\partial M}, (\alpha\beta)_{\partial M}\}$, where the latter two are only possible if $p\le 2$.
    In any of these cases, $\varphi$ can be extended to a fiberwise diffeomorphism of $M$.
\end{proof}

In the next subsection, we will prove that when $M$ is fiber-rigid and nonsingular, then $(M \setminus \interior{D}^3, S^2)$ is hereditarily finite. 
When $M$ is singular, the strategy for proving the analogous statement is to cut $M$ along tori that form the boundary of a fibered solid torus neighborhood of the singular fibers and apply \cref{prop:cut-along-tori}.  The complement of these solid tori is fiber-rigid and nonsingular, so in order to verify the hypotheses of \cref{prop:cut-along-tori}, we need the following lemma.

\begin{lem}\label{lem:solid-torus}
    The pair $((S^1 \times D^2 )\setminus \interior{D}^3, S^2)$ is hereditarily finite. 
\end{lem}
\begin{proof}
The restriction mapping $\Diff(S^1\times D^2)\rightarrow \Diff(T^2)$ is a locally trivial fibration over the components of $\Diff(T^2)$ in its image with fiber $\Diff_\partial(S^1\times D^2)$. The latter is contractible by \cite{Hatcher76, Ivanov76}, hence $\Diff_0(S^1\times D^2)\simeq\Diff_0(T^2)\simeq \SO(2)\times \SO(2)$. Since the meridian is the unique simple closed curve on $T^2$ which bounds in $S^1\times D^2$, and $\pi_0\Diff(T^2)\cong \GL_2(\Z)$, the image of $\pi_0\Diff(S^1\times D^2)$ lies the in the subgroup $\GL_2^{\langle e_1\rangle}(\Z)$ which preserves the subgroup $\langle e_1\rangle$:
\[\GL_2^{\langle e_1\rangle}(\Z)=\left\{
        \begin{pmatrix}
            \pm 1 & * \\
            0 & \pm 1
        \end{pmatrix}
        \right\}.\] 
Conversely, the group $\Aff(T^2)\cong\GL_2(\Z)\ltimes (\SO(2)\times \SO(2))$ of affine diffeomorphisms acts on $T^2$, and the subgroup $\GL_2^{\langle e_1\rangle}(\Z)\ltimes(\SO(2)\times \SO(2)) $ extends to $S^1\times D^2$ as a group of affine diffeomorphisms, when equipped with its standard Euclidean structure. Thus, the inclusion $\GL_2^{\langle e_1\rangle}(\Z)\ltimes\SO(2)^2 \hookrightarrow \Diff(S^1\times D^2)$ is a homotopy equivalence.

    We need to show that $\Fr(S^1 \times D^2) \hq \Diff^\Gamma(S^1 \times D^2)$ is homotopy finite for all $\Gamma < \GL_2^{\langle e_1\rangle}(\bbZ)$. By the above diescussion, $\Diff^\Gamma(S^1 \times D^2)\simeq \Gamma \ltimes \SO(2)^2$.
    The frame bundle satisfies
    \( \Fr(S^1 \times D^2) \simeq S^1\times D^2 \times \Or(3) \simeq S^1 \times \Or(3) \).
    The map $p\colon \Fr(S^1 \times D^2) \to \RP^2$ that records the line spanned by the first frame vector at the point $S^1 \times \{0\}$ is invariant under the action of $\GL_2^{\langle e_1\rangle}(\Z) \ltimes \SO(2)^2$.
    The fiber of $p$ is equivalent to $S^1 \times \Or(2) \times \bbZ/2$, so we have a homotopy fiber sequence
    \[
        (S^1 \times \Or(2) \times \bbZ/2) \hq (\Gamma \ltimes \SO(2)^2)
        \longrightarrow \Fr(S^1 \times D^2) \hq \Diff^\Gamma(S^1 \times D^2)
        \longrightarrow \RP^2.
    \]
    The base is finite so we concentrate on the fiber.
    To compute the homotopy orbits we first take the quotient by the identity component $\SO(2)^2$, which acts freely on the first two factors, and get 
    \[
        (S^1 \times \Or(2) \times \bbZ/2) \hq \SO(2)^2 \simeq(S^1 \times \Or(2) \times \bbZ/2) / \SO(2)^2
        \simeq \bbZ/2 \times \bbZ/2.
    \]
    The remaining action of the mapping class group $\GL_2^{\langle e_1\rangle}(\Z)$ is given by the projection $\theta\colon \GL_2^{\langle e_1\rangle}(\Z) = (\bbZ/2)^2 \ltimes \bbZ \to (\bbZ/2)^2$.
    In particular, the subgroup $\bbZ$ acts trivially, and $(\bbZ/2)^2$ acts in the canonical way.
    For any $\Gamma < \GL_2^{\langle e_1\rangle}(\Z)$ we have a fiber sequence
    \[
        B(\Gamma \cap \ker(\theta)) \longrightarrow
        (\bbZ/2)^2 \hq \Gamma \longrightarrow
        (\bbZ/2)^2/\theta(\Gamma).
    \]
    The base is a finite set.
    The kernel $\ker(\theta)$ is $\bbZ$, so $\Gamma \cap \ker(\theta)$ is either infinite cyclic or $0$ -- either way its classifying space is homotopy finite. Applying \cref{lem:fiber-sequence-finite} completes the proof.
\end{proof}

\subsection{The non-singular fiber-rigid case}\label{subsec:nonsingSF}
Suppose $M$ has a non-singular Seifert fibering $\varphi\colon M\rightarrow S$ such that $\Diff^f(M;[\varphi])\hookrightarrow \Diff(M)$ is a homotopy equivalence, \emph{i.e.}~$M$ is fiber rigid (\cref{def:fiberrigid-flexible}). To simplify notation, we write $\Diff^f(M)=\Diff^f(M;[\varphi])$.
By \cite[Corollary 3.1]{HKMR12} we have a principal fibration
\begin{equation}\label{eqn:fix fiber}
\Diff^v(M)\longrightarrow\Diff^f(M)\longrightarrow \Diff(S).
\end{equation}
This is not necessarily a principal short exact sequence (\cref{defn: principal ses}) because  the map $p\colon \Diff^f(M) \to \Diff(S)$ might not be surjective. 
However, its image always has finite index. 
\begin{lem}\label{lem:SF-finite-index}
    The image of $\Diff^f(M) \to \Diff(S)$ has finite index.
\end{lem}
\begin{proof}
    Since $\Diff^f(M) \to \Diff(S)$ is a fibration its image is a union of path components, so we need to check that the image of $\pi_0\Diff^f(M) \to \pi_0\Diff(S)$ has finite index.
    When $\partial M \neq \emptyset$ this follows from the proof of \cite[Lemma 2.2]{HatcherMcCullough} as Hatcher--McCullough prove that $\pi_0\Diff^f_\partial(M) \to \pi_0\Diff_\partial(S)$ is surjective, and we know that the image of $\pi_0 \Diff_\partial(S) \to \pi_0\Diff(S)$ has finite index.
    When $\partial M$ is empty we can introduce fixed boundary by fixing a disk in the base and its preimage in $M$.
    Then Hatcher--McCullough's lemma tells us that the image of 
    $\pi_0\Diff^f_{\varphi^{-1}(D^2)}(M) \to \pi_0\Diff_{D^2}(S)$ 
    has finite index and since every orientation preserving mapping class on $S$ can be isotoped to fix a disk, this proves the claim. 
\end{proof}

\begin{lem}\label{lem:Diff_S1^v-v2}
    Suppose that $\varphi\colon M\rightarrow S$ is a nonsingular Seifert-fibering over a connected surface $S$, and let $F \subset M$ be a non-empty union of boundary components.
    Then $\Diff_{F}^v(M) \simeq \Z^{n}$ for some $n$.
\end{lem}
\begin{proof}
    If $S$ is a disk then $F = \partial M=\partial (S^1\times D^2)$ and $\Diff_\partial^v(S^1 \times D^2) \simeq *$ by \cref{lem:solid-torus-facts}(1).
    We can now induct on the complexity of the surface.
    Let $\alpha$ be an arc in $S$ connecting two points of $\varphi(F)$ such that either $S \setminus \alpha$ is connected and has lower complexity, or is disconnected and both pieces have lower complexity.
    Let $A = \varphi^{-1}(\alpha)$ be the annulus that is the preimage of $\alpha$.
    By \cite[Corollary 3.8(ii)]{HKMR12} there is a fiber sequence
    \[
        \Diff^v_{F \cup A}(M) \longrightarrow \Diff_F^v(M) \longrightarrow \Diff_{F \cap A}^v(A)
    \]
    The fibering of the annulus $A \to \alpha$ is necessarily trivial, so the base of the fiber sequence can be identified with
    $\Diff_\partial^v(S^1 \times I)$,
    an element of which can be thought of as a smooth loop in the space $\Diff(S^1)$.
    Hence it is equivalent to the discrete group
    \[
        \Diff_{F \cap A}^v(A) \simeq \Omega \Diff(S^1) \simeq \bbZ.
    \]
    The fiber of the fiber sequence is equivalent to $\Diff_{F \cup 2A}^v(M \ca A)$ where we cut along $A$ and fix the new boundary created by $A$.
    By induction hypothesis this is equivalent to $\bbZ^n$ for some $n$.
    Inspecting the long exact sequence we see that the components of $\Diff_F^v(M)$ are contractible and that there is an exact sequence
    \[
        1 \to \bbZ^n \to \pi_0\Diff_F^v(M) \to \bbZ,
    \]
    so $\Diff_F^v(M)$ is equivalent to either $\bbZ^n$ or $\bbZ^{n+1}$.
\end{proof}
We can use the previous lemma to take care of the nonsingular Haken Seifert-fibered case.  
\begin{prop}\label{lem:Haken-SF-Nonsingular}
    Suppose $M$ admits a fiber-rigid Seifert fibering without singular fibers.
    Let $D^3\subset M$ be an embedded disk.
    Then $(M \setminus \interior{D}^3, S^2)$ is hereditarily finite and in particular
    $\BDiff_{D^3}(M)$ is homotopy finite.
\end{prop}
\begin{proof}
    Let $\varphi \colon M \to S$ be the non-singular Seifert fibering such that $\Diff^f(M) \simeq \Diff(M)$.
    We need to show that for all $\Gamma < \pi_0\Diff(\partial M)$ the space
    \[
        \BDiff_{D^3}^\Gamma(M) \simeq \Fr(M)\hq \Diff^\Gamma(M) \simeq \Fr(M) \hq \Diff^{f, \Gamma}(M)
    \]
    is homotopy finite.
    When acting by $\Diff^f(M)$ on the frame bundle, the fiber-direction of the tangent space is always preserved.
    Therefore the map
    \begin{align*}
        \Fr(M) & \longrightarrow \RP^2 \\ 
        (p, \alpha \colon \bbR^3 \cong T_pM) &\longmapsto \alpha^{-1}(\ker d_p\varphi) \subset \bbR^3
    \end{align*}
    is invariant under the action of $\Diff^f(M)$.
    Let $\Fr'(M) \subset \Fr(M)$ denote the fiber of this map, \emph{i.e.}~those framings where the first vector lies in the tangent space of the fiber. By \cref{lem:fiberseq/fiberseq} we have a homotopy fiber sequence
    \[
        \Fr'(M)\hq \Diff^{f,\Gamma}(M) \longrightarrow
        \Fr(M)\hq \Diff^{f,\Gamma}(M) \longrightarrow
        \RP^2
    \]
    and we show that the total space is homotopy finite.
    Since the base is a finite CW complex, by \cref{lem:fiber-sequence-finite} it will suffice to consider the fiber.
    The map $\Fr'(M) \to \Fr(S)$ defined by forgetting the first vector of the framing and then applying $\varphi$ and $d\varphi$ is a fibration, and its fiber $\Fr''(S^1)$ is equivalent to the frame bundle of $S^1$.
    This fiber sequence 
    \begin{equation}\label{eqn:Fr-fibration}
        \Fr''(S^1) \longrightarrow \Fr'(M) \longrightarrow \Fr(S)
    \end{equation}
    is equivariant for the principal short exact sequence
    \begin{equation}\label{eqn:Diff-fibration}
        \Diff^{v,\Gamma}(M) \longrightarrow
        \Diff^{f,\Gamma}(M) \longrightarrow
        \Diff'(S).
    \end{equation}
    that we obtain by restricting \cref{eqn:fix fiber} to account for $\Gamma$.
    Here $\Diff'(S) \le \Diff(S)$ is defined as the image of $\Diff^{f,\Gamma}(M)$, and this is a finite index subgroup by \cref{lem:SF-finite-index}.
    
    Taking the quotient of (\ref{eqn:Fr-fibration}) by (\ref{eqn:Diff-fibration}), \cref{lem:fiberseq/fiberseq} yields a homotopy fiber sequence
    \[
        \Fr(S^1)\hq \Diff^{v,\Gamma}(M) \longrightarrow
        \Fr'(M)\hq \Diff^{f,\Gamma}(M) \longrightarrow
        \Fr(S)\hq \Diff'(S).
    \]
    The base is equivalent to $\BDiff_{D^2}'(S)$, which is a finite cover of $\BDiff_{D^2}(S)$, which in turn is homotopy finite by Lemma \ref{lem:Surface-Fix-Disk-finite}, so it suffices to consider the fiber.
    
    The action of $\Diff^v(M)$ on $\Fr(S^1)$ has either one or two orbits, depending on whether there is a bundle-preserving diffeomorphism that reverses the orientation of the fiber.
    Either way the map
    \[
        \Diff^v(M) \longrightarrow \Diff(S^1) \xrightarrow{\ \simeq\ } \Fr(S^1),
    \]
    which first restricts the diffeomorphism to a fiber and then evaluates it and its derivative at a point, is a fiber bundle by \cite[Corollary 3.8.(ii)]{HKMR12} (furthermore the second map is a trivial Serre fibration). Hence the homotopy orbit-stabiliser lemma (\cref{lem:orbit-stabiliser-v2}) tells us
    \[
        \Fr(S^1)\hq \Diff^{v, \Gamma}(M) \simeq \Diff(S^1)\hq \Diff^{v, \Gamma}(M) \simeq \BDiff_{S^1}^{v, \Gamma}(M) \times \begin{cases}
            * \\
            * \sqcup * 
        \end{cases}
    \]
    where the case distinction depends on whether there exists fiberwise orientation-reversing diffeomorphism.
    For $U\subset M$ a vertical tubular neighbourhood of $S^1$ we have 
    \[
        \Diff_{S^1}^v(M) \simeq \Diff_U^v(M) \simeq \Diff_{\partial U}^v(M\setminus \interior{U}) \simeq \bbZ^n
    \]
    where the last equivalence comes from \cref{lem:Diff_S1^v-v2}.
    Hence the group $\Diff_{S^1}^{v,\Gamma}(M)$ must be equivalent to some subgroup of $\bbZ^n$, so it also is finite rank free abelian and its classifying space is homotopy finite.
\end{proof}

\subsection{The singular fiber-rigid case}\label{subsec:singSF}
Now we prove the general case with possibly singular fibers. 
The method of proof will be to cut up the manifold along tubular neighbourhoods $U$ of the singular fibers using \cref{prop:cut-along-tori} in order to reduce to the non-singular case and the case of a fibered solid torus, both of which we have already covered.
For this we have to show that the tori $T := \partial U$ obtained as the neighbourhoods of the singular fibers are canonical in the sense of \cref{defn:canonical}, \emph{i.e.}~that $\Diff(M, T) \simeq \Diff(M)$.

\begin{lem}\label{lem:singular-fiber-canonical}
    Let $M$ be a Haken Seifert-fibered 3-manifold with a fiber-rigid fibering $\varphi\colon M \to S$.
    Let $U$ be a vertical tubular neighbourhood of the singular fibers, and let $T := \partial U$ be its boundary. 
    
    Then $M \setminus \interior{U} \to \varphi(M \setminus \interior{U})$ is fiber-rigid,
    $\Diff(M, U) = \Diff(M, T)$, and the inclusions
    \[  
        \Diff(M, T) \hookrightarrow \Diff(M) \text{ and } \Diff^f(M, T)\hookrightarrow \Diff(M,T)
    \]
    are both equivalences.
    It follows from the first equivalence that $T$ is a canonical submanifold in the sense of \cref{defn:canonical}.
\end{lem}
\begin{proof}
    $M$ cannot be a solid torus, because the solid torus only has flexible fiberings. 
    Because $M$ is Haken we know either $S$ is not a sphere or it is a sphere and there are at least three singular fibers.
    Therefore $S \setminus \varphi(U)^\circ$ is neither an annulus, nor a disk. 
    By \cref{rem:most-non-singular-are-rigid} the only flexible Seifert-fibered manifolds with boundary are the solid torus and $T^2 \times I$, which only fiber over the disk and the annulus.
    Hence $M \setminus \interior{U} \to S \setminus \varphi(U)^\circ$ is a non-singular fiber-rigid Haken Seifert-fibering.
    Moreover, $U$ is a disjoint union of solid tori and $M \setminus \interior{U}$ is connected and not a solid torus, so any diffeomorphism preserving $T$ must also preserve $U$, \emph{i.e.}~$\Diff(M,T)\simeq \Diff(M,T)$.

    Since the diffeomorphism group of the base orbifold already preserves the singularities, it is equivalent to the subgroup $\Diff(S, \varphi(U)) < \Diff(S)$ where we additionally preserve small disks around the singularities.
    Pulling this equivalence back along the fibration $\Diff^f(M) \to \Diff(S)$ from \cref{eqn:fix fiber} we get that $\Diff^f(M, T) \simeq \Diff^f(M)$.
    This equivalence fits in a square of subgroup inclusions
    \[\begin{tikzcd}
    	{\Diff^f(M, T)} & {\Diff^f(M)} \\
    	{\Diff(M, T)} & {\Diff(M)}
    	\arrow[from=1-1, to=2-1]
    	\arrow[from=2-1, to=2-2]
    	\arrow["\simeq", from=1-1, to=1-2]
    	\arrow["\simeq", from=1-2, to=2-2]
    \end{tikzcd}\]
    and to show that the bottom map is an equivalence as claimed it will suffice to check our final claim: that the left hand map is an equivalence.
    The left map in turn fits into a map of homotopy fiber sequences:
\[\begin{tikzcd}
	{\Diff_\partial^f(U)} & {\Diff^f(M, T)} & {\Diff^f(M \setminus \interior{U})} \\
	{\Diff_\partial(U)} & {\Diff(M, T)} & {\Diff(M \setminus \interior{U})}
	\arrow["\simeq"', from=1-1, to=2-1]
	\arrow[from=1-2, to=2-2]
	\arrow["\simeq", from=1-3, to=2-3]
	\arrow[from=2-2, to=2-3]
	\arrow[from=1-2, to=1-3]
	\arrow[from=1-1, to=1-2]
	\arrow[from=2-1, to=2-2]
\end{tikzcd}\]
    The right vertical map is an equivalence because $M\setminus \interior{U}$ is fiber-rigid.
    The left vertical map is an equivalence because $\Diff_\partial(S^1 \times D^2)$ is contractible by \cite{Hatcher76, Ivanov76} and $\Diff_\partial^f(U)$ is contractible by \cref{lem:solid-torus-facts}.(2).
    From the long exact sequence of homotopy groups it follows that the middle map is an isomorphism on $\pi_n$ for $n\ge1$ and an injection on $\pi_0$.
    To see that it is also surjective on $\pi_0$ let $\xi\in \Diff(M, T)$.
    Because $M\setminus \interior{U}$ is fiber-rigid we can isotope $\xi_{|M \setminus \interior{U}}$ to a fiber-preserving diffeomorphism, and we may extend this isotopy to all of $M$ while preserving $T$ as a subset. Denote the result of this isotopy by $\varphi\in \Diff(M, T)$.
    Now $\varphi_{|U}$ is fiber-preserving on $\partial U$ and hence, by \cref{lem:solid-torus-facts}.(3), there is a $\psi \in \Diff^f(U)$ with $\varphi_{|\partial U} = \psi_{|\partial U}$.
    We have $\varphi_{|U} \circ \psi^{-1} \in \Diff_\partial(U) \simeq *$, so there must be an isotopy from $\varphi_{|U}$ to $\psi$ relative to $\partial U$.
    This shows that $\xi$ is isotopic to an element of $\Diff^f(M, T)$.
\end{proof}

\begin{prop}\label{prop:Haken-SF-Finiteness}
    Let $M$ be a Haken Seifert-fibered 3-manifold that is fiber-rigid, and let $D^3\subset M$ be an embedded disk. 
    Then $(M \setminus \interior{D}^3, S^2)$ is hereditarily finite and in particular $\BDiff_{D^3}(M)$ is homotopy finite.
\end{prop}
\begin{proof}
    If $M$ is non-singular apply \cref{lem:Haken-SF-Nonsingular}. Otherwise, pick a vertical tubular neighbourhood $U$ of the singular fibers as in \cref{lem:singular-fiber-canonical}; then $T := \partial U$ is a canonical submanifold by that lemma.
    We apply \cref{prop:cut-along-tori}.
    The second condition follows from \cref{thm:HM-hereditary}.
    For the third condition, \cref{lem:singular-fiber-canonical} also shows that $\Diff(M, T) \simeq \Diff^f(M, T)$ and 
    so the image of $\Diff(M, T) \to \pi_0 \Diff(T)$ is finite because each of the tori bounds a fibered solid torus and $\pi_0 \Diff^f( U )$ is finite by \cref{lem:solid-torus-facts}(4).
    
    Hence \cref{prop:cut-along-tori} reduces the claim to hereditary finiteness for $(N \setminus \interior{D}^3, S^2)$ where $N$ is any connected component of $M \ca T$.
    Most of the components are solid tori and we checked in \cref{lem:solid-torus} that $((S^1 \times D^2 )\setminus \interior{D}^3, S^2)$ is hereditarily finite.
    The remaining connected component is the non-singular piece $M \setminus \interior{U}$.
    As observed in \cref{lem:singular-fiber-canonical} $M \setminus \interior{U}$ is non-singular Seifert-fibered and still fiber-rigid.
    As such $(M \setminus (U \sqcup D^3)^\circ, S^2)$ satisfies hereditary finiteness by \cref{lem:Haken-SF-Nonsingular}.
\end{proof}

\subsection{Proof of Theorem \ref{thm:sect4main}}

We will now deduce the general theorem of the finiteness of $\BDiff_{D^3}(M)$ for $M$ irreducible, by cutting $M$ along its JSJ decomposition.

That the JSJ tori are canonical in sense of \cref{defn:canonical} will follow from the following theorem of Hatcher~\cite{Hatcherincompressible} that can be found in a rewrite of his 1976 paper \cite{Hatcher76}.

\begin{thm}[Hatcher {\cite[Theorem 1(a)]{Hatcherincompressible}}]\label{thm:Hatcher-Incompressible}
    Let~$T$ be an incompressible torus in~$M$. Then the component of the unparametrised embedding space containing $T$, $\umb_{T}(T^2,\interior{M})$, is contractible unless~$M$ is a $T^2$-bundle over~$S^1$ and~$T$ is a fiber, in which case $\umb_T(T^2,\interior{M})\simeq S^1$.
\end{thm}

\begin{cor}\label{cor:JSJ-canonical}
    Suppose $M$ is not a $T^2$-bundle over $S^1$. 
    Let $T$ be a collection of tori in $M$ giving the JSJ decomposition.
    Then $\Diff(M,T)\simeq \Diff(M)$, i.e.~$T$ is a canonical submanifold of $M$.
\end{cor}
\begin{proof}
    For any union of components $R \subset T$ we let $\umb_0(R, \interior{M}) \subset \umb(R, \interior{M})$ denote the path-component of the space of those unparametrised embeddings that contain the submanifold $R \subset \interior{M}$.
    We will prove that $\umb_0(R, \interior{M})$ is contractible by induction on the number of connected components of $R \subset T$.
    If there is only one torus, then this is Hatcher's theorem \ref{thm:Hatcher-Incompressible}.
    For the induction step, we suppose that $\umb_0(R, \interior{M})$ is contractible. We will show that $\umb_0(R \sqcup A, \interior{M})$ is contractible for some torus $A \subset T \setminus N$.
    We can define a map
    \[  
        f\colon \umb_0(R \sqcup A, \interior{M}) \to \umb_0(R, \interior{M})
    \]
    by taking a submanifold $V \subset \interior{M}$ to $V \setminus B \subset \interior{M}$ where $B \subset V$ is the unique connected component that is isotopic to $A \subset \interior{M}$.
    Because we are in the path component $\umb_0$, we know that there always is such a $B$, and because no two tori in $T$ are isotopic, it is unique.
    Therefore, $f$ is well-defined, continuous, and $\Diff_0(M)$-equivariant.
    But $\umb_0(R, \interior{M})$ is $\Diff_0(M)$-locally retractile by \cref{cor:umb-retractile}, so $f$ is a fiber bundle by \cref{lem:locally-retractile}(1).
    We can identify the fiber $f^{-1}([R])$ with the identity component of unparametrised embeddings of $A$ into $(M\setminus R)^\circ$.
    (Indeed, the long exact sequence of the fibration implies that the fiber is connected, as the total space is connected and the base contractible by the induction hypothesis.)
    Hence we have the fiber sequence
    \[
        \umb_0(A, (M \setminus R)^\circ) \longrightarrow 
        \umb_0(R \sqcup A, \interior{M}) \longrightarrow
        \umb_0(R, \interior{M}).
    \]
    By induction hypothesis both fiber and total space are contractible and we conclude that the total space must also be contractible, completing the induction.

    To obtain the conclusion about $\Diff(M, T)$ we consider the fiber sequence
    \[
        \Diff(M, T) \longrightarrow \Diff(M) \longrightarrow \umb(T, \interior{M})
    \]
    from the discussion after \cref{cor:umb-retractile}.
    The image of the fibration is exactly the $\Diff(M)$-orbit of the canonical unparametrised embedding $T \subset \interior{M}$.
    By the uniqueness of JSJ decompositions we know that for any $\varphi \in \Diff(M)$ we have that $\varphi(T)$ is (unparametrised) isotopic to $T$,
    so the image of the fibration is the path-component of the base point, \emph{i.e.}~$\umb_0(T, \interior{M})$, which we know to be contractible.
    Hence the inclusion $\Diff(M, T) \to \Diff(M)$ is an equivalence as claimed.
\end{proof}

We will use the fact that the JSJ tori are canonical to reduce the finiteness problem to pieces we have already studied, by applying \cref{prop:cut-along-tori}. However, we need to guarantee that hypothesis (3) of the proposition holds.  This is the purpose of the next lemma.

\begin{lem}\label{lem:JSJ-finite-MCG}
    Let $M$ be an irreducible $3$-manifold with non-trivial JSJ decomposition $T \subset \interior{M}$ such that no component of $M \ca T$ is $T^2 \times I$.
    Then the image of $\Diff(M, T) \to \pi_0 \Diff(T)$ is finite.
\end{lem}
\begin{proof}
    First, we restrict to the finite index subgroup of $\Diff(M,T)$ consisting of those diffeomorphisms that fix the dual graph of $T$. It is enough to show that for this subgroup the lemma holds. To do this, we claim that for each torus $T_i \subseteq T$ there are a finite number of mapping classes on $T_i$ that can be extended to an element of $\Diff(M, T)$. Let $M_i'$ and $M_i''$ be the JSJ pieces on either side of $T_i$ (note that $M_i'$ may equal $M_i''$).
    First, consider the case where at least one of $M_i'$ or $M_i''$ is a hyperbolic manifold $N$.
    Then, since the dual graph is fixed, the map in question factors through $\Diff(N, T_i)$, but we know that $\pi_0\Diff(N) = \pi_0 \Isom(N)$ is finite, so the image is finite.
    
    If neither of the sides of $T_i$ is hyperbolic, then they must both be Seifert-fibered. Neither side is $D^2\times S^1$  nor $T^2\times I$, hence they are both fiber-rigid by Remark \ref{rem:most-non-singular-are-rigid}. Let $\alpha_i'$ and $\alpha_i'' \in H_1(T_i)$ be the images of the two fiber curves  on $T_i$. Since $M_i'\cup_{T_i}M_i''$ is not Seifert-fibered, $\alpha_i'$ and $\alpha_i''$ represent linearly independent elements of $\pi_1(T_i)$. Any element of $\Diff(M,T)$ that fixes the dual graph of $T$ thus maps both $\alpha_i'$ and $\alpha_i''$ to themselves or their inverses. As $\alpha_i'$ and $\alpha_i''$ are linearly independent, there at most four elements of $\GL_2(\Z)\cong \pi_0(\Diff(T_i))$ that satisfy this condition.
\end{proof}

We can now complete the proof of \cref{thm:sect4main}, which we restate here.
\begin{rerestate}{Theorem}{thm:sect4main}{restating:sect4main}
    Let $M$ be an irreducible 3-manifold with either empty or incompressible toroidal boundary, and let $D^3\subset \interior{M}$ be an embedded disk. 
    Then $\BDiff_{D^3}(M)$ has the homotopy type of a finite CW complex.
\end{rerestate}
\begin{proof}
     We refer the reader to the flow chart in \cref{fig:flowchartirreducible}. Let $M$ be irreducible. If $M$ has trivial JSJ decomposition, then $M$ is Seifert-fibered or hyperbolic.
    If $M$ is  non-Haken Seifert-fibered or hyperbolic, then $\BDiff_{D^3}(M)$ is homotopy finite by Corollary \ref{cor:NonHaken-and-Hyperbolic}. If $M$ is Haken Seifert-fibered then $\BDiff_{D^3}(M)$ is homotopy finite by Proposition \ref{prop:Haken-SF-Finiteness} in the fiber-rigid case, and Propositions \ref{prop: finiteness sphere or torus x I}, \ref{prop:T3-Finiteness}, and \ref{prop:Hantzsche-Wendt} in the flexible case.

    Now assume $M$ has non-trivial JSJ decomposition. If $M$ is a  $T^2$-bundle over $S^1$ then $M$ must be Anosov (all others are Seifert-fibered), in which case $\BDiff_{D^3}(M)$ is finite by Theorem \ref{thm:Anosov-T2-bundles}.  We may therefore suppose that~$M$ has non-trivial JSJ decomposition and is not the total space of a $T^2$-bundle over $S^1$. Let $T \subset M$ be the collection of JSJ tori, which is a canonical submanifold by \cref{cor:JSJ-canonical}. To prove the theorem, we verify that \cref{prop:cut-along-tori} applies,  which (by \cref{lem:JSJ-finite-MCG}) reduces to checking the hereditary finiteness of $(N, F)$ and $(N \setminus \interior{D}^3, S^2)$ for any connected component $N \subset M \ca T$ and any non-empty union of boundary components $F \subset \partial N$. 
    The former was shown by Hatcher--McCullough \cite{HatcherMcCullough}; see \cref{thm:HM-hereditary}.
    For the latter, since each $N$ is either hyperbolic or Haken Seifert-fibered and fiber-rigid, hereditary finiteness of $(N \setminus \interior{D}^3, S^2)$ holds by \cref{thm:Isom-Equivalence} and \cref{lem:Riemannian-hereditarily-finite} in the hyperbolic case or  \cref{prop:Haken-SF-Finiteness} otherwise.
\end{proof}

\section{Kontsevich's finiteness conjecture and consequences} \label{section: finiteness of BDiff}
In this section we prove Kontsevich's finiteness conjecture \cite[Problem 3.48]{KirbyProblems} -- this is the statement of the following theorem. 
As a consequence in \cref{subsection - finite type} we show that $\BDiff(M)$ is always of finite type and in \cref{subsection - sharpness} we discuss the sharpness of our results.

\begin{thm}\label{thm:Kontsevich-finiteness-conjecture}
    Let $M$ be a compact, connected, orientable $3$-manifold with non-empty boundary~$\partial M$.
    Then 
        $\BDiff_\partial(M)$
    is {homotopy finite}.
\end{thm}

In fact, we will prove the following more general statement.

\begin{figure}
    \centering
    \resizebox{0.6\textwidth}{!}{
\begin{tikzpicture}[node distance=7mm and 7mm]
	\tikzstyle{box} = [rectangle, draw, text width=14em, text centered, rounded corners, minimum height=3em]
    \tikzstyle{leafbox} = [rectangle, draw, text width=13em, text centered, rounded corners, minimum height=3em, fill=orange!10]
	\tikzstyle{thmbox} = [text width=8em, text centered, minimum height=2em]
	\tikzstyle{decoration}=[draw=none, anchor=west, font=\small]
	\tikzstyle{arrow}=[->]
	 
        \node (main thm) [box] {\cref{thm:more general Kontsevich-finiteness-conjecture}:\\$\BDiff_F(M)$ homotopy finite  } ;
        \node (sep contractible) [thmbox, below= of main thm] {\cref{thm: sepNP contractible} };
        \node (new model) [box, below= of sep contractible] {$|\sepNP_\bullet(M)| \hq\Diff_F(M)$ homotopy finite} ;
        \node (lemmas) [thmbox, below= of new model] {\cref{lem:face-map-is-finite-covering}, \cref{lem:orbit-stabiliser-v2}} ;
        \node (stabilisers) [box, below left = and -10mm of lemmas] {$\BDiff_F(M, \Sigma)$ homotopy finite \\ \cref{prop:finiteness-sphere-system}};
        \node (cut spheres) [thmbox, below = of stabilisers] {\cref{cor:cut-along-spheres}};
        \node (orbits) [leafbox, below right = and -10mm of lemmas] {$\sepNP(M)/\Diff_F(M)$ finite \\ \cref{lem:finitely-many-graphs-HM}};
	\node (m hat irreducible) [box, below= of cut spheres] {$\widehat{M}$ irreducible \\ \cref{thm:prime-finiteness}} ;
	\node (thm inductive) [thmbox,below= of m hat irreducible] {\cref{lem:finiteness-fixing-balls}} ;
	\node (m irreducible) [leafbox,below left=and -10mm of thm inductive] {$M$ irreducible \\ \cref{thm:HM-hereditary}} ;
	\node (m minus ball) [leafbox,below right= and -10mm of thm inductive] {$M=\widehat{M}\setminus D^3$\\ $F=\partial D^3$, $\widehat{M}$ irreducible \\
 \cref{thm:sect4main}} ;
	
	\draw[] (m hat irreducible) -- (thm inductive);
	\draw[arrow]  (thm inductive) -- (m minus ball);
	\draw[arrow]  (thm inductive) -- (m irreducible);
        \draw[] (main thm) -- (sep contractible);
	\draw[arrow]  (sep contractible) -- (new model);
	\draw[]  (new model) -- (lemmas);
        \draw[arrow] (lemmas) -- (stabilisers);
        \draw[arrow] (lemmas) -- (orbits);
	\draw[]  (stabilisers) -- (cut spheres);
        \draw[arrow] (cut spheres) -- (m hat irreducible);
\end{tikzpicture}
}
    \caption{Flowchart of the proof of \cref{thm:more general Kontsevich-finiteness-conjecture}. Note that the bottom central leaf is the root vertex of the flow chart in \cref{fig:flowchartirreducible} so gluing the two flow charts together here gives a full flow chart for the main theorem. }
    \label{fig:flowchartmain}
\end{figure}
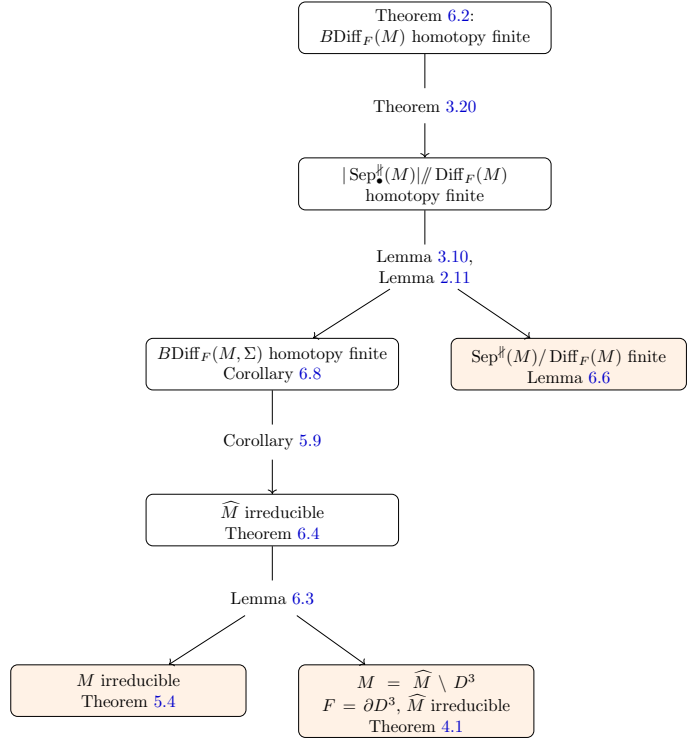
\begin{thm}\label{thm:more general Kontsevich-finiteness-conjecture}
    Let $M$ be a compact, connected, orientable $3$-manifold, and let $\emptyset \neq F \subseteq \partial M$ be a non-empty union of connected components.
    Assume that $\partial M \setminus F$ consists of spheres and incompressible tori.
    Then $\BDiff_F(M)$ is homotopy finite.
\end{thm}

As in the proof of \cref{thm:sect4main}, we have outlined the key steps of the proof as a directed tree in \cref{fig:flowchartmain} for the reader to follow. Each node is labeled by a statement or theorem. 
The basic outline for the proof of \cref{thm:more general Kontsevich-finiteness-conjecture} is as follows.  By \cref{thm: sepNP contractible}, $\|\sepNP_\bullet(M)\| \hq\Diff_F(M)$ is a model for $\BDiff_F(M)$ and we show that only simplices below a certain dimension contribute. Combining this with \cref{lem:face-map-is-finite-covering} we reduce to the claim that $\sepNP(M)\hq \Diff_F(M)$ is homotopy finite.
By the homotopy orbit-stabiliser theorem (\cref{lem:orbit-stabiliser-v2}), this amounts to showing that there are finitely many orbits (\cref{lem:finitely-many-graphs-HM}) and that the classifying space of each orbit stabiliser $\Diff_F(M,\Sigma)$ is homotopy finite (\cref{prop:finiteness-sphere-system}). We introduce the notion of a labeled dual graph $G_\Sigma$ to a separating system $\Sigma$ and use this to show that the set of orbits is finite. That the classifying space of the stabiliser $\Diff_F(M,\Sigma)$ is homotopy finite is further reduced (\cref{cor:cut-along-spheres}), by inductively cutting along 2-spheres in $\Sigma$, to homotopy finiteness of $\BDiff_F(M)$ where $\scl{M}$ is irreducible (\cref{thm:prime-finiteness}). The latter in turn is reduced, by repeatedly filling spheres (\cref{lem:finiteness-fixing-balls}), to our two base cases: when $M$ is itself irreducible (Hatcher--McCullough, \cref{thm:HM-hereditary}), or when $\scl{M}$ is obtained from $M$ by filling a single boundary sphere and $F$  consists of only this sphere (\cref{thm:sect4main}).

\subsection{The case of one irreducible factor}
As recalled in the introduction, Hatcher and McCullough prove Theorem \ref{thm:Kontsevich-finiteness-conjecture} when $M$ is irreducible, and in Theorem~\ref{thm:sect4main} we prove Theorem \ref{thm:Kontsevich-finiteness-conjecture} when $\scl{M}$ is irreducible and $M$ has one spherical boundary component.
However, before we use the space of separating systems to deduce the general case from this, we first have to consider the case where $M$ is allowed to have multiple spherical boundary components.

First we note that cutting out additional disks and fixing or not fixing their boundary preserves homotopy finiteness of the classifying space.
\begin{lem}\label{lem:finiteness-fixing-balls}
    Suppose that $M$ is a connected 3-manifold and $F \subseteq \partial M$ is a union of connected components of its boundary.
    Let $B \subset \interior{M}$ be a finite union of disjoint disks.
    Then we have equivalences 
    \[
        \BDiff_{F \sqcup B}(M) \simeq
        \BDiff_{F \sqcup \partial B}( M\setminus \interior{B} ) 
        \quad \text{ and } \quad 
        \BDiff_F(M, B)
        \simeq \BDiff_F(M \setminus \interior{B}, \partial B) .
    \]
    Suppose moreover that $\BDiff_F(M)$ is homotopy finite, then so are the above spaces.
\end{lem}
\begin{proof}
    The first equivalence holds because the space of collars is contractible, see \cref{thm:contractible-collars}.
    For the second equivalence we use the fiber sequence
    \[
        \Diff_{M\setminus \interior{B}}(M) \longrightarrow \Diff_F(M, B) \longrightarrow \Diff_F(M\setminus \interior{B}, \partial B).
    \]
    Any diffeomorphism of $\partial B$ can be extended over the interior, so the fibration is surjective.
    The group in the fiber is equivalent to $\Diff_\partial(B)$, which is a product of $\Diff_\partial(D^3)$'s, hence  contractible by Hatcher's proof of the Smale conjecture.
    This shows that the fibration is an equivalence, which gives us the second equivalence.

    To prove the homotopy finiteness, we note that $\Diff_F(M)$ acts transitively on the space of orientation preserving embeddings $\emb^+(B, \interior{M})$ and unparametrised embeddings $\umb(B, \interior{M})$ with stabiliser groups $\Diff_{F\sqcup B}(M)$ and $\Diff_F(M, B)$, respectively.
    Both actions are locally retractile, so we can apply \cref{lem:orbit-stabiliser-transitive} to obtain homotopy fiber sequences:
    \begin{align*}
        \umb(B, M) \longrightarrow
        & \BDiff_F(M, B) \longrightarrow \BDiff_F(M)\\
         \emb^+(B, M) \longrightarrow
        &\BDiff_{F\sqcup B}(M) \longrightarrow \BDiff_F(M). 
    \end{align*}
    Since we assumed that the base spaces are homotopy finite in both cases, it will suffice to argue that the fibers a homotopy finite.
    We pass to considering 
    The space $\emb(B, M)$ since $\emb^+(B, M)$ is a union of path components. Then
    $\emb(B, M) = \emb(\amalg_k D^3, M)$ is homotopy equivalent to $\Conf_{k}^{\rm fr}(M)$, the space of ordered framed configurations of $M$. 
    This space is homotopy finite, as can be seen by inductively using the homotopy fiber sequence
    \[
        \Fr(M \setminus \{p_1, \dots, p_{k-1}\}) \longrightarrow
        \Conf_{k}^{\rm fr}(M) \longrightarrow
        \Conf_{k-1}^{\rm fr}(M \setminus \{p\}) 
    \]
    and the fact that the frame bundle of a punctured version of $M$ is still homotopy finite.
    
    The space $\umb(B, M)$ is obtained from $\emb(B, M)$ by taking the quotient with respect to the $\Diff(B)$ action.
    Because $\Diff(D^3) \simeq \SO(3)$, we get that $\umb(B, M)$ is homotopy equivalent to the unordered configuration space of $k$ points in $M$. 
    This is equivalent to the quotient of the free strata-preserving action of $\Sigma_k$ on the Fulton-MacPherson compactification $C_k[M]$ \cite[Theorem 4.10]{Sinha04}.
    As such it is itself a compact manifold with corners and hence homotopy finite.
\end{proof}

\begin{thm}\label{thm:prime-finiteness}
    Suppose that $M$ is such that the manifold $\scl{M}$ obtained by filling all spherical boundary components is irreducible. Let $F \subseteq \partial M$ be a non-empty union of boundary components such that $\partial M \setminus F$ contains only spheres and incompressible tori.
    Then $\BDiff_F(M)$ is homotopy finite.
\end{thm}

\begin{proof}
    We can write $M = \widehat{M}\setminus \interior{B}$ for $B$ a disjoint union of disks. 
    Decompose $B = B_1 \sqcup B_2$ such that $F \cap \partial B = \partial B_1$
    and write $F_0 = F \cap \partial \scl{M}= F \setminus \partial B_1$.

    There are two cases. 
    If $F_0 \neq \emptyset$ we will reduce to \cref{thm: HatcherMcCullough finiteness} of Hatcher--McCullough \cite{HatcherMcCullough}, who show that $\Diff_{F_0}(\widehat{M})$ is homotopy finite.
    Applying \cref{lem:finiteness-fixing-balls} once, we see that $\BDiff_{F_0 \sqcup \partial B_1}(\scl{M}\setminus \interior{B}_1)$ is homotopy finite.
    Applying it again we get that
    \[
        \BDiff_{F_0 \sqcup \partial B_1}(\scl{M}\setminus \interior{B}, \partial B_2)
        =
        \BDiff_{F}(M, \partial B_2)
        = \BDiff_F(M)
    \]
    is homotopy finite.
    Here the first equality holds because $F_0 \sqcup \partial B_1 = F$ and $\widehat{M}\setminus \interior{B} = M$.
    The second equality holds because any diffeomorphism of $M$ must fix $\partial B$ as a subset (because these are the spherical boundaries) and so if it fixes $\partial B_1$ pointwise, it must automatically fix $\partial B_2$ as a subset.

    If $F_0 = \emptyset$, then $  \partial B_1=F \neq \emptyset$ must be non-empty.
    In this case we reduce to \cref{thm:sect4main}.
    Pick some preferred component $D \subset B_1$.
    Then \cref{thm:sect4main} tells us that $\BDiff_D(\scl{M})$
    is homotopy finite.
    Applying \cref{lem:finiteness-fixing-balls} twice gives firstly that $\BDiff_{B_1}(\scl{M})$ is homotopy finite and then that
    \[
        \BDiff_{B_1}(\scl{M}\setminus \interior{B}_2, \partial B_2) 
        \simeq \BDiff_{\partial B_1}(\scl{M}\setminus \interior{B}, \partial B_2) 
        = \BDiff_{\partial B_1}(\scl{M}\setminus \interior{B})
        = \BDiff_{F}(M)
    \]
    is also homotopy finite.
\end{proof}

\subsection{Deducing the general case using sphere systems}

In the final step of the proof, we use contractibility of $\|\sepNP_\bullet(M)\|$, and our results for irreducible manifolds, possibly with disks removed, to prove Kontsevich's finiteness conjecture. We require two more ingredients, relating to the action of $\Diff_F(M)$ on $\sepNP(M)$. First, we show that this action has a finite number orbits, and second, that the stabiliser of a separating system $\Sigma \in \sepNP_\bullet(M)$ is homotopy finite.

\begin{lem}\label{lem:bounding-graphs}
    Let $M$ be a $3$-manifold with prime decomposition
    $$M \cong ((S^1 \times S^2)^{\# g} \# P_1 \# \dots \# P_n) \setminus (\amalg_m \interior{D}^3)$$ 
    where the $P_i$ are irreducible, and assume that $M$ is not prime.
    Then every $\Sigma \in \sepNP(M)$ consists of at most $2(n+m) + 3(g-1)$ spheres and $M \ca \Sigma$ has at most $2(n+m+g-1)$ connected components.
\end{lem}
\begin{proof}
    Since $M$ is not prime, every $\Sigma \in \sepNP(M)$ is non-empty.
    We can construct a dual graph $G_\Sigma$ that has a vertex for each component of $M \ca \Sigma$ and an edge for each sphere in $\Sigma$.
    Let $v$ be the number of vertices and $e$ the number of edges in this graph.
    This dual graph must have first Betti number $g$ and therefore it has Euler characteristic $v-e = 1-g$.
    Let us say that a vertex is \emph{special} if it corresponds to a component $N \subset M \ca \Sigma$ that contains a prime factor or a boundary sphere of $M$, and let $v_0$ be the number of special vertices.
    Then we have $v_0 \le n+m$, as there are $n$ prime factors and $m$ boundary spheres in $M$.
    A vertex is non-special if the manifold obtained from $N \subset M \ca \Sigma$ by filling in the boundary spheres coming from $2\Sigma$ is a $3$-sphere.
    Every non-special vertex is at least trivalent:
    if it was univalent the sphere in $\Sigma$ that bounds it would not be essential, 
    and if it was bivalent the two spheres in $\Sigma$ that bound it would be parallel.
    (The two spheres at a bivalent vertex also cannot be the same sphere because $M \not\cong S^1 \times S^2$.)
    In summary we have at least three half-edges per non-special vertex and at least one half-edge per special vertex (as $M$ is connected and $\Sigma$ non-empty).
    This yields the inequality $2e \ge 3(v-v_0) + v_0$.
    Rewriting, we get $2e \ge 3v - 2v_0 \ge 3v - 2(n+m)$.
    Using the Euler characterisitc equation $v-e=1-g$ from before we get that
    \[
        v \le 2(n+m+g-1) 
        \qquad\text{and}\qquad
        e \le 2(n+m) + 3(g-1)
    \]
    as claimed.
\end{proof}

Recall that in the proof of \cref{thm:two-prime-pieces} we showed that for $M\cong P_1 \# P_2$ any $\Sigma \in \sepNP(M)$ consists of a single sphere. This is an instance of the above lemma with $n=2$ and $g=m=0$.

\begin{lem}\label{lem:finitely-many-graphs-HM}
    Let $M$ be a 3-manifold with possibly empty boundary and $F \subseteq \partial M$ a union of boundary components. 
    Then the set $\sepNP(M)/\Diff_F(M)$ is finite.
\end{lem}
\begin{proof}
    First, note that $\sepNP(M)/\Diff_F(M)$ is a discrete topological space and so any two isotopic separating systems lie in the same orbit.
    This follows because $\sepNP(M)$ is locally $\Diff_F(M)$-retractile or equivalently, by isotopy extension.

    We now essentially follow the proof of \cite[Lemma 2.1]{HatcherMcCullough1990}.
    Recall $M$ decomposes as
    \[
        M\cong \left(P_1\#\cdots \#P_n \# (S^1\times S^2)^{\# g}\right)\setminus \{\amalg_m \interior{D}^3\}
        \cong P_1\#\cdots \#P_n \# (S^1\times S^2)^{\# g} \# (D^3)^{\# m}.
    \]
    We can hence construct $M$ by attaching $g$ $1$-handles to $S^3$, and taking the connected sum with the $P_i$ and with $m$ copies of $D^3$.
    Let $B = \amalg_{n+m+2g} D^3 \subset S^3$ be the collection of $3$-disks along which we performed these attachments.

    Now suppose we have some $\Sigma \in \sepNP(M)$. 
    Then \cite[Lemma 2.1]{HatcherMcCullough1990} (based on work of Scharlemann~\cite[Appendix A]{Bonahon-Cobordism}) tells us that there is a diffeomorphism $\varphi \in \Diff_F(M)$ (in fact fixing a neighbourhood of the boundary) such that $\varphi(\Sigma) \subset M$ lies entirely in $S^3\setminus \interior{B}$.
    Therefore to show the set $\sepNP(M)/\Diff_F(M)$ is finite, it will suffice to show that up to isotopy there are finitely many sphere systems (without parallel spheres) in $S^3\setminus \interior{B}$.
    This is indeed the case because such isotopy classes are uniquely determined by how they partition the finite set of boundary components. Indeed, since $S^3\setminus \interior{B}$ is simply connected, $\pi_2(S^3\setminus B)\cong H_2(S^3\setminus \interior{B})$, so homotopy classes of embedded spheres agree with homology classes, which in turn correspond to how they partition $\partial B$.  By Laudenbach \cite{Laudenbach73}, homotopic spheres are isotopic. The claim for sphere systems is now straightforward.
\end{proof}

\begin{rem}\label{rem:dual-graphs}
    One can show more concretely that $\sepNP(M)/\Diff_F(M)$ is in bijection with a certain set of labelled dual graphs up to label-preserving graph isomorphism, but we will not need this here.
\end{rem}

We now turn our attention towards the stabilisers of the action of $\Diff_F(M)$ on $\sepNP(M)$.

\begin{cor}\label{prop:finiteness-sphere-system}
    Let $M$ be a compact connected $3$-manifold, $\emptyset \neq F \subseteq \partial M$ a non-empty union of boundary components such that $\partial M \setminus F$ contains only spheres and incompressible tori, and let $\Sigma \in \sep(M)$ be a separating system.
    Then $\BDiff_F(M , \Sigma)$
    is homotopy finite.
\end{cor}
\begin{proof}
    We apply \cref{cor:cut-along-spheres}.
    To check assumption ($\ddagger$), let $N \subseteq M \ca \Sigma$ be a component and $\emptyset \neq F_0 \subseteq \partial N$ a non-empty union of boundary components that contains $F \cap N$.
    Then the spherical closure $\scl{N}$ is irreducible because $\Sigma$ was a separating system. Since $\partial N \setminus F_0 \subseteq M \setminus F$ contains only spheres and incompressible tori, the finiteness of $\BDiff_{F_0}(N)$ follows from \cref{thm:prime-finiteness} and we are done.
\end{proof}

We now have all the pieces required to prove our stronger version of Kontsevich's conjecture. We refer the reader to \cref{fig:flowchartmain} for the overall structure and logical dependencies between ingredients in the proof.

\begin{proof}[Proof of \cref{thm:more general Kontsevich-finiteness-conjecture}]
Since $\partial M\neq \emptyset$, $M\neq S^1\times S^2$.  We know by Theorem~\ref{thm: sepNP contractible} that the fat geometric realisation $\|\sepNP_\bullet(M)\|$ is contractible.
Let $\sepNPsemi_\bullet(M)$ denote the semi-simplicial space where $n$-simplices are strict inclusions of separating systems with no parallel spheres.
The simplicial space $\sepNP_\bullet(M)$ is obtained from $\sepNPsemi_\bullet(M)$ by freely adding degeneracies, 
so by \cite[Lemma 2.6]{EbertRandalWilliams} we have an equivalence
$\|\sepNP_\bullet(M)\| \simeq \|\sepNPsemi_\bullet(M)\|$.
(Note that $\|\sepNPsemi_\bullet(M)\|$ is in fact homeomorphic to the \emph{thin} geometric realisation $|\sepNP_\bullet(M)|$.)
The semi-simplicial space $\sepNPsemi_\bullet(M)$ has the advantage that there is a bound on the dimension of the simplices that can appear:
if $M$ has $m$ spherical boundary components and $\scl{M}$ is the connected sum of $n$ irreducible prime factors and $g$ copies of $(S^1 \times S^2)$, then \cref{lem:bounding-graphs} says that each separating system in $\sepNP(M)$ contains at most $2(n+m) + 3(g-1)$ spheres. 
Therefore any chain of more than $2(n+m)+3(g-1)$ inclusions must include an identity,
and so $\sepNPsemi_i(M) = \emptyset$ for $i > 2(n+m)+3(g-1)$.

$\Diff_F(M)$ acts on $\sepNPsemi_\bullet(M)$ levelwise by postcomposition -- non-parallel spheres are taken to non-parallel spheres under diffeomorphisms of~$M$.
We therefore get an action of $\Diff_F(M)$ on $\|\sepNPsemi_\bullet(M)\|$ and since $\|\sepNPsemi_\bullet(M)\|$ is contractible, the homotopy orbit space of this action is a model for the classifying space:
\[
    \BDiff_F(M) \simeq \|\sepNPsemi_\bullet(M)\|\hq \Diff_F(M) \cong \|\sepNPsemi_\bullet \hq \Diff_F(M) \|.
\]
Here we use that $(-)\hq G$ commutes with fat geometric realisation by \cref{lem:homotopy-orbits-functor}.
Since there are only simplices up to a certain dimension, it will suffice to prove that for all $n$ the space
\[
    \sepNPsemi_n(M)/\!\!/\Diff_F(M)
\]
is homotopy finite. 
To reduce this to the case of $n=0$ recall from \cref{lem:face-map-is-finite-covering} that the last vertex map $\sepNPsemi_n(M) \to \sepNP(M)$ is a finite covering.
Therefore the map from $\sepNPsemi_n(M) \hq \Diff_F(M)$ to $\sepNP(M)\hq \Diff_F(M)$ is also a finite covering and it will suffice to show that the latter is homotopy finite.

By the homotopical orbit stabiliser theorem (\cref{lem:orbit-stabiliser-v2}) there is a decomposition 
    \[
        \sepNP(M) /\!\!/ \Diff_F(M)
        \simeq \coprod_{[x]} \BDiff_F(M,\Sigma)
    \]
    where the coproduct runs over a set of representatives of $\sepNP(M)/\Diff_F(M)$. 
    Each of the terms in the coproduct is homotopy finite by  \cref{prop:finiteness-sphere-system}.
    Moreover, we showed $\sepNP(M)/\Diff_F(M)$ is finite in  \cref{lem:finitely-many-graphs-HM}.
    This completes the proof as we have written $\sepNP(M)/\!\!/ \Diff_F(M)$ as a finite coproduct of homotopy finite spaces.
\end{proof}

\subsection{Finite type}\label{subsection - finite type}
So far it has been very important that we always fix part of the boundary of the $3$-manifold, as otherwise we cannot expect $\BDiff(M)$ to be homotopy finite.
For example, if $M$ is hyperbolic, then $\BDiff(M)$ is a $K(G,1)$ for the finite group $G$ of isometries of $M$.
But while classifying spaces of non-trivial finite groups never have a CW model with finitely many cells, they do always have a CW model that has finitely many cells in each dimension.
Our goal in this section is to show that the same is true for $\BDiff(M)$.

We say that a space $X$ is \emph{of finite type}, if there is a weak equivalence $C \simeq X$ where $C$ is a CW complex with finitely many cells of each dimension.
Finite type spaces are closed under more operations than homotopy finite spaces.
Crucially, we have the following stronger analogue of \cref{lem:fiber-sequence-finite}, which for instance implies that if $E \to B$ is a finite covering and $E$ is of finite type, then so is $B$.

\begin{lem}[{\cite[Proposition 2.5]{DrorDwyerKan81}}]\label{lem:finite-type-fiberseq}
    Suppose we have a homotopy fiber sequence
    \[
        F \longrightarrow E \longrightarrow B
    \]
    where $F$ is of finite type and $E \to B$ is surjective on path components.
    Then $E$ is of finite type if \emph{and only if} $B$ is of finite type.
\end{lem}

The next result follows as a direct consequence of \cref{thm:more general Kontsevich-finiteness-conjecture}.

\begin{cor}\label{cor:finite-type-tori-and-spheres}
    Let $M$ be a compact orientable $3$-manifold such that $\partial M$ is either empty or consists of spheres and incompressible tori.
    Then $\BDiff(M)$ is of finite type.
\begin{proof}
    The locally retractile action of $\Diff(M)$ on $\emb(D^3, \interior{M}) \simeq \Fr(M)$ yields via \cref{lem:orbit-stabiliser-transitive} the homotopy fiber sequence
    \[
        \Fr(M)' \longrightarrow \BDiff_{D^3}(M) \longrightarrow \BDiff(M)
    \]
    where $\Fr(M)' \subset \Fr(M)$ is an orbit of the action.
    Since $\Fr(M)' \subset \Fr(M)$ is a union of components it is homotopy finite.
    The total space $\BDiff_{D^3}(M)$ is homotopy finite by \cref{thm:more general Kontsevich-finiteness-conjecture}.
    It follows by \cref{lem:finite-type-fiberseq} that $\BDiff(M)$ is of finite type.
\end{proof}
\end{cor}

In order to generalise this to arbitrary boundary we use a result of McCullough \cite{McCullough1991}, who studied the finiteness properties of mapping class groups of irreducible $3$-manifolds.
\begin{thm}\label{thm:finite-type-irreducible}
    If $M$ is irreducible and $\partial M \neq \emptyset$, then $\BDiff(M)$ is of finite type.
\end{thm}
\begin{proof}
    By work of Hatcher and Ivanov (see \cref{cor:Diff0-list}) we know that $\Diff_0(M) \simeq (S^1)^{\times b}$ for $b\in \{0,1,2\}$ and therefore $\BDiff_0(M) \simeq K(\bbZ^b, 2)$ is of finite type.
    Applying \cref{lem:finite-type-fiberseq} to the homotopy fiber sequence
    \[
        \BDiff_0(M) \longrightarrow \BDiff(M) \longrightarrow B(\pi_0\Diff(M))
    \]
    we see that it suffices to check that $B(\pi_0\Diff(M))$ is of finite type.
    By \cite[Theorem 7.1]{McCullough1991} the group $\pi_0\mathrm{Homeo}(M) \cong \pi_0\Diff(M)$ is finitely presented and of type VFL.
    As pointed out in \cite[Theorem 1.1]{McCullough1991} this means that there is a finite index subgroup $\Gamma \le \pi_0\Diff(M)$ such that $B\Gamma$ is homotopy finite 
    and therefore (e.g.~by \cref{lem:finite-type-fiberseq}) $B(\pi_0\Diff(M))$ is of finite type.
\end{proof}

We can now assemble this in the same way that we showed the homotopy finiteness of $\BDiff_F(M)$ in \cref{thm:more general Kontsevich-finiteness-conjecture}.
\begin{thm}\label{thm:finite-type}
    Let $M$ be a compact, orientable $3$-manifold.
    Then $\BDiff(M)$ is of finite type.
\end{thm}
\begin{proof}
    If $M$ is disconnected, then we consider the finite index subgroup of those diffeomorphisms that do not permute components.
    This subgroup is a product of diffeomorphism groups of connected manifolds, so it will suffice to consider the case that $M$ is connected.
    Moreover, we may assume $M \not\cong S^1 \times S^2$, as this case is covered by \cref{cor:finite-type-tori-and-spheres}.
    
    Since finite type spaces are also closed under homotopy pushouts and passing to finite covers, we can proceed as in the proof of \cref{thm:more general Kontsevich-finiteness-conjecture} to reduce to the claim that for every separating system $\Sigma \in \sepNP(M)$ the space $\BDiff(M, \Sigma)$ is of finite type.
    Here, instead of using \cref{cor:cut-along-spheres} to cut along spheres, we use the homotopy fiber sequence
    \[
        \Diff'(\Sigma) \longrightarrow
        \BDiff_\Sigma(M) \longrightarrow 
        \BDiff(M, \Sigma)
    \]
    where $\Diff'(\Sigma) \subset \Diff(\Sigma)$ is the subgroup of those diffeomorphism of $\Sigma$ that extend to a diffeomorphism of $M$.
    The space $\Diff'(\Sigma)$ is homotopy finite because $\Diff(S^2) \simeq \Or(3)$ is, so by \cref{lem:finite-type-fiberseq} it will suffice to check that that $\BDiff_\Sigma(M)$ is of finite type.
    This space is equivalent to
    \[
        \BDiff_\Sigma(M) \simeq \BDiff_{2\Sigma}(M \ca \Sigma)
        \simeq \prod_{K \subseteq M \,\ca \,\Sigma} \BDiff_{K \cap 2\Sigma}(K).
    \]
    By \cref{lem:finiteness-fixing-balls} $\BDiff_{K \cap 2\Sigma}(K)\simeq \BDiff(\scl{K})$, and by \cref{thm:finite-type-irreducible}, $\BDiff(\scl{K})$ is of finite type for each $K\subseteq M\ca \Sigma$, so we are done.
\end{proof}

\subsection{Strengthening of results and sharpness}
\label{subsection - sharpness}
The restrictions on $F$ in \cref{thm:more general Kontsevich-finiteness-conjecture} do not come from the separating system machinery of \cref{section: systems of spheres} but rather from the base cases of the induction, namely \cref{thm:HM-hereditary} and \cref{thm:sect4main}. Thus, any strengthening of \cref{thm:more general Kontsevich-finiteness-conjecture} will arise from analogous statements when $M$ is irreducible, and the restrictions on $F$ and $\partial M$ are relaxed. 

If $M$ is irreducible and has higher genus incompressible boundary, $M$ still has a JSJ decomposition that is unique up to isotopy, but $T_{\JSJ}$ may include incompressible annuli as well as tori, and some of the components of $M\ca T_{\JSJ}$ may be an $I$-bundle over a a compact surface. There is also an analogous geometric decomposition of $M$ into pieces with finite-volume Riemannian metric with totally geodesic boundary modeled on one of the 8 Thurston geometries. Thus it seems likely that a version of \cref{thm:sect4main} will extend to this setting. If $M$ is irreducible but has some compressible boundary, Bonahon \cite{Bonahon1983} proved the existence and uniqueness for a decomposition of $M$ into a submanifold with incompressible boundary and a so-called \emph{compression body}. The latter can be obtained from $S\times I$ by attaching 2-handles to one boundary component, then taking the spherical closure. 
We believe that, if one knew hereditary finiteness for compression bodies, one could adapt our tools to this setting to prove the following strengthening of Kontsevich's conjecture.

\begin{conj}\label{conj:Strong-Kontsevich} Let $M$ be a compact, connected, orientable 3-manifold and let $F\subset \partial M$ be a non-empty union of boundary components. Then $\BDiff_F(M)$ is homotopy finite.
\end{conj}

In a sense, the above conjecture is the best possible, and (homotopy) finiteness of $\BDiff(M)$ can definitely fail to hold if we do not fix some non-empty union of boundary components.  Indeed, the rational cohomology of $\BDiff(M)$ fails to be bounded for many irreducible $M$. 
We now show that finiteness also fails in a closed, reducible example. 
In fact, in this example we will also see that the rational cohomology of $\BDiff(M)$ need not be bounded.
(In contrast, if $M$ is irreducible, Haken, and not Seifert-fibered,  then by \cref{cor:Diff0-list} and \cite{McCullough1991}, $\Diff(M) \simeq \pi_0\Diff(M)$ has a finite virtual cohomological dimension and hence $\BDiff(M)$ has bounded rational cohomology.)

Concretely, consider $U_g := (S^1 \times S^2)^{\#g}$, and fix a point $x_0 \in U_g$. We know that $\BDiff(U_g)$ has unbounded cohomology for $g=0,1$ by work of Hatcher \cite{Hatcher, Hatcher1981} and for $g=2$ by the computation in our upcoming work, discussed in \cref{subsection - upcoming work}.
We expect this to be true for all $g\ge0$. On the other hand, if we fix a disk in $U_g$,  \cref{thm:Kontsevich-finiteness-conjecture} shows that $\BDiff_{D^3}(U_g)$ is homotopy finite. However, the following example shows that finiteness can fail even if we fix a subset of $U_g$ of codimension greater than one. We will construct a circle action on $U_g$ that fixes $x_0$ (and a circle containing it) and use it to show that $\BDiff(U_g , x_0)$ has unbounded rational cohomology for all $g \ge 0$. Since fixing a disk is equivalent to fixing a point and a frame, the next result may be interpreted as saying we lose finiteness if we fix a point, or even a point and a tangent vector.  

\begin{figure}[ht!]
    \centering
    \resizebox{0.5\linewidth}{!}{
\begingroup%
  \makeatletter%
  \providecommand\color[2][]{%
    \errmessage{(Inkscape) Color is used for the text in Inkscape, but the package 'color.sty' is not loaded}%
    \renewcommand\color[2][]{}%
  }%
  \providecommand\transparent[1]{%
    \errmessage{(Inkscape) Transparency is used (non-zero) for the text in Inkscape, but the package 'transparent.sty' is not loaded}%
    \renewcommand\transparent[1]{}%
  }%
  \providecommand\rotatebox[2]{#2}%
  \newcommand*\fsize{\dimexpr\f@size pt\relax}%
  \newcommand*\lineheight[1]{\fontsize{\fsize}{#1\fsize}\selectfont}%
  \ifx\svgwidth\undefined%
    \setlength{\unitlength}{314.88485982bp}%
    \ifx\svgscale\undefined%
      \relax%
    \else%
      \setlength{\unitlength}{\unitlength * \real{\svgscale}}%
    \fi%
  \else%
    \setlength{\unitlength}{\svgwidth}%
  \fi%
  \global\let\svgwidth\undefined%
  \global\let\svgscale\undefined%
  \makeatother%
  \begin{picture}(1,0.4914523)%
    \lineheight{1}%
    \setlength\tabcolsep{0pt}%
    \put(0,0){\includegraphics[width=\unitlength,page=1]{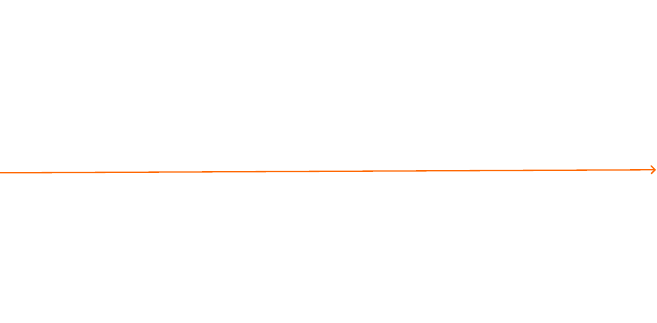}}%
    \put(0.88257567,0.27057097){\color[rgb]{0,0,0}\makebox(0,0)[lt]{\lineheight{1.25}\smash{\begin{tabular}[t]{l}$\operatorname{SO}(2)$\end{tabular}}}}%
    \put(0,0){\includegraphics[width=\unitlength,page=2]{double.pdf}}%
  \end{picture}%
\endgroup%
}
    \caption{The construction of the $\SO(2)$-action on $U_g$ for $g=3$.}
    \label{fig:rotation}
\end{figure}

\begin{prop}\label{prop:circle-action-on-Ug}
    There is a curve $S^1 \subset U_g$ and maps
    \[
        B\!\SO(2) \longrightarrow \BDiff_{S^1}(U_g) \longrightarrow \BDiff_*(U_g) \longrightarrow B\!\SO(3)
    \]
    such that the composite is homotopic to the canonical inclusion.
    As a consequence, neither $\BDiff_{S^1}(U_g)$ nor $\BDiff_*(U_g)$ is homotopy finite (even virtually), as there exists an $\alpha \in H^4(-;\bbQ)$ such that $\alpha^n \neq 0$ for all $n\ge1$.
\end{prop}
\begin{proof}
    Let $M = D^3 \setminus (\amalg_g \interior{D}^3)$ be obtained from the $3$-disk by removing $g$ smaller $3$-disks on the first coordinate axis.
    Then there is a smooth $\SO(2)$-action on $M$ by rotating around that axis as indicated in \cref{fig:rotation}.
    The double $M \cup_{\partial M} M$ of $M$ is diffeomorphic to $U_g$, which therefore also has a smooth $\SO(2)$-action, or equivalently a map $\SO(2) \to \Diff(U_g)$.
    The fixed points of this action are $g+1$ circles in $U_g$; we let $S^1 \subset U_g$ be one of them, and pick $x_0 \in S^1$.
    Then we have group homomorphisms
    \[
        \SO(2) \to \Diff_{S^1}(U_g) \to \Diff_{x_0}(U_g) \to \mathrm{GL}_3^+
    \]
    where the last map sends $\varphi \in \Diff_{x_0}(U_g)$ to its derivative $d_{x_0}\varphi$.
    The composite map is the standard inclusion $\SO(2) \to \mathrm{GL}_3^+$ and since the latter is equivalent to $\SO(3)$, the first claim follows after passing to classifying spaces.
    The second claim then follows since the map $B\!\SO(2) \to B\!\SO(3)$ induces the ring homomorphism on cohomology
    $\bbQ[p_1] \longrightarrow \bbQ[c_1]$ 
    mapping the first Pontrjagin class $p_1$ to the square $c_1^2$ of the first Chern class. The required classes $\alpha$ are therefore the non-trivial image of $\bbQ[p_1]$ under the induced maps on cohomology. 
\end{proof}

\bibliography{mybib}{}

\newcommand{\etalchar}[1]{$^{#1}$}
\begin{thebibliography}{CMRW17}

\bibitem[Asa78]{Asano1978}
Kouhei Asano.
\newblock Homeomorphisms of prism manifolds.
\newblock {\em Yokohama Math. J.}, 26(1):19--25, 1978.

\bibitem[BB22]{BoydBregman22}
Rachael Boyd and Corey Bregman.
\newblock Embedding spaces of split links.
\newblock {\em arXiv preprint 2207.00619}, 2022.

\bibitem[BB23]{bonatto2023}
Luciana Basualdo~Bonatto.
\newblock Decoupling decorations on moduli spaces of manifolds.
\newblock {\em Mathematical Proceedings of the Cambridge Philosophical
  Society}, 174(1):163–198, 2023.

\bibitem[BBM{\etalchar{+}}10]{BessieresEtAl}
Laurent Bessi\`eres, G\'{e}rard Besson, Sylvain Maillot, Michel Boileau, and
  Joan Porti.
\newblock {\em Geometrisation of 3-manifolds}, volume~13 of {\em EMS Tracts in
  Mathematics}.
\newblock European Mathematical Society (EMS), Z\"{u}rich, 2010.

\bibitem[BF81]{BinzFischer78}
Ernst Binz and H.~R. Fischer.
\newblock The manifold of embeddings of a closed manifold.
\newblock In {\em Differential geometric methods in mathematical physics
  ({P}roc. {I}nternat. {C}onf., {T}ech. {U}niv. {C}lausthal,
  {C}lausthal-{Z}ellerfeld, 1978)}, volume 139 of {\em Lecture Notes in Phys},
  pages pp 310--329. Springer, Berlin-New York, 1981.
\newblock With an appendix by P. Michor.

\bibitem[BG21]{BudneyGabai2021}
Ryan Budney and David Gabai.
\newblock Knotted balls in {$S^4$}.
\newblock Preprint https://arxiv.org/abs/1912.09029, 2021.

\bibitem[Bin54]{Bing54}
R.~H. Bing.
\newblock Locally tame sets are tame.
\newblock {\em Ann. of Math. (2)}, 59:145--158, 1954.

\bibitem[Bin59]{Bing59}
R.~H. Bing.
\newblock An alternative proof that {$3$}-manifolds can be triangulated.
\newblock {\em Ann. of Math. (2)}, 69:37--65, 1959.

\bibitem[BK19]{BK19}
Richard Bamler and Bruce Kleiner.
\newblock Ricci flow and contractibility of spaces of metrics.
\newblock Preprint:https://arxiv.org/abs/1909.08710, 2019.

\bibitem[BK23a]{BamlerKleiner2023}
Richard~H. Bamler and Bruce Kleiner.
\newblock Diffeomorphism groups of prime 3-manifolds.
\newblock {\em Journal für die reine und angewandte Mathematik (Crelles
  Journal)}, 2023.

\bibitem[BK23b]{BK23}
Richard~H. Bamler and Bruce Kleiner.
\newblock Ricci flow and diffeomorphism groups of 3-manifolds.
\newblock {\em J. Amer. Math. Soc.}, 36(2):563--589, 2023.

\bibitem[BKK23]{BustamanteKrannichKupers}
Mauricio Bustamante, Manuel Krannich, and Alexander Kupers.
\newblock Finiteness properties of automorphism spaces of manifolds with finite
  fundamental group.
\newblock Preprint https://arxiv.org/abs/2103.13468, 2023.

\bibitem[BL74]{BurgheleaLashof}
Dan Burghelea and Richard Lashof.
\newblock The homotopy type of the space of diffeomorphisms. {I}, {II}.
\newblock {\em Trans. Amer. Math. Soc.}, 196:1--36; ibid. {\bf 196} (1974),
  37--50, 1974.

\bibitem[BO86]{BoileauOtal1986}
Michel Boileau and Jean-Pierre Otal.
\newblock Groupe des diff\'{e}otopies de certaines vari\'{e}t\'{e}s de
  {S}eifert.
\newblock {\em C. R. Acad. Sci. Paris S\'{e}r. I Math.}, 303(1):19--22, 1986.

\bibitem[Bon83a]{Bonahon-Cobordism}
Francis Bonahon.
\newblock Cobordism of automorphisms of surfaces.
\newblock {\em Ann. Sci. \'{E}cole Norm. Sup. (4)}, 16(2):237--270, 1983.

\bibitem[Bon83b]{Bonahon1983}
Francis Bonahon.
\newblock Diff\'{e}otopies des espaces lenticulaires.
\newblock {\em Topology}, 22(3):305--314, 1983.

\bibitem[Bon02]{Bonahon-Geometric-Structures}
Francis Bonahon.
\newblock Geometric structures on 3-manifolds.
\newblock In {\em Handbook of geometric topology}, pages 93--164.
  North-Holland, Amsterdam, 2002.

\bibitem[BS22]{equifibered-properads}
Shaul Barkan and Jan Steinebrunner.
\newblock The equifibered approach to $\infty$-properads.
\newblock Preprint:https://arxiv.org/abs/2211.02576, 2022.

\bibitem[Bud07]{Budney}
Ryan Budney.
\newblock Little cubes and long knots.
\newblock {\em Topology}, 46(1):1--27, 2007.

\bibitem[Cai40]{Cairns1940}
Stewart~S. Cairns.
\newblock Homeomorphisms between topological manifolds and analytic manifolds.
\newblock {\em Ann. of Math. (2)}, 41:796--808, 1940.

\bibitem[CdSR79]{CesardeSaRourke}
Eug\'{e}nia C\'{e}sar~de S\'{a} and Colin Rourke.
\newblock The homotopy type of homeomorphisms of {$3$}-manifolds.
\newblock {\em Bull. Amer. Math. Soc. (N.S.)}, 1(1):251--254, 1979.

\bibitem[Cer59]{Cerf59}
Jean Cerf.
\newblock Groupes d'automorphismes et groupes de diff\'{e}omorphismes des
  vari\'{e}t\'{e}s compactes de dimension {$3$}.
\newblock {\em Bull. Soc. Math. France}, 87:319--329, 1959.

\bibitem[Cer61]{Cerf}
Jean Cerf.
\newblock Topologie de certains espaces de plongements.
\newblock {\em Bull. Soc. Math. France}, 89:227--380, 1961.

\bibitem[Cer68]{Cerf-Gamma4}
Jean Cerf.
\newblock {\em Sur les diff\'{e}omorphismes de la sph\`ere de dimension trois
  {$(\Gamma \sb{4}=0)$}}, volume No. 53 of {\em Lecture Notes in Mathematics}.
\newblock Springer-Verlag, Berlin-New York, 1968.

\bibitem[CMRW17]{CMRW17}
Federico Cantero~Mor\'{a}n and Oscar Randal-Williams.
\newblock Homological stability for spaces of embedded surfaces.
\newblock {\em Geom. Topol.}, 21(3):1387--1467, 2017.

\bibitem[Cos04]{costello2004ainfinity}
Kevin Costello.
\newblock The {A-infinity} operad and the moduli space of curves.
\newblock Preprint:https://arxiv.org/abs/math/0402015, 2004.

\bibitem[CV73]{CharlapVasquez}
Leonard~S. Charlap and Alphonse~T. Vasquez.
\newblock Compact flat riemannian manifolds. {III}. {T}he group of affinities.
\newblock {\em Amer. J. Math.}, 95:471--494, 1973.

\bibitem[CZ06]{CaoZhu}
Huai-Dong Cao and Xi-Ping Zhu.
\newblock A complete proof of the {P}oincar\'{e} and geometrization
  conjectures---application of the {H}amilton-{P}erelman theory of the {R}icci
  flow.
\newblock {\em Asian J. Math.}, 10(2):165--492, 2006.

\bibitem[DDK81]{DrorDwyerKan81}
Emmanuel Dror, William~G. Dwyer, and Daniel~M. Kan.
\newblock Self homotopy equivalences of virtually nilpotent spaces.
\newblock {\em Comment. Math. Helv.}, 56:599--614, 1981.

\bibitem[EE69]{EarleElls69}
Clifford~J. Earle and James Eells.
\newblock A fibre bundle description of {T}eichm\"{u}ller theory.
\newblock {\em J. Differential Geometry}, 3:19--43, 1969.

\bibitem[ERW19]{EbertRandalWilliams}
Johannes Ebert and Oscar Randal-Williams.
\newblock Semi-simplicial spaces.
\newblock {\em Algebr. Geom. Topol.}, 19(4):2099--2150, 2019.

\bibitem[ES70]{EarleSchatz}
Clifford~J. Earle and Alfred Schatz.
\newblock Teichm\"{u}ller theory for surfaces with boundary.
\newblock {\em J. Differential Geometry}, 4:169--185, 1970.

\bibitem[Gab01]{Gabai01}
David Gabai.
\newblock The {S}male conjecture for hyperbolic 3-manifolds: {${\rm
  Isom}(M^3)\simeq{\rm Diff}(M^3)$}.
\newblock {\em J. Differential Geom.}, 58(1):113--149, 2001.

\bibitem[Gia11]{Giansiracusa11}
Jeffrey Giansiracusa.
\newblock The framed little 2-discs operad and diffeomorphisms of handlebodies.
\newblock {\em Journal of Topology}, 4(4):919--941, 2011.

\bibitem[Glu62]{Gluck1962}
Herman Gluck.
\newblock The embedding of two-spheres in the four-sphere.
\newblock {\em Trans. Amer. Math. Soc.}, 104:308--333, 1962.

\bibitem[GMT03]{GMT03}
David Gabai, G.~Robert Meyerhoff, and Nathaniel Thurston.
\newblock Homotopy hyperbolic 3-manifolds are hyperbolic.
\newblock {\em Ann. of Math. (2)}, 157(2):335--431, 2003.

\bibitem[Gra73]{Gramain}
Andr\'{e} Gramain.
\newblock Le type d'homotopie du groupe des diff\'{e}omorphismes d'une surface
  compacte.
\newblock {\em Ann. Sci. \'{E}cole Norm. Sup. (4)}, 6:53--66, 1973.

\bibitem[Gra89]{Grasse1989}
Patricia J.~M. Grasse.
\newblock Finite presentation of mapping class groups of certain
  three-manifolds.
\newblock {\em Topology Appl.}, 32(3):295--305, 1989.

\bibitem[Gro81]{Gromov1979}
Michael Gromov.
\newblock Hyperbolic manifolds (according to {T}hurston and {J}\o rgensen).
\newblock In {\em Bourbaki {S}eminar, {V}ol. 1979/80}, volume 842 of {\em
  Lecture Notes in Math.}, pages 40--53. Springer, Berlin, 1981.

\bibitem[GRW18]{GalatiusRandalWilliams}
S{\o}ren Galatius and Oscar Randal-Williams.
\newblock Homological stability for moduli spaces of high dimensional
  manifolds. {I}.
\newblock {\em J. Amer. Math. Soc.}, 31(1):215--264, 2018.

\bibitem[GRW20]{GRW-users-guide}
S{\o}ren Galatius and Oscar Randal-Williams.
\newblock Moduli spaces of manifolds: a user's guide.
\newblock In {\em Handbook of homotopy theory}, pages 443--485. Boca Raton, FL:
  CRC Press, 2020.

\bibitem[Hat]{HatcherReducible}
Allen Hatcher.
\newblock Diffeomorphism groups of reducible $3$-manifolds.
\newblock Available on the author's homepage:
  \url{https://pi.math.cornell.edu/~hatcher/Papers/DR3M.pdf}.

\bibitem[Hat76]{Hatcher76}
Allen Hatcher.
\newblock Homeomorphisms of sufficiently large {$P^{2}$}-irreducible
  {$3$}-manifolds.
\newblock {\em Topology}, 15(4):343--347, 1976.

\bibitem[Hat81]{Hatcher1981}
Allen Hatcher.
\newblock On the diffeomorphism group of {$S\sp{1}\times S\sp{2}$}.
\newblock {\em Proc. Amer. Math. Soc.}, 83(2):427--430, 1981.

\bibitem[Hat83]{Hatcher}
Allen Hatcher.
\newblock A proof of the {S}male conjecture, {${\rm Diff}(S^{3})\simeq {\rm
  O}(4)$}.
\newblock {\em Ann. of Math. (2)}, 117(3):553--607, 1983.

\bibitem[Hat95]{Hatcher95}
Allen Hatcher.
\newblock Homological stability for automorphism groups of free groups.
\newblock {\em Comment. Math. Helv.}, 70(1):39--62, 1995.

\bibitem[Hat99]{Hatcherincompressible}
Allen Hatcher.
\newblock Spaces of incompressible surfaces.
\newblock Preprint: https://arxiv.org/abs/math/9906074, 1999.

\bibitem[Hat03]{Hatcher1981revised}
Allen Hatcher.
\newblock On the diffeomorphism group of {$S^1\times S^2$}.
\newblock \url{https://pi.math.cornell.edu/~hatcher/Papers/newDiffS1xS2.pdf},
  2003.

\bibitem[Hem04]{Hempel}
John Hempel.
\newblock {\em 3-manifolds}.
\newblock AMS Chelsea Publishing, Providence, RI, 2004.
\newblock Reprint of the 1976 original.

\bibitem[Hir76]{Hirsch}
Morris~W. Hirsch.
\newblock {\em Differential topology}.
\newblock Graduate Texts in Mathematics, No. 33. Springer-Verlag, New
  York-Heidelberg, 1976.

\bibitem[HKMR12]{HKMR12}
Sungbok Hong, John Kalliongis, Darryl McCullough, and J.~Hyam Rubinstein.
\newblock {\em Diffeomorphisms of elliptic 3-manifolds}, volume 2055 of {\em
  Lecture Notes in Mathematics}.
\newblock Springer, Heidelberg, 2012.

\bibitem[HL84]{HendriksLaudenbach1984}
Harrie Hendriks and Fran\c{c}ois Laudenbach.
\newblock Diff\'{e}omorphismes des sommes connexes en dimension trois.
\newblock {\em Topology}, 23(4):423--443, 1984.

\bibitem[HM87]{HendriksMcCullough}
Harrie Hendriks and Darryl McCullough.
\newblock On the diffeomorphism group of a reducible {$3$}-manifold.
\newblock {\em Topology Appl.}, 26(1):25--31, 1987.

\bibitem[HM90]{HatcherMcCullough1990}
Allen Hatcher and Darryl McCullough.
\newblock Finite presentation of {$3$}-manifold mapping class groups.
\newblock In {\em Groups of self-equivalences and related topics ({M}ontreal,
  {PQ}, 1988)}, volume 1425 of {\em Lecture Notes in Math.}, pages 48--57.
  Springer, Berlin, 1990.

\bibitem[HM97]{HatcherMcCullough}
Allen Hatcher and Darryl McCullough.
\newblock Finiteness of classifying spaces of relative diffeomorphism groups of
  {$3$}-manifolds.
\newblock {\em Geom. Topol.}, 1:91--109, 1997.

\bibitem[HW35]{HantzscheWendt}
Walter Hantzsche and Hilmar Wendt.
\newblock Dreidimensionale euklidische {R}aumformen.
\newblock {\em Math. Ann.}, 110(1):593--611, 1935.

\bibitem[HW73]{HatcherWagoner}
Allen Hatcher and John Wagoner.
\newblock {\em Pseudo-isotopies of compact manifolds}, volume No. 6 of {\em
  Ast\'{e}risque}.
\newblock Soci\'{e}t\'{e} Math\'{e}matique de France, Paris, 1973.
\newblock With English and French prefaces.

\bibitem[Iva76]{Ivanov76}
Nikolai~V. Ivanov.
\newblock Groups of diffeomorphisms of {W}aldhausen manifolds.
\newblock {\em Zap. Nau\v{c}n. Sem. Leningrad. Otdel. Mat. Inst. Steklov.
  (LOMI)}, 66:172--176, 209, 1976.
\newblock Studies in topology, II.

\bibitem[Iva79]{Ivanov79}
Nikolai~V. Ivanov.
\newblock Homotopies of automorphism spaces of some three-dimensional
  manifolds.
\newblock {\em Dokl. Akad. Nauk SSSR}, 244(2):274--277, 1979.

\bibitem[Iva82]{Ivanov82}
Nikolai~V. Ivanov.
\newblock Homotopy of spaces of diffeomorphisms of some three-dimensional
  manifolds.
\newblock {\em Zap. Nauchn. Sem. Leningrad. Otdel. Mat. Inst. Steklov. (LOMI)},
  122:72--103, 164--165, 1982.
\newblock Studies in topology, IV.

\bibitem[Iva89]{Ivanov-Teichmuller89}
Nikolai~V. Ivanov.
\newblock Attaching corners to {T}eichm\"{u}ller space.
\newblock {\em Algebra i Analiz}, 1(5):115--143, 1989.

\bibitem[Jac80]{Jaco1980}
William Jaco.
\newblock {\em Lectures on three-manifold topology}, volume~43 of {\em CBMS
  Regional Conference Series in Mathematics}.
\newblock American Mathematical Society, Providence, RI, 1980.

\bibitem[Joh79a]{Johannson1979b}
Klaus Johannson.
\newblock {\em Homotopy equivalences of {$3$}-manifolds with boundaries},
  volume 761 of {\em Lecture Notes in Mathematics}.
\newblock Springer, Berlin, 1979.

\bibitem[Joh79b]{Johannson1979}
Klaus Johannson.
\newblock On the mapping class group of simple {$3$}-manifolds.
\newblock In {\em Topology of low-dimensional manifolds ({P}roc. {S}econd
  {S}ussex {C}onf., {C}helwood {G}ate, 1977)}, volume 722 of {\em Lecture Notes
  in Math.}, pages 48--66. Springer, Berlin, 1979.

\bibitem[JS78]{JacoShalen1976}
William Jaco and Peter~B. Shalen.
\newblock A new decomposition theorem for irreducible sufficiently-large
  {$3$}-manifolds.
\newblock In {\em Algebraic and geometric topology ({P}roc. {S}ympos. {P}ure
  {M}ath., {S}tanford {U}niv., {S}tanford, {C}alif., 1976), {P}art 2}, volume
  XXXII of {\em Proc. Sympos. Pure Math.}, pages 71--84. Amer. Math. Soc.,
  Providence, RI, 1978.

\bibitem[Kir97]{KirbyProblems}
Rob Kirby.
\newblock Problems in low-dimensional topology. ({Edited} by {Rob} {Kirby}).
\newblock Kazez, {William} {H}. (ed.), {Geometric} topology. 1993 {Georgia}
  international topology conference, {August} 2--13, 1993, {Athens}, {GA},
  {USA}. {Providence}, {RI}: {American} {Mathematical} {Society}. {AMS}/{IP}
  {Stud}. {Adv}. {Math}. 2(pt.2), 35-473 (1997)., 1997.

\bibitem[KM96]{KalliongisMcCullough}
John Kalliongis and Darryl McCullough.
\newblock Isotopies of {$3$}-manifolds.
\newblock {\em Topology Appl.}, 71(3):227--263, 1996.

\bibitem[KN63]{KobayashiNomizu}
Shoshichi Kobayashi and Katsumi Nomizu.
\newblock {\em Foundations of differential geometry. {V}ol {I}}.
\newblock Interscience Publishers (a division of John Wiley \& Sons, Inc.), New
  York-London, 1963.

\bibitem[Kne29]{Kneser1929}
Hellmuth Kneser.
\newblock Geschlossene flächen in dreidimensionalen mannigfaltigkeiten.
\newblock {\em Jahresbericht der Deutschen Mathematiker-Vereinigung},
  38:248--259, 1929.

\bibitem[Kon24]{KontsevichCorrespondence2024}
Maxim Kontsevich.
\newblock Private correspondence, April 2024.

\bibitem[KS77]{KirbySiebenmann}
Robion~C. Kirby and Laurence~C. Siebenmann.
\newblock {\em Foundational essays on topological manifolds, smoothings, and
  triangulations}, volume No. 88 of {\em Annals of Mathematics Studies}.
\newblock Princeton University Press, Princeton, NJ; University of Tokyo Press,
  Tokyo, 1977.
\newblock With notes by John Milnor and Michael Atiyah.

\bibitem[Kup19]{Kupers19}
Alexander Kupers.
\newblock Some finiteness results for groups of automorphisms of manifolds.
\newblock {\em Geom. Topol.}, 23(5):2277--2333, 2019.

\bibitem[Lau73]{Laudenbach73}
Fran\c{c}ois Laudenbach.
\newblock Sur les 2-sph{\`e}res d'une vari{\'e}t{\'e} de dimension 3.
\newblock {\em Ann. Math. (2)}, 97:57--81, 1973.

\bibitem[Lee13]{Lee}
John~M. Lee.
\newblock {\em Introduction to smooth manifolds}, volume 218 of {\em Graduate
  Texts in Mathematics}.
\newblock Springer, New York, second edition, 2013.

\bibitem[Lim64]{Lima64}
Elon~L. Lima.
\newblock On the local triviality of the restriction map for embeddings.
\newblock {\em Comment. Math. Helv.}, 38:163--164, 1964.

\bibitem[McC91]{McCullough1991}
Darryl McCullough.
\newblock Virtually geometrically finite mapping class groups of
  {$3$}-manifolds.
\newblock {\em J. Differential Geom.}, 33(1):1--65, 1991.

\bibitem[Mic80]{Michor80}
Peter~W. Michor.
\newblock {\em Manifolds of differentiable mappings}, volume~3.
\newblock Shiva Publishing Ltd., Nantwich, 1980.

\bibitem[Mil62]{Milnor62}
John Milnor.
\newblock A unique decomposition theorem for {$3$}-manifolds.
\newblock {\em Amer. J. Math.}, 84:1--7, 1962.

\bibitem[MN20]{MannNariman}
Kathryn Mann and Sam Nariman.
\newblock Dynamical and cohomological obstructions to extending group actions.
\newblock {\em Math. Ann.}, 377(3-4):1313--1338, 2020.

\bibitem[Moi52]{Moise1952}
Edwin~E. Moise.
\newblock Affine structures in {$3$}-manifolds. {V}. {T}he triangulation
  theorem and {H}auptvermutung.
\newblock {\em Ann. of Math. (2)}, 56:96--114, 1952.

\bibitem[Moi54]{Moise54}
Edwin~E. Moise.
\newblock Affine structures in {$3$}-manifolds. {VIII}. {I}nvariance of the
  knot-types; local tame imbedding.
\newblock {\em Ann. of Math. (2)}, 59:159--170, 1954.

\bibitem[MS86]{Meeks-Scott}
William~H. Meeks, III and Peter Scott.
\newblock Finite group actions on {$3$}-manifolds.
\newblock {\em Invent. Math.}, 86(2):287--346, 1986.

\bibitem[MS13]{McCulloughSoma13}
Darryl McCullough and Teruhiko Soma.
\newblock The {S}male conjecture for {S}eifert fibered spaces with hyperbolic
  base orbifold.
\newblock {\em J. Differential Geom.}, 93(2):327--353, 2013.

\bibitem[MT14]{MorganTian}
John Morgan and Gang Tian.
\newblock {\em The geometrization conjecture}, volume~5 of {\em Clay
  Mathematics Monographs}.
\newblock American Mathematical Society, Providence, RI; Clay Mathematics
  Institute, Cambridge, MA, 2014.

\bibitem[Mun59]{Munkres-Announcement}
James Munkres.
\newblock Obstructions to the smoothing of piecewise-differentiable
  homeomorphisms.
\newblock {\em Bull. Amer. Math. Soc.}, 65:332--334, 1959.

\bibitem[Mun60a]{Munkres}
James Munkres.
\newblock Differentiable isotopies on the {$2$}-sphere.
\newblock {\em Michigan Math. J.}, 7:193--197, 1960.

\bibitem[Mun60b]{Munkres-Obstructions}
James Munkres.
\newblock Obstructions to the smoothing of piecewise-differentiable
  homeomorphisms.
\newblock {\em Ann. of Math. (2)}, 72:521--554, 1960.

\bibitem[Nar21]{Nariman}
Sam Nariman.
\newblock On the finiteness of the classifying space of diffeomorphisms of
  reducible three manifolds.
\newblock {\em arXiv preprint 2104.12338}, 2021.

\bibitem[Pal60]{Palais60}
Richard~S. Palais.
\newblock Local triviality of the restriction map for embeddings.
\newblock {\em Comment. Math. Helv.}, 34:305--312, 1960.

\bibitem[Per02]{Perelman:2002-1}
Grigori Perelman.
\newblock The entropy formula for the {R}icci flow and its geometric
  applications.
\newblock ar{X}iv:0211159, 2002.

\bibitem[Per03a]{Perelman:2003-2}
Grigori Perelman.
\newblock Finite extinction time for the solutions to the {R}icci flow on
  certain three-manifolds.
\newblock ar{X}iv:0307245, 2003.

\bibitem[Per03b]{Perelman:2003-1}
Grigori Perelman.
\newblock {R}icci flow with surgery on three-manifolds.
\newblock ar{X}iv:0303109, 2003.

\bibitem[RB84]{BirmanRubinstein1984}
Joachim~H. Rubinstein and Joan~S. Birman.
\newblock One-sided {H}eegaard splittings and homeotopy groups of some
  {$3$}-manifolds.
\newblock {\em Proc. London Math. Soc. (3)}, 49(3):517--536, 1984.

\bibitem[Rub78]{Rubinstein1978}
Joachim~H. Rubinstein.
\newblock One-sided {H}eegaard splittings of {$3$}-manifolds.
\newblock {\em Pacific J. Math.}, 76(1):185--200, 1978.

\bibitem[Rub79]{Rubinstein1979}
Joachim~H. Rubinstein.
\newblock On {$3$}-manifolds that have finite fundamental group and contain
  {K}lein bottles.
\newblock {\em Trans. Amer. Math. Soc.}, 251:129--137, 1979.

\bibitem[Sco83]{Scott83}
Peter Scott.
\newblock The geometries of {$3$}-manifolds.
\newblock {\em Bull. London Math. Soc.}, 15(5):401--487, 1983.

\bibitem[Sco85]{Scott85}
Peter Scott.
\newblock Homotopy implies isotopy for some {S}eifert fibre spaces.
\newblock {\em Topology}, 24(3):341--351, 1985.

\bibitem[Sin04]{Sinha04}
Dev~P. Sinha.
\newblock Manifold-theoretic compactifications of configuration spaces.
\newblock {\em Selecta Math. (N.S.)}, 10(3):391--428, 2004.

\bibitem[Sma59]{Smale59}
Stephen Smale.
\newblock Diffeomorphisms of the {$2$}-sphere.
\newblock {\em Proc. Amer. Math. Soc.}, 10:621--626, 1959.

\bibitem[Sou78]{Soule78}
Christophe Soul\'{e}.
\newblock The cohomology of {${\rm SL}\sb{3}({\bf Z})$}.
\newblock {\em Topology}, 17(1):1--22, 1978.

\bibitem[Thu86]{ThurstonAtoroidal}
William~P. Thurston.
\newblock Hyperbolic structures on {$3$}-manifolds. {I}. {D}eformation of
  acylindrical manifolds.
\newblock {\em Ann. of Math. (2)}, 124(2):203--246, 1986.

\bibitem[Wal67]{Waldhausen67}
Friedhelm Waldhausen.
\newblock Eine {K}lasse von {$3$}-dimensionalen {M}annigfaltigkeiten. {I},
  {II}.
\newblock {\em Invent. Math.}, 3:308--333; ibid. { 4 (1967), 87--117}, 1967.

\bibitem[Wal68]{Waldhausen68}
Friedhelm Waldhausen.
\newblock On irreducible {$3$}-manifolds which are sufficiently large.
\newblock {\em Ann. of Math. (2)}, 87:56--88, 1968.

\bibitem[Whi40]{Whitehead}
John Henry~C. Whitehead.
\newblock On {$C^1$}-complexes.
\newblock {\em Ann. of Math. (2)}, 41:809--824, 1940.

\bibitem[Whi61]{Whitehead61}
John Henry~C. Whitehead.
\newblock Manifolds with transverse fields in euclidean space.
\newblock {\em Ann. of Math. (2)}, 73:154--212, 1961.

\bibitem[Zim90]{ZimmermannHW}
Bruno Zimmermann.
\newblock On the {H}antzsche-{W}endt manifold.
\newblock {\em Monatsh. Math.}, 110(3-4):321--327, 1990.

\bibitem[ZZ79]{ZieschangZimmermann}
Heiner Zieschang and Bruno Zimmermann.
\newblock Endliche {G}ruppen von {A}bbildungsklassen gefaserter
  {$3$}-{M}annigfaltigkeiten.
\newblock {\em Math. Ann.}, 240(1):41--62, 1979.

\end{thebibliography}
\bibliographystyle{alpha}

\end{document}